\documentclass[11pt]{article}
\usepackage{amsmath, amsfonts, amssymb, upgreek, amsthm}
\usepackage{graphicx}
\usepackage{caption, subcaption}

\newtheorem{thm}{Theorem}
\newtheorem{prop}[thm]{Proposition}
\newtheorem{lem}[thm]{Lemma}

\newtheorem{cor}[thm]{Corollary}
\newtheorem{rem}[thm]{Remark}

\newtheorem{ex}[thm]{Example}

\renewcommand{\epsilon}{\varepsilon}
\renewcommand{\phi}{\varphi}

\renewcommand{\deg}{\operatorname{deg}}
\renewcommand{\P}{\operatorname{P}}

\newcommand{\SP}{\operatorname{S}}
\newcommand{\cone}{\operatorname{Cone}}

\newcommand{\BB}{\mathbb}

\newcommand{\separate}{\vskip5pt}

\newcommand{\im}{\operatorname{Im}}
\newcommand{\re}{\operatorname{Re}}
\newcommand{\tr}{\operatorname{Tr}}

\newcommand{\B}{\overline}
\newcommand{\HC}{\BB H_{\BB C}}

\newcommand{\degt}{\widetilde{\operatorname{deg}}}

\newcommand{\Span}{\operatorname{Span}}
\newcommand{\M}{\operatorname{Mx}}

\newcommand{\m}{\upmu}
\newcommand{\tlap}{\widetilde{\square}_{\m}}

\usepackage[OT2,T1]{fontenc}
\newcommand\textcyr[1]{{\fontencoding{OT2}\fontfamily{wncyr}\selectfont #1}}
\newcommand{\Zh}{\textit{\textcyr{Zh}}}

\begin{document}

\title{\bf Anti De Sitter Deformation of Quaternionic Analysis
and the Second Order Pole}
\author{Igor Frenkel and Matvei Libine}
\maketitle

\begin{abstract}
This is a continuation of a series of papers \cite{FL1, FL2, FL3},
where we develop quaternionic analysis from the point of view of
representation theory of the conformal Lie group and its Lie algebra.
In this paper we continue to study the quaternionic analogues of
Cauchy's formula for the second order pole.
These quaternionic analogues are closely related to regularization
of infinities of vacuum polarization diagrams in
four-dimensional quantum field theory.
In order to add some flexibility, especially when dealing with
Cauchy's formula for the second order pole, we introduce a
one-parameter deformation of quaternionic analysis.
This deformation of quaternions preserves conformal invariance
and has a geometric realization as anti de Sitter space sitting
inside the five-dimensional Euclidean space.
We show that many results of quaternionic analysis -- including
the Cauchy-Fueter formula -- admit a simple and canonical deformation.
We conclude this paper with a deformation of the quaternionic analogues of
Cauchy's formula for the second order pole.
\end{abstract}

\section{Introduction}

Let $\BB H$ denote the algebra of quaternions
$$
\BB H = 1\BB R \oplus i \BB R \oplus j \BB R \oplus k \BB R
$$
with the norm
$$
N(X)=(x^0)^2+(x^1)^2+(x^2)^2+(x^3)^2,
\qquad X = x^0 + ix^1 + jx^2 + kx^3 \in \BB H.
$$
Since the early days of quaternionic analysis, when the quaternionic
analogue of complex holomorphic functions was introduced, there was a
fundamental question about the natural quaternionic analogue of the ring
structure of holomorphic
functions\footnote{Some readers may point out the Cauchy-Kovalevskaya product
of quaternionic regular functions. But this operation is not satisfactory,
since it does not have good invariance properties with respect to the
conformal group action.}.
In particular, one can ask what is the
quaternionic version of the ring of polynomials $\BB C[z]$ and Laurent
polynomials $\BB C[z,z^{-1}]$. The representation theoretic approach that
we have developed in \cite{FL1,FL2,FL3} suggests the most naive candidates
for an answer: $\BB H$-valued polynomial functions on $\BB H$ and
$\BB H^{\times} = \{X \in \BB H ;\: X \ne 0 \}$ respectively:
\begin{equation}  \label{H-rings}
\BB H[x^0,x^1,x^2,x^3] \qquad \text{and} \qquad
\BB H[x^0,x^1,x^2,x^3,N(X)^{-1}].
\end{equation}
Another option is just to consider $\BB R$-valued
polynomial functions on $\BB H$ and $\BB H^{\times}$:
\begin{equation}  \label{R-rings}
\BB R[x^0,x^1,x^2,x^3] \qquad \text{and} \qquad
\BB R[x^0,x^1,x^2,x^3,N(X)^{-1}].
\end{equation}
Clearly, all four of these spaces of functions have natural ring
structures. However, these functions are typically neither regular nor
harmonic, so all the regular quaternionic structure is lost.
Representation theory asserts that all four quaternionic rings yield the
so-called middle series of representations of the conformal Lie algebra
$\mathfrak{sl}(2,\BB H) \simeq \mathfrak{so}(5,1)$
and that there are intertwining maps from tensor products of left regular
and right regular polynomials into the two rings (\ref{H-rings}) and from the
tensor products of harmonic polynomials into the two rings (\ref{R-rings}).
It is also natural to complexify the quaternions
$$
\HC = \BB H \otimes_{\BB R} \BB C
= 1\BB C \oplus i\BB C \oplus j\BB C \oplus k\BB C
$$
and the spaces of polynomial functions on them (\ref{H-rings}) and
(\ref{R-rings}). This brings us to the main objects of our study:
\begin{align}
{\cal W}^+ &= \HC[z^0,z^1,z^2,z^3],  \label{1} \\
{\cal W} &= \HC[z^0,z^1,z^2,z^3,N(Z)^{-1}], \label{2} \\
\Zh^+ &= \BB C [z^0,z^1,z^2,z^3], \label{3} \\
\Zh &= \BB C [z^0,z^1,z^2,z^3,N(Z)^{-1}]  \label{4}
\end{align}
-- the two versions of the rings of ordinary and Laurent polynomials in one
complex variable.

The relation between the quaternionic ring (\ref{3}) and harmonic functions
yields a reproducing integral formula for functions in $\Zh^+$.
On the other hand, the relation between the ring (\ref{1}) and the
regular functions is similar, but instead of a reproducing formula
we get an integral expression for a certain second order differential operator
$$
\M: {\cal W}^+ \to {\cal W}^+, \qquad \M f = \nabla f \nabla - \square f^+.
$$
(The operator $\M$ is directly related to the solutions of the Maxwell
equations for the gauge potential.)
All of these formulas can be regarded as quaternionic analogues of
the Cauchy's formula for the second order pole
$$
f'(z_0) = \frac 1{2\pi i} \oint \frac {f(z)\,dz}{(z-z_0)^2}
$$
(see \cite{FL1} for details).
The corresponding formulas for ${\cal W}$ and $\Zh$ are more involved and
are the main subject of this paper. They require certain regularizations
of infinities which are well known in four-dimensional quantum field theory
as ``vacuum polarizations'' in spinor and scalar cases respectively.
They are usually encoded by the Feynman diagrams shown in
Figure \ref{vpolar-diag} and play a key role in renormalization theory
(see, for example, \cite{Sm}).

\begin{figure}
\begin{center}
\begin{subfigure}[b]{0.25\textwidth}
\centering
\includegraphics[scale=1]{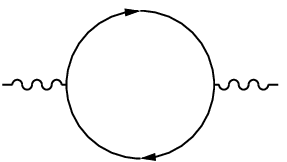}
\caption{Spinor case}
\end{subfigure}
\qquad
\begin{subfigure}[b]{0.25\textwidth}
\centering
\includegraphics[scale=1]{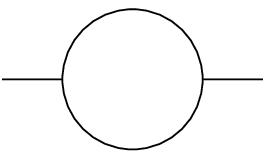}
\caption{Scalar case}
\end{subfigure}
\end{center}
\caption{Vacuum polarization diagrams}
\label{vpolar-diag}
\end{figure}

In order to explain our reproducing formula in the scalar case,
we recall the space ${\cal H}$ of harmonic functions on $\BB H^{\times}$.
It decomposes into two irreducible components:
\begin{equation}  \label{H-decomp-intro}
{\cal H} = {\cal H}^- \oplus {\cal H}^+
\end{equation}
with respect to the action of the conformal algebra $\mathfrak{sl}(2,\BB H)$.
Similarly, we can decompose $\Zh$ into three irreducible components:
\begin{equation}  \label{Zh-decomp-intro}
\Zh = \Zh^- \oplus \Zh^0 \oplus \Zh^+.
\end{equation}
The spaces $\Zh^+$ and $\Zh^-$ have already appeared in \cite{FL1},
but the appearance of $\Zh^0$ in quaternionic analysis is new.
We study equivariant embeddings of $\Zh^-$, $\Zh^0$ and $\Zh^+$ into
tensor products ${\cal H}^{\pm} \otimes {\cal H}^{\pm}$.
The cases $\Zh^-$ and $\Zh^+$ are fairly straightforward, but the case of
$\Zh^0$ -- which is the core of the scalar vacuum polarization -- is more
subtle.
As a consequence of these equivariant embeddings, we obtain projectors of
$\Zh$ onto its irreducible components.
Using these projectors we get a reproducing formula for all functions in $\Zh$,
which may be loosely stated as follows. Let
$$
(I_1f)(Z_1,Z_2) = \frac i{2\pi^3}
\int_{W \in U(2)} \frac{f(W) \,dV}{N(W-Z_1) \cdot N(W-Z_2)},
\qquad f \in \Zh, \: Z_1,Z_2 \in \BB D^+ \sqcup \BB D^-,
$$
where $\BB D^+$ and $\BB D^-$ are two certain open regions in $\HC$ both having
$U(2)$ as their Shilov boundary. Then
\begin{multline}  \label{reproducing-formula}
f(Z) =
 \lim_{\genfrac{}{}{0pt}{}{Z_1, Z_2 \to Z}{Z_1 \in \BB D^+,\: Z_2 \in \BB D^+}} (I_1f)(Z_1,Z_2)
- \lim_{\genfrac{}{}{0pt}{}{Z_1, Z_2 \to Z}{Z_1 \in \BB D^+,\: Z_2 \in \BB D^-}} (I_1f)(Z_1,Z_2) \\
- \lim_{\genfrac{}{}{0pt}{}{Z_1, Z_2 \to Z}{Z_1 \in \BB D^-,\: Z_2 \in \BB D^+}} (I_1f)(Z_1,Z_2)
+ \lim_{\genfrac{}{}{0pt}{}{Z_1, Z_2 \to Z}{Z_1 \in \BB D^-,\: Z_2 \in \BB D^-}} (I_1f)(Z_1,Z_2),
\qquad f \in \Zh, \: Z \in U(2).
\end{multline}
(see Remark \ref{reproducing-remark}).
A similar formula can be deduced for the operator $\M$ acting on ${\cal W}$.

The treatment of the projector onto $\Zh^0$ and the resulting reproducing
formula are not completely satisfactory, since the points $Z_1$ and $Z_2$
belong to the non-intersecting domains $\BB D^+$ and $\BB D^-$.
This phenomenon is well known in physics, where it results in the divergence
of the Feynman integral corresponding to the scalar vacuum polarization diagram.
Physicists have several methods to achieve this isolation of singularity
involving introduction of an auxiliary parameter.
Depending on the method, this auxiliary parameter can be interpreted as
dimension or mass.
The former method is incompatible with representation theoretic approach and
the latter is better from our point of view, but still breaks the conformal
symmetry down to the famous Poincare group.
There is, however, a third way to introduce an auxiliary parameter while
fully preserving the conformal invariance -- namely via anti de Sitter
deformation of the flat Minkowski space.
This is the method we pursue in the second part of the paper to
develop a deformation of quaternionic analysis.
First of all, we define a one-parameter family of conformal Laplacians
\begin{equation}  \label{tlap-intro}
\tlap = \square + \m^2 \bigl( \degt^2 + \degt \bigr)
\end{equation}
depending on a real parameter $\m$, where $\degt$ denotes the degree operator
plus identity and $\square$ is the ordinary Laplacian on $\BB H$.
As usual, the deformed Laplacian admits a quaternionic factorization into
two first order differential operators
\begin{equation}  \label{tlap-factorization-intro}
\tlap = \overrightarrow{\nabla}_{\m} (\overrightarrow{\nabla}_{\m}-\m)
= \overleftarrow{\nabla}_{\m} (\overleftarrow{\nabla}_{\m}+\m),
\end{equation}
where the arrows indicate that the operator $\overrightarrow{\nabla}_{\m}$
is applied to functions on the left and $\overleftarrow{\nabla}_{\m}$ is
applied on the right. 
This factorization allows us to define a one-parameter family of left and
right regular functions by the requirement
$$
\overrightarrow{\nabla}_{\m}f =0 \qquad \text{and} \qquad
g \overleftarrow{\nabla}_{\m} =0.
$$
Then we prove analogues of Cauchy-Fueter and Poisson formulas
as well as generalize certain other constructions and results from
quaternionic analysis.
In particular, the Poisson kernel $N(X-Y)^{-1}$ is replaced by the following
family of kernels depending on $\m$
\begin{equation} \label{K-intro}
\frac1{\langle \hat X - \hat Y,\hat X - \hat Y \rangle_{1,4}},
\end{equation}
where
$$
\hat X = \bigl( \sqrt{\m^{-2}+N(X)},x^0,x^1,x^2,x^3 \bigr), \quad
\hat Y = \bigl( \sqrt{\m^{-2}+N(Y)},y^0,y^1,y^2,y^3 \bigr)
$$
and the 5-dimensional space $\BB R^{1,4}$ is equipped with an indefinite
inner product
$$
\langle W, W' \rangle_{1,4} = w^0w'^0-w^1w'^1-w^2w'^2-w^3w'^3-w^4w'^4,
$$
$W=(w^0,w^1,w^2,w^3,w^4)$, $W'=(w'^0,w'^1,w'^2,w'^3,w'^4) \in \BB R^{1,4}$.
After developing basics of the anti de Sitter deformation of quaternionic
analysis we turn to the treatment of the second order pole.

As the expression for the reproducing kernel (\ref{K-intro}) indicates,
various one-parameter generalizations of results from quaternionic analysis
admit a natural geometric interpretation
when we identify the space of quaternions with a single sheet of a two-sheeted
hyperboloid in $\BB R^{1,4}$.
This hyperboloid is known to physicists as the anti de Sitter space.
Thus the anti de Sitter space-time geometry -- which has been extensively
studied by physicists (see, for example, \cite{BGMT} and references therein)
-- naturally provides a one-parameter deformation of (classical)
quaternionic analysis.
We obtain, in particular, a deformation of the representations
${\cal H}^{\pm}_{\m}$, $\Zh^{\pm}_{\m}$ and $\Zh^0_{\m}$ of the conformal Lie algebra
and find projectors onto these spaces.
This brings us back to our original motivation of the one-parameter deformation
of quaternionic analysis -- finding a representation-theoretic interpretation
of the regularization in quantum field theory.
This question will be addressed in a subsequent work.

The paper consists of two parts related by a common motivation of
development of quaternionic analysis using representation-theoretic methods.
In Sections \ref{preliminaries-section}-\ref{M-section} we study
structures related to the second order pole and in
Sections \ref{AdS-section}-\ref{Zh_mu-section} we develop the one-parameter
deformation of quaternionic analysis using geometry of the anti de Sitter space.
In Section \ref{preliminaries-section} we summarize the results of
quaternionic analysis that are used in this article and, in particular,
introduce the representation $(\rho_1,\Zh)$ of the conformal algebra
$\mathfrak{gl}(2,\HC) \simeq \mathfrak{gl}(4,\BB C)$
which is one of the main subject of this work.
In Section \ref{K-type-section} we give explicit $K$-types of $(\rho_1,\Zh)$,
and in Section \ref{irred-section} we show that the representation
$(\rho_1,\Zh)$ decomposes into three irreducible components
(\ref{Zh-decomp-intro}) (Theorem \ref{Zh-decomposition}).
We also prove that the subspaces $\Zh^-$, $\Zh^0$ and $\Zh^+$ are the images
under the natural multiplication maps of, respectively,
${\cal H}^- \otimes {\cal H}^-$, ${\cal H}^- \otimes {\cal H}^+$ and
${\cal H}^+ \otimes {\cal H}^+$ (Lemma \ref{image-lemma}).
In Section \ref{formal-section} we make a formal calculation of the
reproducing kernel for $\Zh^0$.
(Note that the reproducing kernels for $\Zh^+$ and $\Zh^-$
were computed in \cite{FL1}.)
In Section \ref{embeddings-section} we study conformally invariant
embeddings of the irreducible components $\Zh^{\pm}$ and $\Zh^0$
into tensor products ${\cal H}^\pm \otimes {\cal H}^\pm$,
the case of $\Zh^0$ being more subtle, and, as a consequence of these
embeddings, we obtain projectors of $\Zh$
onto its irreducible components (Theorem \ref{embedding},
Corollary \ref{proj-cor} and Theorem \ref{Zh^0-projector}).
In Section \ref{integral-section} we give a new derivation of the
identification of the one-loop Feynman diagram with the integral kernels of
the projection operators
$$
{\cal P}^+: {\cal H}^+ \otimes {\cal H}^+ \twoheadrightarrow \Zh^+
\hookrightarrow {\cal H}^+ \otimes {\cal H}^+ \qquad \text{and} \qquad
{\cal P}^-: {\cal H}^- \otimes {\cal H}^- \twoheadrightarrow \Zh^-
\hookrightarrow {\cal H}^- \otimes {\cal H}^-
$$
(cf. \cite{FL1}).
In Section \ref{M-section} we realize the spaces $\Zh^{\pm}$, $\Zh^0$
as well as the results of Section \ref{embeddings-section} in the setting
of the Minkowski space $\BB M$.
In the second part of the paper,Section \ref{AdS-section}, we introduce
the anti de Sitter deformation of the space of quaternions $\BB H$
together with the conformal Laplacian (\ref{tlap-intro}).
Then in Section \ref{conformal-section} we describe the action of the
conformal algebra $\mathfrak{so}(1,5)$ on the kernel of the conformal Laplacian.
In Section \ref{extension-section} we find simple extensions of the elements
of the kernel to $\BB R^{1,4}$ as solutions of the wave equation.
In Section \ref{H_mu-section} we introduce a space ${\cal H}_{\m}$ consisting
of the $K$-finite elements of the kernel of the conformal Laplacian.
Similarly to (\ref{H-decomp-intro}), we have a decomposition into irreducible
components ${\cal H}_{\m} = {\cal H}_{\m}^+ \oplus {\cal H}_{\m}^-$.
In Section \ref{Poisson-section} we prove an analogue of the Poisson formula
for the solutions of $\tlap\phi=0$ (Theorem \ref{Poisson5}).
In Section \ref{reg-fun-section} we factor the conformal Laplacian as a product
of two Dirac-type operators (\ref{tlap-factorization-intro}).
In Sections \ref{reg-fun-prop-section} and \ref{Cauchy-Fueter-section}
we proceed to study deformed quaternionic regular functions.
We prove analogues of the Cauchy's Theorem and the Cauchy-Fueter formula in the
deformed setting (Corollary \ref{cauchy} and Theorem \ref{Cauchy-Fueter-5}).
Finally, in Section \ref{Zh_mu-section} we introduce a deformation $\Zh_{\m}$
of the space $\Zh$ associated with the second order pole,
similarly to (\ref{Zh-decomp-intro}), decompose it into a direct sum
$\Zh^-_{\m} \oplus \Zh^0_{\m} \oplus \Zh^+_{\m}$
and obtain projectors onto these direct summands.

The first author was supported by the NSF grant DMS-1001633;
the second author was supported by the NSF grant DMS-0904612.

\section{Preliminaries}  \label{preliminaries-section}

We recall some notations from \cite{FL1}.
Let $\HC$ denote the space of complexified quaternions:
$\HC = \BB H \otimes \BB C$, it can be identified with the algebra of
$2 \times 2$ complex matrices:
$$
\HC = \BB H \otimes \BB C \simeq \biggl\{
Z= \begin{pmatrix} z_{11} & z_{12} \\ z_{21} & z_{22} \end{pmatrix}
; \: z_{ij} \in \BB C \biggr\}
= \biggl\{ Z= \begin{pmatrix} z^0-iz^3 & -iz^1-z^2 \\ -iz^1+z^2 & z^0+iz^3
\end{pmatrix} ; \: z^k \in \BB C \biggr\}.
$$
For $Z \in \HC$, we write
$$
N(Z) = \det \begin{pmatrix} z_{11} & z_{12} \\ z_{21} & z_{22} \end{pmatrix}
= z_{11}z_{22}-z_{12}z_{21} = (z^0)^2 + (z^1)^2 + (z^2)^2 + (z^3)^2
$$
and think of it as (the square of) the norm of $Z$.
We denote by $\HC^{\times}$ the group of invertible complexified quaternions:
$$
\HC^{\times} = \{ Z \in \HC ;\: N(Z) \ne 0 \}.
$$
Clearly, $\HC^{\times} \simeq GL(2,\BB C)$.
We realize $U(2)$ as a subgroup of $\HC^{\times}$:
$$
U(2) = \{ Z \in \HC ;\: Z^*=Z^{-1} \},
$$
where $Z^*$ denotes the complex conjugate transpose of a complex matrix $Z$.
For $R>0$, we set
$$
U(2)_R = \{ RZ ;\: Z \in U(2) \}
$$
and orient it as in \cite{FL1} so that
$\int_{U(2)_R} \frac{dV}{N(Z)^2} = -2\pi^3 i$,
where  $dV$ is a holomorphic 4-form defined by
$$
dV = dz^0 \wedge dz^1 \wedge dz^2 \wedge dz^3
= \frac14 dz_{11} \wedge dz_{12} \wedge dz_{21} \wedge dz_{22}.
$$
Recall that a group $GL(2,\HC) \simeq GL(4,\BB C)$ acts on $\HC$ by fractional
linear (or conformal) transformations:
\begin{equation}  \label{conformal-action}
h: Z \mapsto (aZ+b)(cZ+d)^{-1} = (a'-Zc')^{-1}(-b'+Zd'),
\qquad Z \in \HC,
\end{equation}
where
$h = \bigl(\begin{smallmatrix} a & b \\ c & d \end{smallmatrix}\bigr)
\in GL(2,\HC)$ and 
$h^{-1} = \bigl(\begin{smallmatrix} a' & b' \\ c' & d' \end{smallmatrix}\bigr)$.

For convenience we recall Lemmas 10 and 61 from \cite{FL1}:

\begin{lem}  \label{Z-W}
For $h = \bigl( \begin{smallmatrix} a & b \\ c & d \end{smallmatrix} \bigr)
\in GL(2,\HC)$
with $h^{-1} =
\bigl( \begin{smallmatrix} a' & b' \\ c' & d' \end{smallmatrix} \bigr)$,
let $\tilde Z = (aZ+b)(cZ+d)^{-1}$ and $\tilde W = (aW+b)(cW+d)^{-1}$.
Then
\begin{align*}
(\tilde Z - \tilde W) &= (a'-Wc')^{-1} \cdot(Z-W) \cdot (cZ+d)^{-1}  \\
&= (a'-Zc')^{-1} \cdot(Z-W) \cdot (cW+d)^{-1}.
\end{align*}
\end{lem}

\begin{lem}  \label{Jacobian_lemma}
Let $d\tilde V$ denote the pull-back of $dV$ under the map
$Z \mapsto (aZ+b)(cZ+d)^{-1}$, where
$h = \bigl( \begin{smallmatrix} a & b \\ c & d \end{smallmatrix} \bigr)
\in GL(2,\HC)$ and
$h^{-1}= \bigl( \begin{smallmatrix} a' & b' \\ c' & d' \end{smallmatrix} \bigr)$.
Then
$$
dV = N(cZ+d)^2 \cdot N(a'-Zc')^2 \,d\tilde V.
$$
\end{lem}

We often use the matrix coefficient functions of $SU(2)$ described by
equation (27) of \cite{FL1} (cf. \cite{V}):
\begin{equation}  \label{t}
t^l_{n\,\underline{m}}(Z) = \frac 1{2\pi i}
\oint (sz_{11}+z_{21})^{l-m} (sz_{12}+z_{22})^{l+m} s^{-l+n} \,\frac{ds}s,
\qquad
\begin{matrix} l = 0, \frac12, 1, \frac32, \dots, \\ m,n \in \BB Z +l, \\
 -l \le m,n \le l, \end{matrix}
\end{equation}
$Z=\bigl(\begin{smallmatrix} z_{11} & z_{12} \\
z_{21} & z_{22} \end{smallmatrix}\bigr) \in \HC$,
the integral is taken over a loop in $\BB C$ going once around the origin
in the counterclockwise direction.
We regard these functions as polynomials on $\HC$.
For future use we state the multiplicativity property of matrix coefficients
\begin{equation} \label{t-mult}
t^l_{m \, \underline{n}} (Z_1Z_2) =
\sum_{j=-l}^l t^l_{m \, \underline{j}}(Z_1) \cdot t^l_{j \, \underline{n}}(Z_2).
\end{equation}
It is also useful to recall that
$$
t^l_{m\underline{n}}(Z^{-1}) = t^l_{m\underline{n}}(Z^+) \cdot N(Z)^{-2l}
\quad \text{is proportional to} \quad
t^l_{-n\underline{-m}}(Z) \cdot N(Z)^{-2l}.
$$

As in Section 2 of \cite{FL2}, we consider the space of $\BB C$-valued
functions on $\HC$ (possibly with singularities) which are holomorphic with
respect to the complex variables $z^0$, $z^1$, $z^2$, $z^3$ or
$z_{11}$, $z_{12}$, $z_{21}$, $z_{22}$ and harmonic,
i.e. satisfying $\square \phi=0$, where
$$
\square = \frac{\partial^2}{(\partial z^0)^2}+
\frac{\partial^2}{(\partial z^1)^2} + \frac{\partial^2}{(\partial z^2)^2} +
\frac{\partial^2}{(\partial z^3)^2}
= 4 \left( \frac{\partial^2}{\partial z_{11} \partial z_{22}}
- \frac{\partial^2}{\partial z_{12} \partial z_{21}} \right).
$$
We denote this space by $\widetilde{\cal H}$.
Then the conformal group $GL(2,\HC) \simeq GL(4,\BB C)$ acts on
$\widetilde{\cal H}$ by two slightly different actions:
\begin{align*}
\pi^0_l(h): \: \phi(Z) \quad &\mapsto \quad \bigl( \pi^0_l(h)\phi \bigr)(Z) =
\frac 1{N(cZ+d)} \cdot \phi \bigl( (aZ+b)(cZ+d)^{-1} \bigr),  \\
\pi^0_r(h): \: \phi(Z) \quad &\mapsto \quad \bigl( \pi^0_r(h)\phi \bigr)(Z) =
\frac 1{N(a'-Zc')} \cdot \phi \bigl( (a'-Zc')^{-1}(-b'+Zd') \bigr),
\end{align*}
where
$h = \bigl(\begin{smallmatrix} a & b \\ c & d \end{smallmatrix}\bigr)
\in GL(2,\HC)$ and 
$h^{-1} = \bigl(\begin{smallmatrix} a' & b' \\ c' & d' \end{smallmatrix}\bigr)$.
We have 
$$
(aZ+b)(cZ+d)^{-1} = (a'-Zc')^{-1}(-b'+Zd'), \qquad \forall Z \in \HC,
$$
and these two actions coincide on $SL(2, \HC) \simeq SL(4,\BB C)$,
which is defined as the connected Lie subgroup of $GL(2,\HC)$ with Lie algebra
$$
\mathfrak{sl}(2,\HC) = \{ x \in \mathfrak{gl}(2,\HC) ;\: \re (\tr x) =0 \}
\simeq \mathfrak{sl}(4,\BB C).
$$

We introduce two spaces of harmonic polynomials:
$$
{\cal H}^+ = \widetilde{\cal H} \cap \BB C[z_{11},z_{12},z_{21},z_{22}],
$$
$$
{\cal H} = \widetilde{\cal H} \cap \BB C[z_{11},z_{12},z_{21},z_{22}, N(Z)^{-1}]
$$
and the space of harmonic polynomials regular at infinity:
$$
{\cal H}^- = \bigl\{ \phi \in \widetilde{\cal H};\:
N(Z)^{-1} \cdot \phi(Z^{-1}) \in {\cal H}^+ \bigr\}.
$$
Then
\begin{align*}
{\cal H} &= {\cal H}^- \oplus {\cal H}^+,  \\
{\cal H}^+ &= \Span \bigl\{ t^l_{n\,\underline{m}}(Z) \bigr\}, \\
{\cal H}^- &= \Span \bigl\{ t^l_{n\,\underline{m}}(Z) \cdot N(Z)^{-(2l+1)} \bigr\}.
\end{align*}
In particular, there are no homogeneous harmonic functions in
$\BB C[z_{11},z_{12},z_{21},z_{22},N(Z)^{-1}]$ of degree $-1$.
Differentiating the actions $\pi^0_l$ and $\pi^0_r$, we obtain actions of
$\mathfrak{gl}(2,\HC) \simeq \mathfrak{gl}(4,\BB C)$ which preserve
the spaces ${\cal H}$, ${\cal H}^-$ and ${\cal H}^+$.
By abuse of notation, we denote these Lie algebra actions by
$\pi^0_l$ and $\pi^0_r$ respectively.
They are described in Subsection 3.2 of \cite{FL2}.

By Theorem 28 in \cite{FL1}, for each $R>0$, we have a
bilinear pairing between $(\pi^0_l, {\cal H})$ and $(\pi^0_r, {\cal H})$:
\begin{equation}  \label{H-pairing}
(\phi_1,\phi_2)_R = \frac 1{2\pi^2}
\int_{S^3_R} (\degt \phi_1)(Z) \cdot \phi_2(Z) \,\frac{dS}R,
\qquad \phi_1, \phi_2 \in {\cal H},
\end{equation}
where $S^3_R \subset \BB H$ is the three-dimensional sphere of radius $R$
centered at the origin
$$
S^3_R = \{ X \in \BB H ;\: N(X)=R^2 \},
$$
$dS$ denotes the usual Euclidean volume element on $S^3_R$, and
$\degt$ denotes the degree operator plus identity:
$$
\degt f = f + \deg f = f + z_{11}\frac{\partial f}{\partial z_{11}} +
z_{12}\frac{\partial f}{\partial z_{12}} + z_{21}\frac{\partial f}{\partial z_{21}}
+ z_{22}\frac{\partial f}{\partial z_{22}}.
$$
When this pairing is restricted to ${\cal H}^+ \times {\cal H}^-$,
it is $\mathfrak{gl}(2,\HC)$-invariant, independent of the choice of $R>0$,
non-degenerate and antisymmetric
$$
(\phi_1,\phi_2)_R = - (\phi_2,\phi_1)_R,
\qquad \phi_1 \in {\cal H}^+, \: \phi_2 \in {\cal H}^-.
$$
We have the following orthogonality relations with respect to the pairing
(\ref{H-pairing}):
\begin{equation}  \label{t-orthog}
\bigl( t^{l'}_{n'\,\underline{m'}}(Z), t^l_{m\underline{n}}(Z^{-1}) \cdot N(Z)^{-1} \bigr)_R
= -\bigl(t^l_{m\underline{n}}(Z^{-1}) \cdot N(Z)^{-1},t^{l'}_{n'\,\underline{m'}}(Z)\bigr)_R
= \delta_{ll'} \delta_{mm'} \delta_{nn'},
\end{equation}
where the indices $l,m,n$ are
$l = 0, \frac12, 1, \frac32, \dots$, $m,n \in \BB Z +l$, $-l \le m,n \le l$
and similarly for $l',m',n'$.

Let $\widetilde{\Zh}$ denote the space of $\BB C$-valued functions on $\HC$
(possibly with singularities) which are holomorphic with respect to the
complex variables $z_{11}$, $z_{12}$, $z_{21}$, $z_{22}$.
(There are no differential equations imposed on functions in $\widetilde{\Zh}$
whatsoever.) We recall the action of $GL(2,\HC)$ on $\widetilde{\Zh}$ given by
equation (49) in \cite{FL1}:
$$
\rho_1(h): \: f(Z) \quad \mapsto \quad \bigl( \rho_1(h)f \bigr)(Z) =
\frac {f \bigl( (aZ+b)(cZ+d)^{-1} \bigr)}{N(cZ+d) \cdot N(a'-Zc')},
$$
where
$h = \bigl(\begin{smallmatrix} a & b \\ c & d \end{smallmatrix}\bigr)
\in GL(2,\HC)$ and
$h^{-1} = \bigl(\begin{smallmatrix} a' & b' \\ c' & d' \end{smallmatrix}\bigr)$.
We have a natural $GL(2,\HC)$-equivariant multiplication map
$$
M : (\pi_l^0, \widetilde{\cal H}) \otimes (\pi_r^0, \widetilde{\cal H})
\to (\rho_1,\widetilde{\Zh})
$$
which is determined on pure tensors by
$$
M \bigl( \phi_1(Z_1) \otimes \phi_2(Z_2) \bigr) = (\phi_1 \cdot \phi_2)(Z),
\qquad \phi_1, \phi_2 \in \widetilde{\cal H}.
$$
Differentiating the $\rho_1$-action, we obtain an action of
$\mathfrak{gl}(2,\HC) \simeq \mathfrak{gl}(4,\BB C)$.
For convenience we recall Lemma 68 from \cite{FL1}.

\begin{lem}  \label{rho-algebra-action}
Let $\partial = \bigl( \begin{smallmatrix} \partial_{11} & \partial_{21} \\
\partial_{12} & \partial_{22} \end{smallmatrix} \bigr)= \frac 12 \nabla$, where
$\partial_{ij} = \frac{\partial}{\partial z_{ij}}$.
The Lie algebra action $\rho_1$ of $\mathfrak{gl}(2,\HC)$ on $\widetilde{\Zh}$
is given by
\begin{align*}
\rho_1 \begin{pmatrix} A & 0 \\ 0 & 0 \end{pmatrix} &:
f \mapsto \tr \bigl( A \cdot (-Z \cdot \partial f - f) \bigr)  \\
\rho_1 \begin{pmatrix} 0 & B \\ 0 & 0 \end{pmatrix} &:
f \mapsto \tr \bigl( B \cdot (-\partial f ) \bigr)  \\
\rho_1 \begin{pmatrix} 0 & 0 \\ C & 0 \end{pmatrix} &:
f \mapsto \tr \Bigl( C \cdot \bigl(
Z \cdot (\partial f) \cdot Z +2Zf \bigr) \Bigr)
= \tr \Bigl( C \cdot \bigl(Z \cdot \partial (Zf) \bigr) \Bigr)  \\
\rho_1 \begin{pmatrix} 0 & 0 \\ 0 & D \end{pmatrix} &:
f \mapsto \tr \Bigl( D \cdot \bigl( (\partial f) \cdot Z + f \bigr) \Bigr)
= \tr \Bigl( D \cdot \bigl( \partial (Zf) - f \bigr) \Bigr).
\end{align*}
\end{lem}

This lemma implies that $\mathfrak{gl}(2, \HC)$ preserves the spaces
\begin{align*}
\Zh^+ &= \{\text{polynomial functions on $\HC$}\}
= \BB C[z_{11},z_{12},z_{21},z_{22}]
\qquad \qquad \text{and} \\
\Zh &= \{\text{polynomial functions on $\HC^{\times}$}\}
= \BB C[z_{11},z_{12},z_{21},z_{22},N(Z)^{-1}].
\end{align*}
Define
$$
\Zh^- = \biggl\{ f \in \Zh;\:
\rho_1\begin{pmatrix} 0 & 1 \\ 1 & 0 \end{pmatrix}f(Z)
= N(Z)^{-2} \cdot f(Z^{-1}) \in \Zh^+ \biggr\},
$$
this is another $\mathfrak{gl}(2,\HC)$-invariant space.
Comparing this with Definition 16 in \cite{FL1}, we can say that
$\Zh^-$ consists of those elements of $\Zh$ that are regular at infinity
according to the $\rho_1$-action of $GL(2,\HC)$.
Note that $\Zh^- \oplus \Zh^+$ is a proper subspace of $\Zh$.

Next we describe an invariant bilinear pairing on $\Zh$.
Recall Proposition 69 from \cite{FL1}:

\begin{prop}
The representation $(\rho_1,\Zh)$ of $\mathfrak{gl}(2,\HC)$
has a non-degenerate symmetric bilinear pairing
\begin{equation}  \label{1-pairing}
\langle f_1,f_2 \rangle =
\frac i{2\pi^3} \int_{U(2)_R} f_1(Z) \cdot f_2(Z) \,dV,
\qquad f_1, f_2 \in \Zh,
\end{equation}
where $R>0$. This bilinear pairing is $\mathfrak{gl}(2,\HC)$-invariant and
independent of the choice of $R>0$.
\end{prop}

We have the following orthogonality relations with respect to the pairing
(\ref{1-pairing}):
\begin{equation}  \label{orthogonality}
\bigl\langle t^{l'}_{n'\,\underline{m'}}(Z) \cdot N(Z)^{k'},
t^l_{m\underline{n}}(Z^{-1}) \cdot N(Z)^{-k-2} \bigr\rangle
= \frac1{2l+1} \delta_{kk'}\delta_{ll'} \delta_{mm'} \delta_{nn'},
\end{equation}
where the indices $k,l,m,n$ are $k \in \BB Z$,
$l = 0, \frac12, 1, \frac32, \dots$, $m,n \in \BB Z +l$, $-l \le m,n \le l$
and similarly for $k',l',m',n'$.

We know from \cite{JV} and \cite{FL1} that the representations
$(\rho_1, \Zh^+)$ and $(\rho_1, \Zh^-)$ are $\BB C$-linear dual to each other
with respect to (\ref{1-pairing}), irreducible when restricted to
$\mathfrak{sl}(2,\HC)$ and possess inner products which make them unitary
representations of the real form $\mathfrak{su}(2,2)$ of $\mathfrak{sl}(2,\HC)$,
where we regard $\mathfrak{su}(2,2)$ and $\mathfrak{u}(2,2)$ as subalgebras
of $\mathfrak{gl}(2,\HC)$ as in (\ref{u(2,2)}).

We often regard the group $U(2,2)$ as a subgroup of $GL(2,\HC)$
as described in Subsection 3.5 of \cite{FL1}. That is
$$
U(2,2) = \Biggl\{ \begin{pmatrix} a & b \\ c & d \end{pmatrix}
\in GL(2,\HC) ;\: a,b,c,d \in \HC,\:
\begin{matrix} a^*a = 1+c^*c \\ d^*d = 1+b^*b \\ a^*b=c^*d \end{matrix}
\Biggr\}.
$$
The Lie algebra of $U(2,2)$ is
\begin{equation}  \label{u(2,2)}
\mathfrak{u}(2,2) = \Bigl\{
\begin{pmatrix} A & B \\ B^* & D \end{pmatrix} \in \mathfrak{gl}(2,\HC);\:
A,B,D \in \HC ,\: A=-A^*, D=-D^* \Bigr\}.
\end{equation}
The maximal compact subgroup of $U(2,2)$ is
\begin{equation} \label{U(2)xU(2)}
U(2) \times U(2) = \biggl\{
\begin{pmatrix} a & 0 \\ 0 & d \end{pmatrix} \in GL(2, \HC);\:
a,d \in \HC, \: a^*a=d^*d=1 \biggr\}.
\end{equation}

The group $U(2,2)$ acts on $\HC$ by fractional linear transformations
(\ref{conformal-action}) preserving $U(2) \subset \HC$ and open domains
\begin{equation*}  
\BB D^+ = \{ Z \in \HC;\: ZZ^*<1 \}, \qquad
\BB D^- = \{ Z \in \HC;\: ZZ^*>1 \},
\end{equation*}
where the inequalities $ZZ^*<1$ and $ZZ^*>1$ mean that the matrix $ZZ^*-1$
is negative and positive definite respectively.
The sets $\BB D^+$ and $\BB D^-$ both have $U(2)$ as the Shilov boundary.

Similarly, for each $R>0$ we can define a conjugate of $U(2,2)$
$$
U(2,2)_R = \begin{pmatrix} R & 0 \\ 0 & 1 \end{pmatrix} U(2,2)
\begin{pmatrix} R^{-1} & 0 \\ 0 & 1 \end{pmatrix} \quad \subset GL(2,\HC).
$$
Each group $U(2,2)_R$ is a real form of $GL(2,\HC)$, preserves $U(2)_R$
and open domains
\begin{equation}  \label{D_R}
\BB D^+_R = \{ Z \in \HC ;\: ZZ^*<R^2 \}, \qquad
\BB D^-_R = \{ Z \in \HC ;\: ZZ^*>R^2 \}.
\end{equation}
These sets $\BB D^+_R$ and $\BB D^-_R$ both have $U(2)_R$ as the Shilov boundary.

\section{$K$-type Basis of $(\rho_1,\Zh)$}  \label{K-type-section}

In this section we describe a convenient basis of $(\rho_1,\Zh)$ consisting
of $K$-types for the maximal compact subgroup $U(2) \times U(2)$ of $U(2,2)$.

\begin{prop}  \label{basis-prop}
The functions 
\begin{equation}  \label{Zh-basis}
t^l_{n\,\underline{m}}(Z) \cdot N(Z)^k, \qquad
l=0, \frac12, 1, \frac32, \dots, \quad m,n=-l,-l+1,\dots,l, \quad k=0,1,2,\dots,
\end{equation}
form a vector space basis of $\Zh^+ = \BB C[z_{11},z_{12},z_{21},z_{22}]$.
\end{prop}

\begin{proof}
Clearly, the functions $t^l_{n\,\underline{m}}(Z) \cdot N(Z)^k$ are polynomials.
From the orthogonality relations (\ref{orthogonality}) it follows that
they are linearly independent.
It remains to show that they span all of $\Zh^+$.
We can do that by comparing the dimensions of the subspaces homogeneous
functions of degree $d$ in $\Zh^+$ and the space spanned by (\ref{Zh-basis}).

The number of monomials
$(z_{11})^{\alpha_{11}}(z_{12})^{\alpha_{12}}(z_{21})^{\alpha_{21}}(z_{22})^{\alpha_{22}}$
in $\Zh^+$ with $\alpha_{11} + \alpha_{12} + \alpha_{21} + \alpha_{22} = d$ is
\begin{equation}  \label{dim1}
\begin{pmatrix} d+3 \\ 3 \end{pmatrix} = \frac{(d+3)(d+2)(d+1)}6.
\end{equation}
On the other hand, for $k$ and $l$ fixed, there are exactly $(2l+1)^2$
basis elements (\ref{Zh-basis}) and they are all homogeneous of degree $2l+2k$.
Therefore, the dimension of the subspace of homogeneous functions of degree $d$
inside the span of (\ref{Zh-basis}) is
\begin{equation}  \label{dim2}
(d+1)^2+(d-1)^2+(d-3)^2+\dots.
\end{equation}
Finally, it is easy to show by induction that (\ref{dim1}) and (\ref{dim2})
are in fact equal.
\end{proof}

We conclude this section with a decomposition of $(\rho_1,\Zh)$ into $K$-types.

\begin{cor}
The functions 
$$
t^l_{n\,\underline{m}}(Z) \cdot N(Z)^k, \qquad
l=0, \frac12, 1, \frac32, \dots, \quad m,n=-l,-l+1,\dots,l, \quad k \in\BB Z,
$$
form a vector space basis of $\Zh = \BB C[z_{11},z_{12},z_{21},z_{22},N(Z)^{-1}]$.
\end{cor}

\begin{proof}
The functions $t^l_{n\,\underline{m}}(Z) \cdot N(Z)^k$ are linearly independent
by (\ref{orthogonality}) and
by Proposition \ref{basis-prop} span the entire space $\Zh$.
\end{proof}

\section{Irreducible Components of $(\rho_1,\Zh)$}  \label{irred-section}

In this section we decompose $(\rho_1,\Zh)$ into irreducible components,
identify these irreducible components as images of multiplication maps
and describe their unitary structures.

\begin{thm}  \label{Zh-decomposition}
The representation $(\rho_1,\Zh)$ of $\mathfrak{gl}(2,\HC)$ has the
following decomposition into irreducible components:
$$
(\rho_1,\Zh) = (\rho_1,\Zh^-) \oplus (\rho_1,\Zh^0) \oplus (\rho_1,\Zh^+),
$$
where
\begin{align*}
\Zh^+ &= \BB C-\text{span of }
\bigl\{ t^l_{n\,\underline{m}}(Z) \cdot N(Z)^k;\: k \ge 0 \bigr\}, \\
\Zh^- &= \BB C-\text{span of }
\bigl\{ t^l_{n\,\underline{m}}(Z) \cdot N(Z)^k;\: k \le -(2l+2) \bigr\}, \\
\Zh^0 &= \BB C-\text{span of }
\bigl\{ t^l_{n\,\underline{m}}(Z) \cdot N(Z)^k;\: -(2l+1) \le k \le -1 \bigr\}.
\end{align*}
\end{thm}

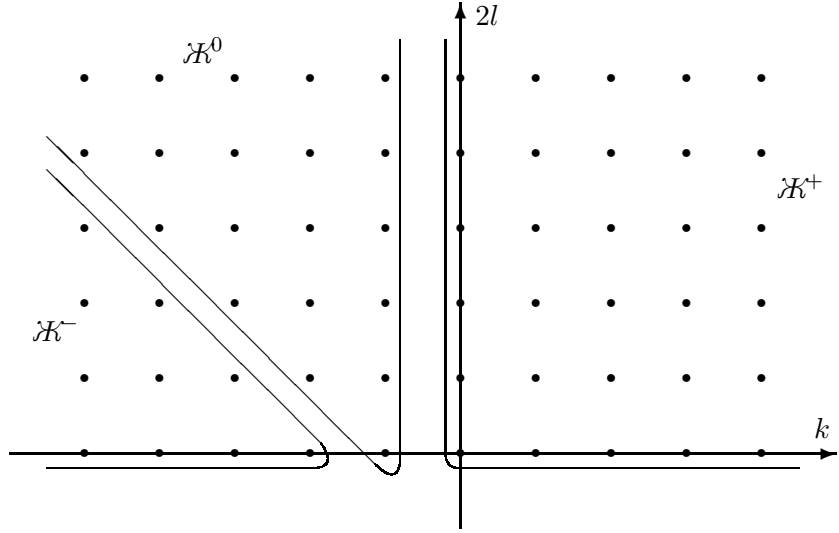
\begin{figure}
\begin{center}
\setlength{\unitlength}{1mm}
\begin{picture}(110,70)
\multiput(10,10)(10,0){10}{\circle*{1}}
\multiput(10,20)(10,0){10}{\circle*{1}}
\multiput(10,30)(10,0){10}{\circle*{1}}
\multiput(10,40)(10,0){10}{\circle*{1}}
\multiput(10,50)(10,0){10}{\circle*{1}}
\multiput(10,60)(10,0){10}{\circle*{1}}
\thicklines
\put(60,0){\vector(0,1){70}}
\put(0,10){\vector(1,0){110}}
\thinlines
\put(58,10){\line(0,1){55}}
\put(60,8){\line(1,0){45}}
\put(52,10){\line(0,1){55}}
\put(5,8){\line(1,0){35}}
\put(41.4,11.4){\line(-1,1){36.4}}
\put(48.6,8.6){\line(-1,1){43.6}}
\qbezier(58,10)(58,8)(60,8)
\qbezier(40,8)(43.8,8)(41.4,11.4)
\qbezier(52,10)(52,5.2)(48.6,8.6)
\put(62,67){$2l$}
\put(107,12){$k$}
\put(3,25){$\Zh^-$}
\put(23,62){$\Zh^0$}
\put(102,44){$\Zh^+$}
\end{picture}
\end{center}
\caption{Decomposition of $(\rho_1,\Zh)$ into irreducible components}
\end{figure}


\begin{proof}
Note that the basis elements (\ref{Zh-basis}) consist of functions
of the kind
$$
f_l(Z) \cdot N(Z)^k, \qquad \square f_l(Z)=0, \quad
l=0, \frac12, 1, \frac32, \dots, \quad k \in\BB Z,
$$
where the functions $f_l(Z)$ range over a basis of harmonic functions which
are polynomials of degree $2l$.
Recall that we consider $U(2) \times U(2)$ as a subgroup of $GL(2,\HC)$
via (\ref{U(2)xU(2)}).
For $k$ and $l$ fixed, these functions span an irreducible representation
of $U(2) \times U(2)$, which -- when restricted to $SU(2) \times SU(2)$ --
becomes isomorphic to $V_l \boxtimes V_l$, where $V_l$ denotes the
irreducible representation of $SU(2)$ of dimension $2l+1$.

To determine the effect of matrices of the kind
$\bigl( \begin{smallmatrix} 0 & B \\ 0 & 0 \end{smallmatrix} \bigr)
\in \mathfrak{gl}(2,\HC)$ with $B \in \HC$,
we use Lemma \ref{rho-algebra-action} describing their action and compute
$$
\partial \bigl( f_l(Z) \cdot N(Z)^k \bigr)
= \partial f_l \cdot N(Z)^k + k Z^+f_l \cdot N(Z)^{k-1}.
$$
By direct computation we have:
$$
\partial f_l \cdot N(Z) = Z^+ \deg f_l - Z^+ \cdot (\partial^+ f_l) \cdot Z^+
= 2l Z^+f_l - Z^+ \cdot (\partial^+ f_l) \cdot Z^+,
$$
$$
\square (Z^+f_l) = Z^+ \square f_l + 4\partial f_l
\qquad \text{and} \qquad
\square ( N(Z) \cdot g ) = N(Z) \cdot \square g + 4(\deg + 2) g.
$$
Hence we can write
\begin{equation}  \label{Z^+f}
Z^+f_l = \Bigl(Z^+f_l - \frac{\partial f_l \cdot N(Z)}{2l+1} \Bigr)
+ \frac{\partial f_l \cdot N(Z)}{2l+1}
= \frac{Z^+ \cdot (\partial^+ f_l) \cdot Z^+ + Z^+f_l}{2l+1}
+ \frac{\partial f_l \cdot N(Z)}{2l+1}
\end{equation}
and
\begin{equation}  \label{B-action}
\partial \bigl( f_l(Z) \cdot N(Z)^k \bigr)
= \frac{2l+k+1}{2l+1} \partial f_l \cdot N(Z)^k
+ \frac{k}{2l+1} \bigl( Z^+ \cdot (\partial^+ f_l) \cdot Z^+ + Z^+f_l \bigr)
\cdot N(Z)^{k-1}
\end{equation}
with $\partial f_l$ and $Z^+ \cdot (\partial^+ f_l) \cdot Z^+ + Z^+f_l$
being harmonic and having degrees $2l-1$ and $2l+1$ respectively.

Next we determine the effect of matrices of the kind
$\bigl( \begin{smallmatrix} 0 & 0 \\ C & 0 \end{smallmatrix} \bigr)
\in \mathfrak{gl}(2,\HC)$ with
$C \in \HC$. Again, we use Lemma \ref{rho-algebra-action} and compute
$$
Z \cdot \partial \bigl( f_l \cdot N(Z)^k \bigr) \cdot Z + 2 Zf_l \cdot N(Z)^k
= Z \cdot (\partial f_l) \cdot Z \cdot N(Z)^k
+ (k+2) Zf_l \cdot N(Z)^k.
$$
Conjugating (\ref{Z^+f}) we see that
$$
Zf_l = \frac{Z \cdot (\partial f_l) \cdot Z + Zf_l}{2l+1}
+ \frac{\partial^+ f_l \cdot N(Z)}{2l+1}.
$$
Therefore,
\begin{multline}  \label{C-action}
Z \cdot \partial \bigl( f_l \cdot N(Z)^k \bigr) \cdot Z + 2 Zf_l \cdot N(Z)^k \\
= \frac{2l+k+2}{2l+1} \bigl( Z \cdot (\partial f_l) \cdot Z + Zf_l \bigr)
\cdot N(Z)^k + \frac{k+1}{2l+1} \partial^+ f_l \cdot N(Z)^{k+1}
\end{multline}
with $Z \cdot (\partial f_l) \cdot Z + Zf_l$ and $\partial^+ f_l$
being harmonic and having degrees $2l+1$ and $2l-1$ respectively.

The actions of
$\bigl( \begin{smallmatrix} A & 0 \\ 0 & 0 \end{smallmatrix} \bigr)$,
$\bigl( \begin{smallmatrix} 0 & B \\ 0 & 0 \end{smallmatrix} \bigr)$,
$\bigl( \begin{smallmatrix} 0 & 0 \\ C & 0 \end{smallmatrix} \bigr)$ and
$\bigl( \begin{smallmatrix} 0 & 0 \\ 0 & D \end{smallmatrix} \bigr)$
are illustrated in Figure \ref{actions}. In the diagram describing
$\rho_1\bigl( \begin{smallmatrix} 0 & B \\ 0 & 0 \end{smallmatrix} \bigr)$
the vertical arrow disappears if $l=0$ or $2l+k+1=0$
and the diagonal arrow disappears if $k=0$.
Similarly, in the diagram describing
$\rho_1\bigl( \begin{smallmatrix} 0 & 0 \\ C & 0 \end{smallmatrix} \bigr)$
the vertical arrow disappears if $2l+k+2=0$ and the
diagonal arrow disappears if $k=-1$ or $l=0$.
This proves that $\Zh^+$, $\Zh^-$ and $\Zh^0$ are
$\mathfrak{gl}(2,\HC)$-invariant subspaces of $\Zh$.
Note that 
$$
\tr (Z \cdot \partial f + f) =
\tr \begin{pmatrix}
z_{11} \partial_{11}f + z_{12} \partial_{12}f + f & \ast \ast \ast \\
\ast \ast \ast & z_{21} \partial_{21}f + z_{22} \partial_{22}f + f
\end{pmatrix}
= (\deg+2)f,
$$
hence
$Z \cdot (\partial f_l) \cdot Z + Zf_l = (Z \cdot \partial f_l + f_l) \cdot Z$
and its conjugate $Z^+ \cdot (\partial^+ f_l) \cdot Z^+ + Z^+f_l$ are never zero.
It follows from (\ref{B-action}) and (\ref{C-action}) that the
subrepresentations $(\rho_1, \Zh^+)$, $(\rho_1, \Zh_2^-)$, $(\rho_1, \Zh^0)$
are irreducible with respect to $\rho_1$-action of $\mathfrak{gl}(2,\HC)$.
\end{proof}

\begin{figure}
\begin{center}
\begin{subfigure}[b]{0.21\textwidth}
\centering
\includegraphics[scale=0.2]{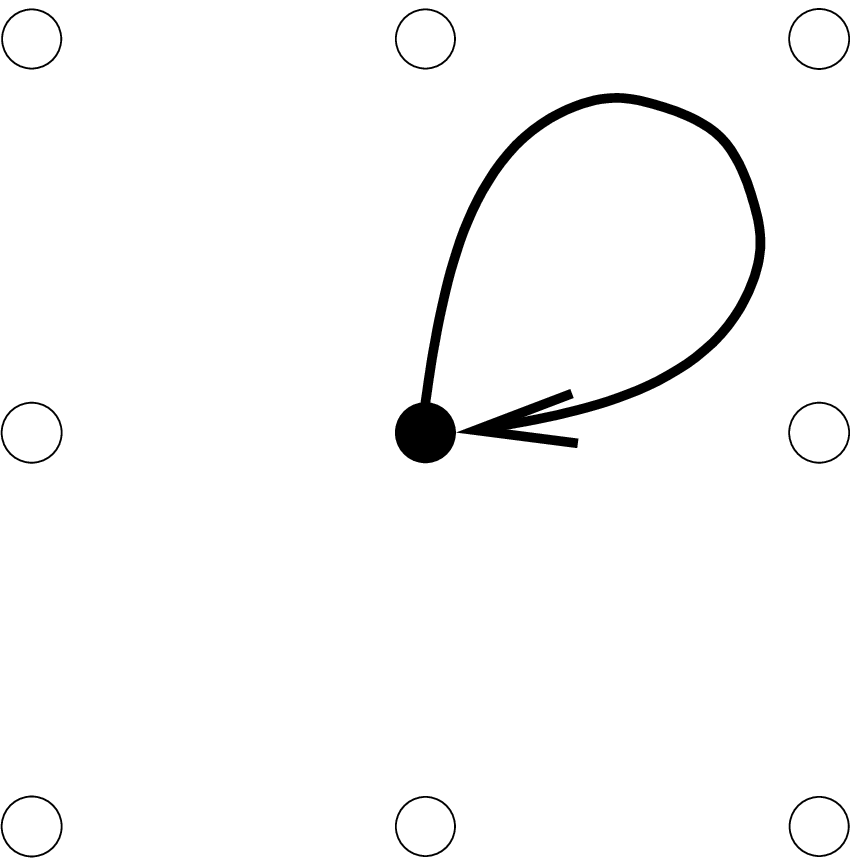}
\caption{Action of
$\rho_1\bigl(\begin{smallmatrix} A & 0 \\ 0 & 0 \end{smallmatrix}\bigr)$}
\end{subfigure}
\quad
\begin{subfigure}[b]{0.21\textwidth}
\centering
\includegraphics[scale=0.2]{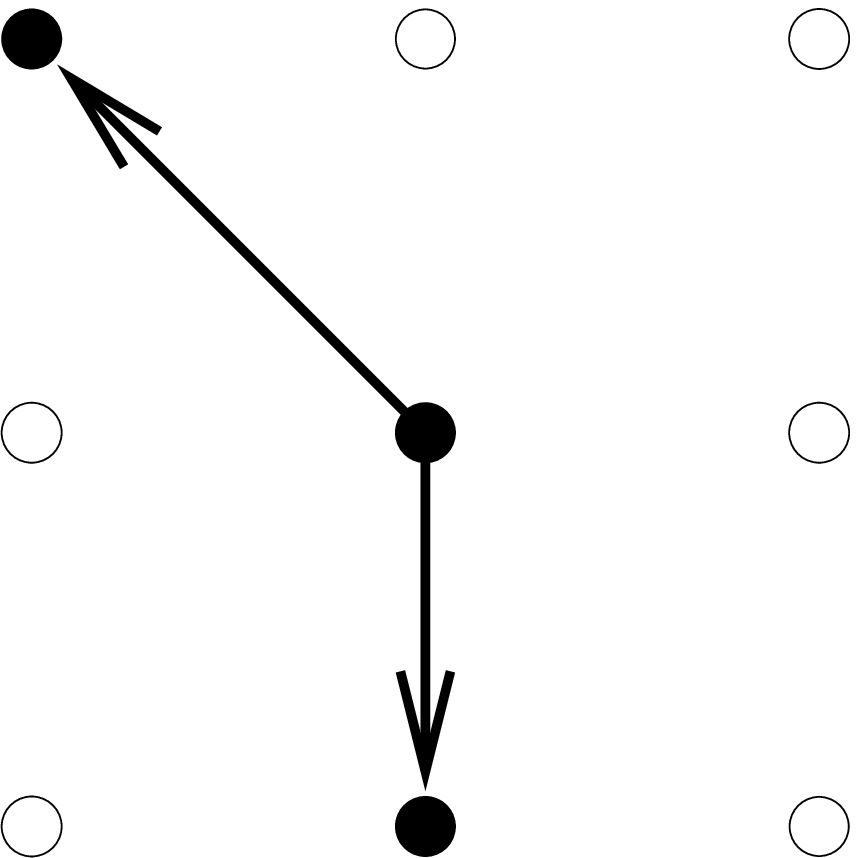}
\caption{Action of
$\rho_1\bigl(\begin{smallmatrix} 0 & B \\ 0 & 0 \end{smallmatrix}\bigr)$}
\end{subfigure}
\quad
\begin{subfigure}[b]{0.21\textwidth}
\centering
\includegraphics[scale=0.2]{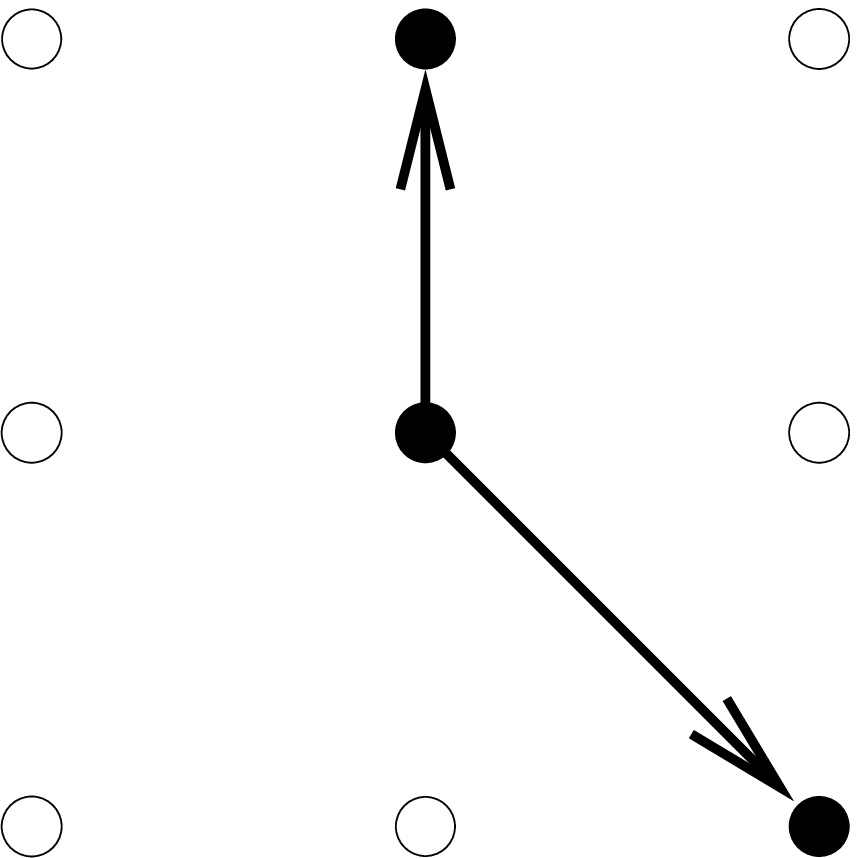}
\caption{Action of
$\rho_1\bigl(\begin{smallmatrix} 0 & 0 \\ C & 0 \end{smallmatrix}\bigr)$}
\end{subfigure}
\quad
\begin{subfigure}[b]{0.21\textwidth}
\centering
\includegraphics[scale=0.2]{AD.eps}
\caption{Action of
$\rho_1\bigl(\begin{smallmatrix} 0 & 0 \\ 0 & D \end{smallmatrix}\bigr)$}
\end{subfigure}
\end{center}
\caption{}
\label{actions}
\end{figure}


Our next task is to identify the images under the natural
$\mathfrak{gl}(2,\HC)$-equivariant multiplication maps:
\begin{equation}  \label{M}
M: (\pi^0_l, {\cal H}^\pm) \otimes (\pi^0_r, {\cal H}^\pm) \to (\rho_1, \Zh)
\end{equation}
sending pure tensors
$$
\phi_1(Z_1) \otimes \phi_2(Z_2) \mapsto (\phi_1 \cdot \phi_2)(Z).
$$

\begin{lem}  \label{image-lemma}
Under the multiplication maps
$(\pi^0_l, {\cal H}^\pm) \otimes (\pi^0_r, {\cal H}^\pm) \to (\rho_1, \Zh)$,
\begin{enumerate}
\item
The image of ${\cal H}^+ \otimes {\cal H}^+$ in $\Zh$ is $\Zh^+$;
\item
The image of ${\cal H}^- \otimes {\cal H}^-$ in $\Zh$ is $\Zh^-$;
\item
The image of ${\cal H}^- \otimes {\cal H}^+$ in $\Zh$ is $\Zh^0$.
\end{enumerate}
\end{lem}

\begin{proof}
Note that the space ${\cal H}^+$ consists of harmonic polynomials.
The product of two polynomials is another polynomial,
hence the image of ${\cal H}^+ \otimes {\cal H}^+$ lies in $\Zh^+$.
Since $(\rho_1, \Zh)$ is irreducible, the image is all of $\Zh^+$.

Applying
$(\pi^0_l \otimes \pi^0_r)
\bigl(\begin{smallmatrix} 0 & 1 \\ 1 & 0 \end{smallmatrix}\bigr)$
to the left hand side of ${\cal H}^+ \otimes {\cal H}^+ \to \Zh^+$ and
$\rho_1 \bigl(\begin{smallmatrix} 0 & 1 \\ 1 & 0 \end{smallmatrix}\bigr)$
to the right hand side,
we see that the image of ${\cal H}^- \otimes {\cal H}^-$ is $\Zh^-$.

Let us denote by $J$ the image of ${\cal H}^- \otimes {\cal H}^+$ in $\Zh$.
Clearly, $J$ contains the function $N(Z)^{-1}$, which generates $\Zh^0$.
Hence $\Zh^0 \subset J$. It remains to show that $J \subset \Zh^0$.
By Theorem \ref{Zh-decomposition}, if $\Zh^0 \subsetneq J$,
then $J$ also contains $\Zh^+$ or $\Zh^-$ and hence functions $N(Z)^k$
with $k \ne -1$.
Thus it is sufficient to prove that $J$ cannot contain $N(Z)^k$ with $k \ne -1$.

By construction, $J$ is spanned by
\begin{equation}  \label{N(Z)^k}
N(Z)^{-(2l+1)} \cdot t^l_{n\,\underline{m}}(Z) \cdot t^{l'}_{n'\,\underline{m'}}(Z).
\end{equation}
Note that if $V_l$ and $V_{l'}$ are two irreducible representations of $SU(2)$
of dimensions $2l+1$ and $2l'+1$ respectively, then their tensor product
contains a copy of the trivial representation if and only if $l=l'$.
This means that a linear combination of the functions (\ref{N(Z)^k})
can express $N(Z)^k$ only if $l=l'$.
But then the homogeneity degree of (\ref{N(Z)^k}) is $-2$.
Therefore, $N(Z)^k \notin J$ if $k \ne -1$.
\end{proof}

As we have mentioned, the representations $(\rho_1, \Zh^+)$ and $(\rho_1, \Zh^-)$
are $\BB C$-linear dual to each other with respect to (\ref{1-pairing}).
On the other hand, the $\BB C$-linear dual of $(\rho_1, \Zh^0)$
with respect to (\ref{1-pairing}) is $(\rho_1, \Zh^0)$ itself.
We conclude this section with an explicit description of the unitary
structures on $(\rho_1,\Zh^+)$, $(\rho_1,\Zh^-)$ and $(\rho_1,\Zh^0)$. Define
\begin{equation}  \label{inner-product}
(f_1, f_2) = \frac{i}{2\pi^3} \int_{U(2)} f_1(Z) \cdot \B{f_2(Z)}
\,\frac{dV}{N(Z)^2}, \qquad f_1,f_2 \in \Zh.
\end{equation}
This pairing is an inner product.

\begin{prop}
The restrictions of $(\rho_1,\Zh^+)$, $(\rho_1,\Zh^-)$ and $(\rho_1,\Zh^0)$ to
$\mathfrak{u}(2,2)$ are unitary with respect to the inner product
(\ref{inner-product}).
\end{prop}

\begin{proof}
We only need to prove that the pairing (\ref{inner-product}) is
$\mathfrak{u}(2,2)$-invariant. It is enough to show that, for all
$h \in U(2,2)$ sufficiently close to the identity element, we have
$$
(f_1, f_2) = \bigl(\rho_1(h)f_1,\rho_1(h)f_2\bigr), \qquad f_1,f_2 \in \Zh.
$$
If $h^{-1}=
\bigl(\begin{smallmatrix} a & b \\ c & d \end{smallmatrix} \bigr)\in U(2,2)$,
then
$h = \bigl(\begin{smallmatrix} a^* & -c^* \\ -b^* & d^* \end{smallmatrix}\bigr)$.
(If $Z \in \HC$, $Z^* \in \HC$ denotes the matrix adjoint of $Z$ under
the standard identification of $\HC$ with $2 \times 2$ complex matrices,
see \cite{FL1} for details.) Writing $\tilde Z = (aZ+b)(cZ+d)^{-1}$ and
using Lemma \ref{Jacobian_lemma} together with the fact that $U(2,2)$ preserves
$U(2)=\{Z \in \HC ;\: Z^* = Z^{-1} \}$ we obtain:
\begin{multline*}
-2\pi^3i \cdot \bigl(\rho_1(h) f_1,\rho_1(h) f_2\bigr)  \\
= \int_{Z \in U(2)} \frac{f_1(\tilde Z)}{N(cZ+d) \cdot N(a^*+Zb^*)} \cdot
\B{\frac{f_2(\tilde Z)}{N(cZ+d) \cdot N(a^*+Zb^*)}} \,\frac{dV}{N(Z)^2}  \\
= \int_{\tilde Z \in U(2)} f_1(\tilde Z) \cdot \B{f_2(\tilde Z)}
\frac{N(c^*+Zd^*) \cdot N(aZ+b)}{\B{N(cZ+d)} \cdot \B{N(a^*+Zb^*)} \cdot N(Z)^2}
\,\frac{dV}{N(\tilde Z)^2} \\
= \int_{\tilde Z \in U(2)} f_1(\tilde Z) \cdot \B{f_2(\tilde Z)}
\,\frac{dV}{N(\tilde Z)^2}
= -2\pi^3i \cdot (f_1, f_2).
\end{multline*}
\end{proof}

\section{Formal Calculation of the Reproducing Kernel for $(\rho_1, \Zh^0)$}
\label{formal-section}

In \cite{FL1}, Proposition 27, we computed the reproducing kernels for
$(\rho_1, \Zh^+)$ and $(\rho_1, \Zh^-)$ by finding expansions for
$\frac1{N(Z-W)^2}$ in terms of basis functions (\ref{Zh-basis}).
In both cases the reproducing kernel is $\frac1{N(Z-W)^2}$,
but one gets different results depending on whether $ZW^{-1}$ lies in
$\BB D^+$ or $\BB D^-$:

\begin{prop}[Proposition 27, \cite{FL1}]  \label{prop27}
We have the following matrix coefficient expansions
$$
\frac 1{N(Z-W)^2} = \sum_{k,l,m,n}
(2l+1) t^l_{m \, n}(Z^{-1}) \cdot N(Z)^{-k-2} \cdot t^l_{n \, m}(W) \cdot N(W)^k
$$
which converges pointwise absolutely in the region
$\{ (Z,W) \in \HC^{\times} \times \HC; WZ^{-1} \in \BB D^+ \}$, and
$$
\frac 1{N(Z-W)^2} = \sum_{k,l,m,n} 
(2l+1) t^l_{m \, n}(Z) \cdot N(Z)^k \cdot t^l_{n \, m}(W^{-1}) \cdot N(W)^{-k-2}
$$
which converges pointwise absolutely in the region
$\{ (Z,W) \in \HC \times \HC^{\times}; ZW^{-1} \in \BB D^+\}$.
The sums are taken first over all $m,n = -l, -l+1, \dots, l$,
then over $k=0,1,2,3,\dots$ and $l=0,\frac 12, 1, \frac 32,\dots$.
\end{prop}

In this section we formally compute the reproducing kernel for
$(\rho_1, \Zh^0)$. There are some issues with convergence that require
justification, but it is nice to see that this formally computed kernel
agrees with the formula for a projector onto $\Zh^0$ that will be obtained
in the next section (Theorem \ref{Zh^0-projector}).

Recall that $\Zh^0$ is the $\BB C$-span of
$\bigl\{ t^l_{n\,\underline{m}}(Z) \cdot N(Z)^k;\: -(2l+1) \le k \le -1 \bigr\}$.
In light of the orthogonality relations (\ref{orthogonality})
we would like to compute the series
\begin{equation}  \label{expansion}
\sum_{k,l,m,n} (2l+1) t^l_{n \, \underline{m}}(Z^{-1}) \cdot N(Z)^{-k-2} \cdot
t^l_{m \, \underline{n}}(W) \cdot N(W)^k,
\end{equation}
the sum is being taken over all $l=0,\frac 12, 1, \frac 32,\dots$,
$m,n \in \BB Z + l$ with $m,n = -l, -l+1, \dots, l$ and
$-(2l+1) \le k \le -1$.
By the multiplicativity property of matrix coefficients (\ref{t-mult}),
(\ref{expansion}) equals
$$
\sum_{k,l,n} \frac{2l+1}{N(Z)^2} \cdot
t^l_{n \, \underline{n}}(Z^{-1}W) \cdot N(Z^{-1}W)^k
= \sum_{l,n} \frac{2l+1}{N(Z)^2} \cdot t^l_{n \, \underline{n}}(Z^{-1}W) \cdot
\frac{N(Z^{-1}W)^{-(2l+1)}-1}{1-N(Z^{-1}W)}.
$$
Assume further that $Z^{-1}W$ can be diagonalized as
$\bigl(\begin{smallmatrix} \lambda_1 & 0 \\
0 & \lambda_2 \end{smallmatrix}\bigr)$ with
$\lambda_1 \ne \lambda_2$.
This is allowed since the set of matrices with different eigenvalues is
dense in $\HC$.
Then the sum $\sum_{n} t^l_{n \, \underline{n}}(Z^{-1}W)$
is just the character $\chi_l(Z^{-1}W)$ of the irreducible representation of
$GL(2,\BB C)$ of dimension $2l+1$ and equals
$\frac{\lambda_1^{2l+1}-\lambda_2^{2l+1}}{\lambda_1-\lambda_2}$.
Hence (\ref{expansion}) is equal to
\begin{multline*}
\sum_l \frac{2l+1}{N(Z)^2} \cdot
\frac{\lambda_1^{2l+1}-\lambda_2^{2l+1}}{\lambda_1-\lambda_2}
\cdot \frac{(\lambda_1\lambda_2)^{-(2l+1)}-1}{1-\lambda_1\lambda_2} \\
=
- \sum_l \frac{2l+1}{N(Z)^2} \cdot
\frac{\lambda_1^{2l+1}-\lambda_2^{-(2l+1)}}
{(\lambda_1-\lambda_2)(1-\lambda_1\lambda_2)}
- \sum_l \frac{2l+1}{N(Z)^2} \cdot
\frac{\lambda_1^{-(2l+1)}-\lambda_2^{2l+1}}
{(\lambda_1-\lambda_2)(1-\lambda_1\lambda_2)}.
\end{multline*}
The first sum converges absolutely if $|\lambda_1|<1$ and $|\lambda_2|>1$:
\begin{multline*}
\sum_l \frac{2l+1}{N(Z)^2} \cdot
\frac{\lambda_1^{2l+1}-\lambda_2^{-(2l+1)}}
{(\lambda_1-\lambda_2)(1-\lambda_1\lambda_2)}
=
\frac{N(Z)^{-2}}{(\lambda_1-\lambda_2)(1-\lambda_1\lambda_2)}
\biggl( \frac{\lambda_1}{(1-\lambda_1)^2}
- \frac{\lambda_2}{(1-\lambda_2)^2} \biggr)  \\
=
\frac{N(Z)^{-2}}{(1-\lambda_1)^2(1-\lambda_2)^2}
= \frac{N(Z)^{-2}}{N(1-Z^{-1}W)^2}
= \frac1{N(Z-W)^2}.
\end{multline*}
The second sum converges absolutely if $|\lambda_1|>1$ and $|\lambda_2|<1$:
$$
\sum_l \frac{2l+1}{N(Z)^2} \cdot
\frac{\lambda_1^{-(2l+1)}-\lambda_2^{2l+1}}
{(\lambda_1-\lambda_2)(1-\lambda_1\lambda_2)}
= \frac1{N(Z-W)^2}.
$$
Of course, the set of $Z$ and $W$ where both sums converge absolutely is empty,
but these formal calculations strongly suggest that there is a way to make
sense of the series (\ref{expansion}) in terms of distributions:
\begin{multline}  \label{expansion-formal}
\sum_{k,l,m,n} (2l+1) t^l_{n \, \underline{m}}(Z^{-1}) \cdot N(Z)^{-k-2} \cdot
t^l_{m \, \underline{n}}(W) \cdot N(W)^k \\
= - \biggl( \operatorname{Reg}_+ \frac1{N(Z-W)^2}
+ \operatorname{Reg}_-\frac1{N(Z-W)^2} \biggr)
\end{multline}
with $ZW^{-1} \in U(2)$ and $\operatorname{Reg}_{\pm} \frac1{N(Z-W)^2}$
denoting some sort of regularizations of $\frac1{N(Z-W)^2}$.

\section{Equivariant Embeddings of and Projectors onto
the Irreducible Components of $(\rho_1, \Zh)$}
\label{embeddings-section}


In this section we construct $\mathfrak{gl}(2,\HC)$-equivariant embeddings
of the irreducible components of $(\rho_1, \Zh)$ into tensor products
$(\pi^0_l, {\cal H}^\pm) \otimes (\pi^0_r, {\cal H}^\pm)$
with the property that, when composed with the multiplication map,
the result is the identity map on that irreducible component.
The tensor product $(\pi^0_l, {\cal H}^+) \otimes (\pi^0_r, {\cal H}^+)$
was decomposed into a direct sum of irreducible components in \cite{JV}
with $(\rho_1, \Zh^+)$ being one of these components,
and it was shown that each irreducible component has multiplicity one.
Hence the $\mathfrak{gl}(2,\HC)$-equivariant map
$\Zh^+ \to {\cal H}^+ \otimes {\cal H}^+$
is unique up to a scalar multiple.
Dually, the multiplicity of $(\rho_1, \Zh^-)$ in
$(\pi^0_l, {\cal H}^-) \otimes (\pi^0_r, {\cal H}^-)$ is one
and the $\mathfrak{gl}(2,\HC)$-equivariant map
$\Zh^- \to {\cal H}^- \otimes {\cal H}^-$
is unique up to a scalar multiple as well.
On the other hand, the equivariant embedding
$\Zh^0 \hookrightarrow {\cal H}^- \otimes {\cal H}^+$
requires a more subtle approach.
As an immediate application of these embedding maps we obtain
projectors of $(\rho_1, \Zh)$ onto its irreducible components.

We consider the maps
\begin{equation}  \label{fork}
\Zh \ni f \quad \mapsto \quad (I_R f)(Z_1,Z_2) =
\frac i{2\pi^3} \int_{W \in U(2)_R} \frac{f(W) \,dV}{N(W-Z_1) \cdot N(W-Z_2)}
\quad \in \B{{\cal H} \otimes {\cal H}},
\end{equation}
where $\B{{\cal H} \otimes {\cal H}}$ denotes the Hilbert space obtained by
completing ${\cal H} \otimes {\cal H}$ with respect to the unitary structure
coming from the tensor product of unitary representations
$(\pi^0_l, {\cal H})$ and $(\pi^0_r, {\cal H})$.
If $Z_1, Z_2 \in \BB D^-_R \sqcup \BB D^+_R$, the integrand has no singularities
and the result is a holomorphic function in two variables $Z_1, Z_2$
which is harmonic in each variable separately.
We will see soon that the result depends on whether $Z_1$ and $Z_2$
are both in $\BB D^+_R$, both in $\BB D^-_R$ or one is in $\BB D^+_R$ and
the other is in $\BB D^-_R$.
Thus the expression (\ref{fork}) determines four different maps.

\begin{lem}  \label{equivariance}
The maps $f \mapsto (I_R f)(Z_1,Z_2)$ are $U(2,2)_R$ and
$\mathfrak{gl}(2,\HC)$-equivariant.
\end{lem}

\begin{proof}
We need to show that, for all $h \in U(2,2)_R$, the maps (\ref{fork})
commute with the action of $h$. Writing
$h= \bigl(\begin{smallmatrix} a' & b' \\ c' & d' \end{smallmatrix}\bigr)$,
$h^{-1}= \bigl(\begin{smallmatrix} a & b \\ c & d \end{smallmatrix}\bigr)$,
$$
\tilde Z_1 = (aZ_1+b)(cZ_1+d)^{-1}, \qquad
\tilde Z_2 = (aZ_2+b)(cZ_2+d)^{-1}, \qquad
\tilde W = (aW+b)(cW+d)^{-1}
$$
and using Lemmas \ref{Z-W} and \ref{Jacobian_lemma} we obtain:
\begin{multline*}
\int_{W \in U(2)_R} \frac {(\rho_1(h)f)(W) \,dV}{N(W-Z_1) \cdot N(W-Z_2)} \\
= \int_{W \in U(2)_R} \frac{f(\tilde W) \cdot N(cW+d)^{-2} \cdot N(a'-Wc')^{-2} \,dV}
{N(\tilde W-\tilde Z_1) \cdot N(\tilde W- \tilde Z_2) \cdot
N(cZ_1+d) \cdot N(a'-Z_2c')}  \\
= \frac1{N(cZ_1+d) \cdot N(a'-Z_2c')} \int_{\tilde W \in U(2)_R}
\frac{f(\tilde W) \,dV} {N(\tilde W-\tilde Z_1) \cdot N(\tilde W- \tilde Z_2)}.
\end{multline*}
This proves the $U(2,2)_R$-equivariance.
The $\mathfrak{gl}(2,\HC)$-equivariance then follows since
$\mathfrak{gl}(2,\HC) \simeq \BB C \otimes \mathfrak{u}(2,2)_R$.
\end{proof}

Now we compose the embedding maps $I_R$ with the multiplication map $M$
defined by (\ref{M}).

\begin{thm}  \label{embedding}
The maps $f \mapsto (I_R f)(Z_1,Z_2)$ have the following properties:
\begin{enumerate}
\item  \label{one}
If $Z_1, Z_2 \in \BB D^+_R$, then $I_R: \Zh \to {\cal H}^+ \otimes {\cal H}^+$,
$$
M \circ (I_R f)(Z_1,Z_2) = f \quad \text{if $f \in \Zh^+$}
\qquad \text{and} \qquad
(I_R f)(Z_1,Z_2) = 0 \quad \text{if $f \in \Zh^- \oplus \Zh^0$};
$$
\item
If $Z_1, Z_2 \in \BB D^-_R$, then $I_R: \Zh \to {\cal H}^- \otimes {\cal H}^-$,
$$
M \circ (I_R f)(Z_1,Z_2) = f \quad \text{if $f \in \Zh^-$}
\qquad \text{and} \qquad
(I_R f)(Z_1,Z_2) = 0 \quad \text{if $f \in \Zh^0 \oplus \Zh^+$}.
$$
\end{enumerate}
\end{thm}

\begin{proof}
We prove part {\em \ref{one}} only,
the other part can be proven in the same way.
Note that the representations $(\rho_1, \Zh^-)$, $(\rho_1, \Zh^0)$ and
$(\rho_1, \Zh^+)$ are generated by $N(W)^{-2}$, $N(W)^{-1}$ and $1$ respectively.
For this reason we compute $(I_R N(W)^k)(Z_1,Z_2)$ for $k=-2,-1,0$.
Suppose $Z_1, Z_2 \in \BB D^+_R$ and use the matrix coefficient
expansion given by Proposition 25 in \cite{FL1}
\begin{equation}  \label{1/N-expansion}
\frac 1{N(Z-W)}= N(W)^{-1} \cdot \sum_{l,m,n}
t^l_{m\,\underline{n}}(Z) \cdot t^l_{n\,\underline{m}}(W^{-1}),
\qquad \begin{matrix} l=0,\frac 12, 1, \frac 32,\dots, \\
m,n = -l, -l+1, \dots, l, \end{matrix}
\end{equation}
which converges pointwise absolutely in the region
$\{ (Z,W) \in \HC \times \HC^{\times}; \: ZW^{-1} \in \BB D^+ \}$.
We compute:
\begin{multline*}
(I_R N(W)^k)(Z_1,Z_2)
= \frac i{2\pi^3} \int_{W \in U(2)_R} \frac{N(W)^k \,dV}{N(W-Z_1) \cdot N(W-Z_2)} \\
= \biggl\langle \frac{N(W)^k}{N(W-Z_1)}, \frac1{N(W^+-Z_2^+)} \biggr\rangle_W \\
= \sum_{l,m,n,l',m',n'} t^l_{n \, \underline{m}}(Z_1) \cdot
t^{l'}_{m' \, \underline{n'}}(Z_2^+) \cdot
\bigl\langle N(W)^{k-2-2l'} \cdot t^l_{m \, \underline{n}}(W^{-1}),
t^{l'}_{n' \, \underline{m'}}(W) \bigr\rangle.
\end{multline*}
By the orthogonality relations (\ref{orthogonality}) this is zero unless
$l=l'$, $m=m'$, $n=n'$ and $k=2l$.
Therefore,
$$
(I_R N(W)^k)(Z_1,Z_2) =
\begin{cases} 0 & \text{if $k=-2,-1$;} \\
1 & \text{if $k=0$.} \end{cases}
$$
By $\mathfrak{gl}(2,\HC)$-equivariance (Lemma \ref{equivariance}) we see that
$(I_R f)(Z_1,Z_2)$ is always a polynomial in $Z_1$ and $Z_2$, hence an element
of ${\cal H}^+ \otimes {\cal H}^+$ and part {\em \ref{one}} follows.
\end{proof}

\begin{ex}
Using similar computations one can show that
$$
(I_R N(W))(Z_1,Z_2) = \frac12 \tr(Z_1Z_2^+) \quad \text{and} \quad
(I_R w_{ij})(Z_1,Z_2) = \frac12 \bigl((z_{ij})_1+(z_{ij})_2\bigr),
\quad Z_1,Z_2 \in \BB D^+_R,
$$
where $w_{ij}$ denotes the $ij$-entry of the $2 \times 2$ matrix $W$.
\end{ex}

\begin{cor}  \label{proj-cor}
The maps $f \mapsto (I_Rf)(Z_1,Z_2)$ followed by the multiplication map
provide projectors onto the irreducible components of $(\rho_1, \Zh)$.
More precisely,
\begin{enumerate}
\item
If $Z \in \BB D^+_R$, then the map
$$
f \mapsto (\P^+ f)(Z) = 
\frac i{2\pi^3} \int_{W \in U(2)_R} \frac{f(W) \,dV}{N(W-Z)^2}
$$
is a projector onto $\Zh^+$;
\item
If $Z \in \BB D^-_R$, then the map
$$
f \mapsto (\P^- f)(Z) = 
\frac i{2\pi^3} \int_{W \in U(2)_R} \frac{f(W) \,dV}{N(W-Z)^2}
$$
is a projector onto $\Zh^-$.
\end{enumerate}
In particular, these maps $\P^+$ and $\P^-$ provide reproducing formulas
for functions in $\Zh^+$ and $\Zh^-$ respectively.
\end{cor}

(The reproducing formula for functions in $\Zh^+$ was obtained in
\cite{FL1}, Theorem 70.)

Now we suppose $Z_1 \in \BB D^+_R$ and $Z_2 \in \BB D^-_R$,
this case is much more subtle. Using the matrix coefficient expansion
(\ref{1/N-expansion}) of $N(Z-W)^{-1}$ one more time, we compute:
\begin{multline*}
(I_R N(W)^k)(Z_1,Z_2)
= \frac i{2\pi^3} \int_{W \in U(2)_R} \frac{N(W)^k \,dV}{N(W-Z_1) \cdot N(W-Z_2)} \\
= \biggl\langle \frac{N(W)^k}{N(W-Z_1)}, \frac1{N(W-Z_2)} \biggr\rangle_W \\
= N(Z_2)^{-1} \sum_{l,m,n,l',m',n'} t^l_{n \, \underline{m}}(Z_1) \cdot 
t^{l'}_{m' \, \underline{n'}}(Z_2^{-1}) \cdot
\bigl\langle N(W)^{k-1} \cdot t^l_{m \, \underline{n}}(W^{-1}),
t^{l'}_{n' \, \underline{m'}}(W) \bigr\rangle.
\end{multline*}
By the orthogonality relations (\ref{orthogonality}) this is zero unless
$l=l'$, $m=m'$, $n=n'$ and $k=-1$.
By $\mathfrak{gl}(2,\HC)$-equivariance (Lemma \ref{equivariance}) we can
conclude that $(I_R f)(Z_1,Z_2) =0$ if $f \in \Zh^- \oplus \Zh^+$.
So, let us assume now $l=l'$, $m=m'$, $n=n'$ and $k=-1$. In this case we get
$$
(I_R N(W)^{-1})(Z_1,Z_2)
= \sum_{l,m,n} \frac{N(Z_2)^{-1}}{2l+1} t^l_{n \, \underline{m}}(Z_1) \cdot
t^l_{m \, \underline{n}}(Z_2^{-1})
= \sum_{l,n} \frac{N(Z_2)^{-1}}{2l+1} t^l_{n \, \underline{n}}(Z_1 \cdot Z_2^{-1}).
$$
Assume further that $Z_1 \cdot Z_2^{-1}$ can be diagonalized as
$\bigl(\begin{smallmatrix} \lambda_1 & 0 \\
0 & \lambda_2 \end{smallmatrix}\bigr)$ with $\lambda_1 \ne \lambda_2$.
This is allowed since the set of matrices with different eigenvalues is
dense in $\HC$.
Since $Z_1 \in \BB D^+_R$ and $Z_2 \in \BB D^-_R$,
we have $|\lambda_1|, |\lambda_2| <1$.
Recall that $\chi_l$ denotes the character of the irreducible representation of
$GL(2,\BB C)$ of dimension $2l+1$ and $\chi_l(Z_1 \cdot Z_2^{-1})=
\frac{\lambda_1^{2l+1}-\lambda_2^{2l+1}}{\lambda_1-\lambda_2}$.
Hence
\begin{multline}  \label{I_R}
(I_R N(W)^{-1})(Z_1,Z_2)
= \sum_l \frac{N(Z_2)^{-1}}{2l+1} \chi_l(Z_1 \cdot Z_2^{-1})  \\
= \sum_l \frac{N(Z_2)^{-1}}{2l+1} \frac{\lambda_1^{2l+1}-\lambda_2^{2l+1}}
{\lambda_1-\lambda_2}
= \frac{N(Z_2)^{-1}}{\lambda_2-\lambda_1}
\log \biggl( \frac{1-\lambda_1}{1-\lambda_2} \biggr).
\end{multline}
Although this expression is valid only in the region where
$\lambda_1 \ne \lambda_2$, the right hand side clearly
continues analytically across the set of $Z_1 \cdot Z_2^{-1}$ for which
$\lambda_1 = \lambda_2$.
However, this is obviously not a polynomial in $Z_1$, $Z_2$, $N(Z_1)^{-1}$,
$N(Z_2)^{-1}$ and hence not an element of ${\cal H} \otimes {\cal H}$.
Note that composing $(I_R N(W)^{-1})(Z_1,Z_2)$ with the multiplication map $M$
amounts to setting $Z_1=Z_2=Z$ and letting $\lambda_1,\lambda_2 \to 1$,
but then the limit is infinite!
To get around this problem, observe that (\ref{I_R}) remains valid if we let
$Z_1$ and $Z_2$ approach two different points in $U(2)_R$ so that
$Z_1 \in \BB D^+_R$ and $Z_2 \in \BB D^-_R$. Thus we have a well defined operator
$$
f \quad \mapsto \quad (I_R^{+-} f)(Z_1,Z_2) = \frac i{2\pi^3}
\lim_{\genfrac{}{}{0pt}{}{Z'_1 \to Z_1,\: Z'_1 \in \BB D^+_R}{Z'_2 \to Z_2,\: Z'_2 \in \BB D^-_R}}
\int_{W \in U(2)_R} \frac{f(W) \,dV}{N(W-Z'_1) \cdot N(W-Z'_2)},
$$
where $Z_1, Z_2 \in U(2)_R$ and none of the eigenvalues of $Z_1 \cdot Z_2^{-1}$
is 1, i.e. $N(Z_1-Z_2)\ne 0$. Similarly, we can switch the roles of $Z_1$ and
$Z_2$ and define another operator
$$
f \quad \mapsto \quad (I_R^{-+} f)(Z_1,Z_2) = \frac i{2\pi^3}
\lim_{\genfrac{}{}{0pt}{}{Z'_1 \to Z_1,\: Z'_1 \in \BB D^-_R}{Z'_2 \to Z_2,\: Z'_2 \in \BB D^+_R}}
\int_{W \in U(2)_R} \frac{f(W) \,dV}{N(W-Z'_1) \cdot N(W-Z'_2)},
$$
where $Z_1, Z_2 \in U(2)_R$ and $N(Z_1-Z_2)\ne 0$.

It follows from Lemma \ref{equivariance} that the operators
$I_R^{+-}$ and $I_R^{-+}$ are $U(2,2)_R$-equivariant.
We already know that these operators annihilate $N(Z)^k$ for $k \ne -1$.
Hence they annihilate the entire $\Zh^- \oplus \Zh^+$.
Next we compute the limit
$$
\lim_{Z_1, Z_2 \to Z} \bigl( (I_R^{+-}+ I_R^{-+}) N(W)^{-1} \bigr)(Z_1,Z_2),
\qquad Z \in U(2)_R.
$$
As before, suppose that $Z_1 \cdot Z_2^{-1}$ has eigenvalues
$\lambda_1$ and $\lambda_2$ with $|\lambda_1|=|\lambda_2|=1$,
$\lambda_1 \ne 1$ and $\lambda_2 \ne 1$.
Then $N(Z_1)=\lambda_1\lambda_2 \cdot N(Z_2)$, $Z_2 \cdot Z_1^{-1}$
has eigenvalues $\lambda_1^{-1}$ and $\lambda_2^{-1}$.
Assume for a moment that $\lambda_1 \ne \lambda_2$,
then by (\ref{I_R}) we have:
\begin{multline*}
\bigl( (I_R^{+-}+ I_R^{-+}) N(W)^{-1} \bigr)(Z_1,Z_2)
= \frac{N(Z_2)^{-1}}{\lambda_2-\lambda_1}
\log \biggl( \frac{1-\lambda_1}{1-\lambda_2} \biggr)
+ \frac{N(Z_1)^{-1}}{\lambda_2^{-1}-\lambda_1^{-1}}
\log \biggl( \frac{1-\lambda_1^{-1}}{1-\lambda_2^{-1}} \biggr)  \\
= \frac{N(Z_2)^{-1}}{\lambda_2-\lambda_1}
\log \biggl( \frac{1-\lambda_1}{1-\lambda_2} \biggr)
- \frac{N(Z_2)^{-1}}{\lambda_2-\lambda_1}
\log \biggl( \frac{\lambda_2(\lambda_1-1)}{\lambda_1(\lambda_2-1)} \biggr)
=
-\frac1{N(Z_2)} \frac{\log\lambda_2-\log\lambda_1}{\lambda_2-\lambda_1}.
\end{multline*}
Hence,
$$
\bigl( (I_R^{+-}+ I_R^{-+}) N(W)^{-1} \bigr)(Z_1,Z_2)
= -\frac1{N(Z_2)} \cdot
\begin{cases} \frac{\log\lambda_2-\log\lambda_1}{\lambda_2-\lambda_1} &
\text{if $\lambda_1 \ne \lambda_2$;} \\
\lambda^{-1} & \text{if $\lambda_1 = \lambda_2 = \lambda$.} \end{cases}
$$
Therefore,
$$
\lim_{\genfrac{}{}{0pt}{}{Z_1, Z_2 \to Z}{N(Z_1-Z_2) \ne 0}}
\bigl( (I_R^{+-}+ I_R^{-+}) N(W)^{-1} \bigr)(Z_1,Z_2)
= - N(Z)^{-1}, \qquad Z \in U(2)_R.
$$
From the $U(2,2)_R$-equivariance we see that we have obtained a projector
onto $\Zh^0$:

\begin{thm}  \label{Zh^0-projector}
The $\mathfrak{gl}(2,\HC)$-equivariant map
$$
f \mapsto \bigl((I_R^{+-}+ I_R^{-+})f\bigr)(Z_1,Z_2)
\quad \in \B{{\cal H} \otimes {\cal H}},
\qquad f \in \Zh, \quad Z_1,Z_2 \in U(2)_R,
$$
is well-defined, annihilates $\Zh^- \oplus \Zh^+$ and satisfies
$$
M \circ \bigl((I_R^{+-}+ I_R^{-+})f\bigr) = f \quad \text{if $f \in \Zh^0$}.
$$

In particular, an operator $\P^0$ on $\Zh$
$$
f \mapsto (\P^0 f)(Z) =
- \lim_{\genfrac{}{}{0pt}{}{Z_1, Z_2 \to Z}{N(Z_1-Z_2) \ne 0}}
\bigl( (I_R^{+-}+ I_R^{-+})f \bigr)(Z_1,Z_2), \qquad Z \in U(2)_R,
$$
is well-defined, annihilates $\Zh^- \oplus \Zh^+$ and is
the identity mapping  on $\Zh^0$.

Finally, the operator $\P^0$ on $\Zh$ can be computed as follows:
\begin{multline*}
(\P^0 f)(Z) = \frac1{2\pi^3i} \lim_{\theta \to 0} \lim_{s \to 1} \biggl(
\int_{W \in U(2)_R} \frac{f(W)\,dV}{N(W-se^{i\theta}Z) \cdot N(W-s^{-1}e^{-i\theta}Z)}\\
+ \int_{W \in U(2)_R} \frac{f(W)\,dV}{N(W-s^{-1}e^{i\theta}Z) \cdot N(W-se^{-i\theta}Z)}
\biggr), \qquad Z \in U(2)_R.
\end{multline*}
\end{thm}

Note that the space $\Zh$ consists of rational functions, and rational
functions on $\HC$ as well as analytic ones are completely determined by
their values on $U(2)_R$.
Note also that this integral formula for $\P^0$ is in complete agreement
with our previous formal computation (\ref{expansion-formal})
of the reproducing kernel for $(\rho_1,\Zh^0)$.

\begin{rem}  \label{reproducing-remark}
Every function $f \in \Zh$ can be written as
$f = \P^-f + \P^0f + \P^+f$. Combining the integral expressions for
$\P^{\pm}f$ and $\P^0f$ obtained in Corollary \ref{proj-cor} and
Theorem \ref{Zh^0-projector} we get a reproducing formula for all functions
in $\Zh$ that is equivalent to (\ref{reproducing-formula}).
\end{rem}

\section{The One-Loop Feynman Integral and Its Relation to \\
$(\pi^0_l, {\cal H}^{\pm}) \otimes (\pi^0_r, {\cal H}^{\pm})$}
\label{integral-section}

In this section we show that the identification of the one-loop
Feynman diagram with the integral kernel $p^0_1(Z_1,Z_2;W_1,W_2)$
of the integral operators expressing ${\cal P}^+$ and ${\cal P}^-$
found in \cite{FL1} is an immediate consequence of Theorem \ref{embedding}.
These operators ${\cal P}^+$ and ${\cal P}^-$ are the
$\mathfrak{gl}(2,\HC)$-equivariant composition maps
\begin{equation}  \label{P}
{\cal P}^+: {\cal H}^+ \otimes {\cal H}^+ \twoheadrightarrow \Zh^+
\hookrightarrow {\cal H}^+ \otimes {\cal H}^+ \qquad \text{and} \qquad
{\cal P}^-: {\cal H}^- \otimes {\cal H}^- \twoheadrightarrow \Zh^-
\hookrightarrow {\cal H}^- \otimes {\cal H}^-
\end{equation}
(the multiplication map followed by the embedding).
As we mentioned earlier, the multiplicities of $(\rho_1, \Zh^+)$ in 
$(\pi^0_l, {\cal H}^+) \otimes (\pi^0_r, {\cal H}^+)$ and of $(\rho_1, \Zh^-)$ in 
$(\pi^0_l, {\cal H}^-) \otimes (\pi^0_r, {\cal H}^-)$ are both one.
So the maps ${\cal P}^+$ and ${\cal P}^-$ are unique up to multiplication by
scalars and they are pinned down by imposing
$$
{\cal P}^+ (1 \otimes 1) = 1 \otimes 1
\qquad \text{and} \qquad
{\cal P}^- \bigl( N(Z_1)^{-1} \otimes N(Z_2)^{-1} \bigr)
= N(W_1)^{-1} \otimes N(W_2)^{-1}.
$$

For convenience we restate Theorem 34 and Corollary 39 from \cite{FL1}.
We define operators on $\cal H$ by
\begin{align*}
\bigl( \SP_R^+ \phi \bigr)(Z) &=
\frac1{2\pi^2} \int_{X \in S^3_R}
\frac {(\degt \phi)(X)}{N(X-Z)} \cdot \frac {dS}R,
\qquad Z \in \BB D_R^+, \\
\bigl( \SP_R^- \phi \bigr)(Z) &=
\frac1{2\pi^2} \int_{X \in S^3_R}
\frac {(\degt \phi)(X)}{N(X-Z)} \cdot \frac {dS}R,
\qquad Z \in \BB D_R^-.
\end{align*}

\begin{thm}  \label{Poisson}
The operators $\SP_R^-$ and $\SP_R^+$ are continuous linear operators
$\cal H \to \cal H$.
The operator $\SP_R^+$ has image in ${\cal H}^+$ and sends
$$
\begin{matrix}
t^l_{n\,\underline{m}}(X) \quad \mapsto \quad t^l_{n\,\underline{m}}(Z), \\
\quad \\ t^l_{n\,\underline{m}}(X) \cdot N(X)^{-2l-1} \quad \mapsto \quad
- R^{-2(2l+1)} \cdot t^l_{n\,\underline{m}}(Z),
\end{matrix}
\qquad
\begin{matrix}
l = 0, \frac12, 1, \frac32, \dots, \\ m,n \in \BB Z+l,\\  -l \le m, n \le l.
\end{matrix}
$$
The operator $\SP_R^-$  has image in ${\cal H}^-$ and sends
$$
\begin{matrix}
t^l_{n\,\underline{m}}(X) \quad \mapsto \quad
R^{2(2l+1)} \cdot N(Z)^{-2l-1} \cdot t^l_{n\,\underline{m}}(Z),\\
\quad \\ t^l_{n\,\underline{m}}(X) \cdot N(X)^{-2l-1}\quad \mapsto \quad
- t^l_{n\,\underline{m}}(Z) \cdot N(Z)^{-2l-1},
\end{matrix}
\qquad
\begin{matrix}
l = 0, \frac12, 1, \frac32, \dots, \\ m,n \in \BB Z+l,\\  -l \le m, n \le l.
\end{matrix}
$$
\end{thm}

Now, let us take a close look at the function of three variables
$$
\frac1{N(W-Z_1) \cdot N(W-Z_2)}.
$$
On the one hand, this function has appeared in (\ref{fork}) and is
responsible for $\mathfrak{gl}(2,\HC)$-equivariant embeddings of
$\Zh^\pm$ into ${\cal H}^{\pm} \otimes {\cal H}^{\pm}$.
On the other hand, as can be seen from Theorem \ref{Poisson}, this function
can be used to express the multiplication maps (\ref{M}):

\begin{lem}  \label{mult-integral}
Fix $R_1, R_2 >0$ and consider a map $\widetilde{M}$ on
${\cal H} \otimes {\cal H}$ sending pure tensors
$$
\phi_1(Z_1) \otimes \phi_2(Z_2) \quad \mapsto \quad
\frac1{(2\pi^2)^2} \iint_{\genfrac{}{}{0pt}{}{Z_1 \in S^3_{R_1}}{Z_2 \in S^3_{R_2}}}
\frac {(\degt \phi_1)(Z_1) \cdot (\degt \phi_2)(Z_2)}{N(W-Z_1) \cdot N(W-Z_2)}
\cdot \frac{dS_1\,dS_2}{R_1R_2}.
$$
\begin{enumerate}
\item
If $\phi_1, \phi_2 \in {\cal H}^+$ and $W \in \BB D_{R_1}^+ \cap \BB D_{R_2}^+$,
then $\widetilde{M}$ is the multiplication map:
$$
\widetilde{M} \bigl( \phi_1(Z_1) \otimes \phi_2(Z_2) \bigr)
= (\phi_1\cdot\phi_2)(W);
$$
\item
If $\phi_1, \phi_2 \in {\cal H}^-$ and $W \in \BB D_{R_1}^- \cap \BB D_{R_2}^-$,
then $\widetilde{M}$ is the multiplication map:
$$
\widetilde{M} \bigl( \phi_1(Z_1) \otimes \phi_2(Z_2) \bigr)
= (\phi_1\cdot\phi_2)(W);
$$
\item
If $\phi_1 \in {\cal H}^+$, $\phi_2 \in {\cal H}^-$ and
$W \in \BB D_{R_1}^+ \cap \BB D_{R_2}^-$, then $\widetilde{M}$
is the negative of the multiplication map:
$$
\widetilde{M} \bigl( \phi_1(Z_1) \otimes \phi_2(Z_2) \bigr)
= - (\phi_1\cdot\phi_2)(W).
$$
\end{enumerate}
\end{lem}

Combining Theorem \ref{embedding} and Lemma \ref{mult-integral} we see that
the function
$$
p^0_1(Z_1,Z_2;W_1,W_2) =
\frac i{2\pi^3} \int_{T \in U(2)}
\frac{dV}{N(Z_1-T) \cdot N(Z_2-T) \cdot N(W_1-T) \cdot N(W_2-T)}
$$
can be interpreted as the integral kernel of the integral operators
expressing the $\mathfrak{gl}(2,\HC)$-equivariant compositions (\ref{P}).
Explicitly, we have:
\begin{multline*}
({\cal P}^+ (\phi_1 \otimes \phi_2))(W_1,W_2) \\
= \frac1{(2\pi^2)^2} \iint_{\genfrac{}{}{0pt}{}{Z_1 \in S^3_{R_1}}{Z_2 \in S^3_{R_2}}}
p^0_1(Z_1,Z_2;W_1,W_2) \cdot (\degt_{Z_1} \phi_1)(Z_1) \cdot
(\degt_{Z_2} \phi_2)(Z_2) \, \frac{dS_1\,dS_2}{R_1R_2},
\end{multline*}
where $\phi_1, \phi_2 \in {\cal H}^+$, $W_1, W_2 \in \BB D^+_1$ and $R_1, R_2>1$.
Similarly,
\begin{multline*}
({\cal P}^- (\phi_1 \otimes \phi_2))(W_1,W_2) \\
= \frac1{(2\pi^2)^2} \iint_{\genfrac{}{}{0pt}{}{Z_1 \in S^3_{R_1}}{Z_2 \in S^3_{R_2}}}
p^0_1(Z_1,Z_2;W_1,W_2) \cdot (\degt_{Z_1} \phi_1)(Z_1) \cdot
(\degt_{Z_2} \phi_2)(Z_2) \, \frac{dS_1\,dS_2}{R_1R_2},
\end{multline*}
where $\phi_1, \phi_2 \in {\cal H}^-$, $W_1, W_2 \in \BB D^-_1$ and
$0< R_1, R_2<1$.

We conclude this section with a comment that the integral kernel
$p^0_1(Z_1,Z_2;W_1,W_2)$ can be rewritten as an integral over $\BB R^4$ instead
of $U(2)$, as was done after Corollary 90 in \cite{FL1}.
Thus $p^0_1(Z_1,Z_2;W_1,W_2)$ gets identified with the integral represented
by the one-loop Feynman diagram (see Figure \ref{diagram}).

\begin{figure}
\center
\includegraphics[scale=1]{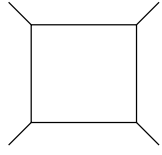}
\caption{One-loop Feynman diagram}
\label{diagram}
\end{figure}

\section{Minkowski Space Realization of $\Zh^-$, $\Zh^0$ and $\Zh^+$} \label{M-section}

In this section we realize the spaces $\Zh^-$, $\Zh^0$ and $\Zh^+$ in the
setting of the Minkowski space $\BB M$.
As in \cite{FL1}, we use $e_0$, $e_1$, $e_2$, $e_3$ in place of the more
familiar generators $1,i,j,k$ of $\BB H$, so that the symbol $i$ can be used
for $\sqrt{-1} \in \BB C$; and let $\tilde e_0 = -ie_0 \in \HC$. Then
$$
\BB M = \tilde e_0 \BB R \oplus e_1 \BB R \oplus e_2 \BB R \oplus e_3 \BB R
= \biggl\{
Z= \begin{pmatrix} z_{11} & z_{12} \\ z_{21} & z_{22} \end{pmatrix} \in \HC
; \: z_{11}, z_{22} \in i \BB R, \: z_{21} = -\overline{z_{12}} \biggr\}.
$$
Recall the generalized upper and lower half-planes introduced in
Section 3.5 in \cite{FL1}:
\begin{align*}
\BB T^- &= \{ Z = W_1 + i W_2 \in \HC ;\: W_1, W_2 \in \BB M ,\:
\text{$iW_2$ is positive definite} \},  \\
\BB T^+ &= \{ Z = W_1 + i W_2 \in \HC ;\: W_1, W_2 \in \BB M ,\:
\text{$iW_2$ is negative definite} \}
\end{align*}
and element $\gamma \in GL(2,\BB C)$ from Lemmas 54 and 63 of \cite{FL1}
which induces a fractional linear transformation on $\HC$ that we call
the ``Cayley transform''. Thus
$\gamma = \frac1{\sqrt 2} \bigl(\begin{smallmatrix} i & 1 \\ i & -1
\end{smallmatrix} \bigr) \in GL(2,\HC)$ with
$\gamma^{-1} = \frac1{\sqrt 2} \bigl(\begin{smallmatrix} -i & -i \\
1 & -1 \end{smallmatrix}\bigr)$.
The fractional linear map on $\HC$
$$
\pi_l(\gamma): \: Z \mapsto (Z-i)(Z+i)^{-1}
$$
maps $\BB D^+ \to \BB T^+$, $\BB D^- \to \BB T^-$, $U(2) \to \BB M$
(with singularities) and sends the sphere
$\{ Z \in U(2); N(Z)=1 \} = SU(2)$ into 
the two-sheeted hyperboloid $\{ Y \in \BB M;\: N(Y) = -1 \}$.
Conversely, the fractional linear map on $\HC$
$$
\pi_l(\gamma^{-1}): \: Z \mapsto -i(Z+1)(Z-1)^{-1}
$$
maps $\BB T^+ \to \BB D^+$, $\BB T^- \to \BB D^-$, $\BB M \to U(2)$,
has no singularities on $\BB M$, and sends the two-sheeted hyperboloid
$\{ Y \in \BB M;\: N(Y) = -1 \}$ into the sphere
$\{ Z \in U(2); N(Z)=1 \} = SU(2)$.

These fractional linear transformations induce the following maps on functions:
$$
\pi^0_l(\gamma): \: \phi(Z) \mapsto
\bigl( \pi^0_l(\gamma)\phi \bigr)(Z) = \frac2{N(Z-1)} \cdot
\phi \bigl( -i(Z+1)(Z-1)^{-1} \bigr),
$$
sends harmonic functions\footnote{
By harmonic functions on $U(2)$ we mean functions that are holomorphic and
harmonic in some open neighborhood of $U(2)$.}
on $\BB D^+$, $\BB D^-$ and $U(2)$ into solutions of the wave equation on,
respectively, $\BB T^+$, $\BB T^-$ and $\BB M$. Similarly,
$$
\pi^0_l(\gamma^{-1}): \: \phi(Z) \mapsto
\bigl( \pi^0_l(\gamma^{-1})\phi \bigr)(Z) = \frac{-2}{N(Z+i)} \cdot
\phi \bigl( (Z-i)(Z+i)^{-1} \bigr),
$$
sends solutions of the wave equation on $\BB T^+$, $\BB T^-$ and $\BB M$ into
harmonic functions on, respectively, $\BB D^+$, $\BB D^-$ and $U(2)$.
In particular, $\pi^0_l(\gamma)$ maps
\begin{equation}  \label{M-gen}
1 \mapsto 2 \cdot N(Z-1)^{-1}, \qquad N(Z)^{-1} \mapsto -2 \cdot N(Z+1)^{-1}.
\end{equation}


The light cone
$$
\cone = \{Y \in \BB M;\: N(Y)=0 \}
$$
can be divided into two parts:
$$
\cone^+ =\{ Y \in \cone ;\: i\tr Y \ge 0 \}
\qquad \text{and} \qquad
\cone^- =\{ Y \in \cone ;\: i\tr Y \le 0 \}.
$$
Next we calculate the (inverse) Fourier transform of the delta distributions
on $\cone^+$ and $\cone^-$:

\begin{lem}
We have the following absolutely convergent expansions:
\begin{align*}
\frac 1{N(Z)} &= \frac{1}{4\pi} \int_{P \in \cone^-}
e^{i\langle Z, P \rangle} \,\frac{dp^1dp^2dp^3}{|p^0|}, \qquad Z \in \BB T^-, \\
\frac 1{N(Z)} &= \frac{1}{4\pi} \int_{P\in \cone^+}
e^{i\langle Z, P \rangle} \,\frac{dp^1dp^2dp^3}{|p^0|}, \qquad Z \in \BB T^+,
\end{align*}
where $\langle Z,P \rangle = \tr (Z^+P)/2 = \tr (P^+Z)/2$ and
$P = p^0 \tilde e_0 + p^1e_1 + p^2e_2 + p^3e_3 \in \cone^{\pm} \subset \BB M$.
\end{lem}

\begin{proof}
Note that $iW \in i \BB M$ is positive definite if and only if $N(W)<0$
and $i\tr W >0$ or, equivalently, if and only if
$W=w^0 \tilde e_0 + w^1 e_1 + w^2 e_2 + w^3 e_3 \in \BB M$ and
$w^0 > |\overrightarrow{w}| = \sqrt{(w^1)^2 + (w^2)^2 + (w^3)^2}$.
Similarly, $iW \in i \BB M$ is negative definite if and only if
$W=w^0 \tilde e_0 + w^1 e_1 + w^2 e_2 + w^3 e_3 \in \BB M$ and
$w^0 < |\overrightarrow{w}|$.
This implies that the two integrals converge absolutely
on the respective regions.

Each integral defines a complex analytic function of $Z$
on $\BB T^-$ and $\BB T^+$ respectively.
Hence to establish that the integrals are equal to $N(Z)^{-1}$,
it is sufficient to prove that for $Z$ of the form
$z^0 \tilde e_0 + z^1 e_1 + z^2 e_2 + z^3 e_3 \in \HC$ with
$z^0 \in \BB C^{\times}$ and $\overrightarrow{z} = (z^1, z^2, z^3) \in \BB R^3$.
Rotating $Z$ if necessary, without loss of generality we can assume that
$z^2=z^3=0$.
For $\overrightarrow{p} = (p^1, p^2 , p^3) \in \BB R^3$, let
\begin{align*}
P_+= |\overrightarrow{p}| \tilde e_0 + p^1e_1 + p^2e_2 + p^3e_3 \quad
&\in \cone^+,
\qquad |\overrightarrow{p}| = \sqrt{(p^1)^2+(p^2)^2+(p^3)^2}, \\
P_-= -|\overrightarrow{p}| \tilde e_0 + p^1e_1 + p^2e_2 + p^3e_3 \quad
&\in \cone^-.
\end{align*}
Let $s = \sqrt{(p^2)^2 + (p^3)^2}$ and substitute $u = \sqrt{(p^1)^2+s^2}$,
then the integrals in question become
\begin{multline*}
\int_{P \in \cone^{\mp}} e^{i\langle Z, P \rangle} \,\frac{dp^1dp^2dp^3}{|p^0|}
= \int_{\overrightarrow{p} \in \BB R^3}
\exp\bigl(i(\pm z^0 |\overrightarrow{p}| + z^1p^1)\bigr)
\,\frac{dp^1dp^2dp^3}{|\overrightarrow{p}|}  \\
=
2\pi \iint_{\genfrac{}{}{0pt}{}{s \ge 0}{-\infty < p^1 < \infty}}
s \exp\bigl(i(\pm z^0 \sqrt{(p^1)^2 + s^2} + z^1p^1)\bigr)
\,\frac {dp^1ds}{\sqrt{(p^1)^2 + s^2}}  \\
=
2\pi \int_{-\infty}^{\infty} \biggl( \int_{u \ge |p^1|}
e^{i(\pm z^0u + z^1p^1)} \,du \biggr)dp^1  \\
=
\pm \frac {2\pi i}{z^0} \int_{-\infty}^{\infty} e^{i(\pm z^0|p^1| + z^1p^1)} \,dp^1
= \frac {4\pi}{(z^1)^2-(z^0)^2} = \frac{4\pi}{N(Z)}.
\end{multline*}
\end{proof}

Therefore,
\begin{align*}
\frac1{N(Z_1-Z_2)} &= \frac{1}{4\pi} \int_{P \in \cone^-}
e^{i\langle Z_1-Z_2, P \rangle} \,\frac{dp^1dp^2dp^3}{|p^0|},
\qquad \text{whenever $Z_1-Z_2 \in \BB T^-$,}  \\
\text{and} \qquad
\frac1{N(Z_1-Z_2)} &= \frac{1}{4\pi} \int_{P \in \cone^+}
e^{i\langle Z_1-Z_2, P \rangle} \,\frac{dp^1dp^2dp^3}{|p^0|},
\qquad \text{whenever $Z_1-Z_2 \in \BB T^+$.}  \\
\end{align*}

\begin{cor}
Up to proportionality coefficients, the Fourier transforms of the following
distributions on $\BB M$ are:
\begin{align*}
\text{the FT of} \quad \frac1{N(Y-1)} \quad
&\text{is the distribution} \quad
f \mapsto \int_{P \in \cone^+} f(P) \cdot e^{-i\tr P/2} \,\frac{dp^1dp^2dp^3}{|p^0|},\\
\text{the FT of} \quad \frac1{N(Y+1)} \quad
&\text{is the distribution} \quad
f \mapsto \int_{P \in \cone^-} f(P) \cdot e^{i\tr P/2} \,\frac{dp^1dp^2dp^3}{|p^0|}.
\end{align*}
\end{cor}

(The presence of the rapidly decaying term $e^{\pm i\tr P/2}$ ensures
convergence of the integrals.)

Combining this corollary with (\ref{M-gen}), we see that the Fourier transform
maps $\pi^0_l(\gamma)({\cal H}^+)$ into distributions supported on $\cone^+$
and $\pi^0_l(\gamma)({\cal H}^-)$ into distributions supported on $\cone^-$.
Since the Fourier transform maps products of functions into convolutions,
by Lemma \ref{image-lemma} the Fourier transform maps
$\rho_1(\Zh^+)$, $\rho_1(\Zh^-)$ and $\rho_1(\Zh^0)$
into distributions supported respectively in
$\{ Y \in \BB M ;\: N(Y)<0 ,\: i\tr Y \ge 0 \}$
-- the ``interior of $\cone^+$'',
$\{ Y \in \BB M ;\: N(Y)<0 ,\: i\tr Y \le 0 \}$
-- the ``interior of $\cone^-$'' and
$\{ Y \in \BB M ;\: N(Y)>0 \}$ -- the ``exterior of $\cone$''.

Next we set $R=1$ and pull back the maps $I_1$ defined by (\ref{fork}) via
$\pi_l(\gamma^{-1})$. Using Lemmas \ref{Z-W} and \ref{Jacobian_lemma}
we obtain a formula that formally looks like (\ref{fork}):
$$
\Zh \ni f \quad \mapsto \quad (I_1 f)(Z_1,Z_2) =
\frac i{2\pi^3} \int_{Y \in \BB M} \frac{f(Y) \,dV}{N(Y-Z_1) \cdot N(Y-Z_2)}
\quad \in \B{{\cal H} \otimes {\cal H}};
$$
however, the integration is over $Y \in \BB M$, the two copies of ${\cal H}$
are realized as solutions of the wave equation on $\BB M$ and
$Z_1, Z_2 \in \BB T^- \sqcup \BB T^+$.
Setting $Z_1=Z_2 \in \BB T^-$ and $Z_1=Z_2 \in \BB T^+$ results in projectors
of $\rho_1(\Zh)$ onto $\rho_1(\Zh^-)$ and $\rho_1(\Zh^+)$ respectively.

\begin{thm}
Let $f \in \rho_1(\Zh)$.
If $Z \in \BB T^+$, then the map
$$
f \mapsto \P^+_{\BB M}(Z) =
\frac i{2\pi^3} \int_{Y \in \BB M} \frac {f(Y) \,dV}{N(Y-Z)^2}
$$
is a projector onto $\rho_1(\Zh^+)$
and, in particular, provides a reproducing formula for functions in
$\rho_1(\Zh^+)$.
Similarly, if $Z \in \BB T^-$, then the map
$$
f \mapsto \P^-_{\BB M}(Z) =
\frac i{2\pi^3} \int_{Y \in \BB M} \frac {f(Y) \,dV}{N(Y-Z)^2}
$$
is a projector onto $\rho_1(\Zh^-)$
and, in particular, provides a reproducing formula for functions in
$\rho_1(\Zh^-)$.
\end{thm}

(The reproducing formulas for $\rho_1(\Zh^+)$ and $\rho_1(\Zh^-)$
were obtained in \cite{FL1}, Theorem 74.)

Next we introduce operators
$$
f \quad \mapsto \quad (I_{\BB M}^{+-} f)(Z_1,Z_2) = \frac i{2\pi^3}
\lim_{\genfrac{}{}{0pt}{}{Z'_1 \to Z_1,\: Z'_1 \in \BB T^+}{Z'_2 \to Z_2,\: Z'_2 \in \BB T^-}}
\int_{Y \in \BB M} \frac{f(Y) \,dV}{N(Y-Z'_1) \cdot N(Y-Z'_2)},
$$
where $Z_1, Z_2 \in \BB M$ and $N(Z_1-Z_2)\ne 0$.
Similarly, we can switch the roles of $Z_1$ and
$Z_2$ and define another operator
$$
f \quad \mapsto \quad (I_{\BB M}^{-+} f)(Z_1,Z_2) = \frac i{2\pi^3}
\lim_{\genfrac{}{}{0pt}{}{Z'_1 \to Z_1,\: Z'_1 \in \BB T^-}{Z'_2 \to Z_2,\: Z'_2 \in \BB T^+}}
\int_{Y \in \BB M} \frac{f(Y) \,dV}{N(Y-Z'_1) \cdot N(Y-Z'_2)},
$$
where $Z_1, Z_2 \in \BB M$ and $N(Z_1-Z_2)\ne 0$.
From Theorem \ref{Zh^0-projector} we obtain the following result:

\begin{thm}
The $\mathfrak{gl}(2,\HC)$-equivariant map
$$
f \mapsto \bigl((I_{\BB M}^{+-}+ I_{\BB M}^{-+})f\bigr)(Z_1,Z_2)
\quad \in \B{{\cal H} \otimes {\cal H}}, \qquad
f \in \rho_1(\Zh), \quad Z_1,Z_2 \in \BB M,
$$
is well-defined, annihilates $\rho_1(\Zh^-) \oplus \rho_1(\Zh^+)$
and satisfies
$$
M \circ \bigl((I_{\BB M}^{+-}+ I_{\BB M}^{-+})f\bigr) = f
\qquad \text{if $f \in \rho_1(\Zh^0)$}.
$$

In particular, an operator $\P^0$ on $\rho_1(\Zh)$
$$
f \mapsto (\P^0 f)(Z) =
- \lim_{\genfrac{}{}{0pt}{}{Z_1, Z_2 \to Z}{N(Z_1-Z_2) \ne 0}}
\bigl( (I_{\BB M}^{+-}+ I_{\BB M}^{-+})f \bigr)(Z_1,Z_2), \qquad Z \in \BB M,
$$
is well-defined, annihilates $\rho_1(\Zh^-) \oplus \rho_1(\Zh^+)$
and is the identity mapping on $\rho_1(\Zh^0)$.

Finally, the operator $\P^0$ on $\rho_1(\Zh)$
can be computed as follows:
\begin{multline*}
(\P^0 f)(Z) = \frac1{2\pi^3i} \lim_{t \to 0} \lim_{s \to 0} \biggl(
\int_{Y\in \BB M} \frac{f(Y)\,dV}{N(Y-Z+it+s) \cdot N(Y-Z-it-s)}\\
+ \int_{Y \in \BB M} \frac{f(Y)\,dV}{N(Y-Z+it-s) \cdot N(Y-Z-it+s)}
\biggr), \qquad Z \in \BB M.
\end{multline*}
\end{thm}

\section{Anti de Sitter Space}  \label{AdS-section}

We consider a 5-dimensional space $\BB R^{1,4}$ with coordinates
$(w^0,w^1,w^2,w^3,w^4)$ and metric coming from an indefinite inner product
$$
\langle W, W' \rangle_{1,4} = w^0w'^0-w^1w'^1-w^2w'^2-w^3w'^3-w^4w'^4.
$$
Corresponding to this metric, we have a wave operator
$$
\square_{1,4} = \frac{\partial^2}{(\partial w^0)^2}
- \frac{\partial^2}{(\partial w^1)^2} - \frac{\partial^2}{(\partial w^2)^2}
- \frac{\partial^2}{(\partial w^3)^2} - \frac{\partial^2}{(\partial w^4)^2}.
$$
We introduce notations
$$
\BB R^{1,4}_+ = \Bigl\{ (w^0,w^1,w^2,w^3,w^4) \in \BB R^{1,4} ;\:
w^0 > \sqrt{(w^1)^2+(w^2)^2+(w^3)^2+(w^4)^2} \Bigr\},
$$
$$
\|W\|_{1,4}= \sqrt{ (w^0)^2-(w^1)^2-(w^2)^2-(w^3)^2-(w^4)^2},
\qquad W \in \BB R^{1,4}_+.
$$
We fix a parameter $\m>0$ and introduce new coordinates
$(\rho,v^1,v^2,v^3,v^4)$ on $\BB R^{1,4}_+$ as follows:
$$
\begin{cases}
\rho = \|W\|_{1,4}, \\
v^i = (\m\rho)^{-1} w^i, & i=1,2,3,4.
\end{cases}
$$
Then
$$
\begin{cases}
w^0 = \m\rho \bigl( \m^{-2} + (v^1)^2+(v^2)^2+(v^3)^2+(v^4)^2 \bigr)^{1/2}, \\
w^i = \m\rho v^i, & i=1,2,3,4.
\end{cases}
$$
For each $\rho>0$, let us denote by $H_{\rho}$ the single sheet of
a two-sheeted hyperboloid
\begin{equation}  \label{H_rho}
H_{\rho} = \bigl\{ W \in \BB R^{1,4};\:
(w^0)^2-(w^1)^2-(w^2)^2-(w^3)^2-(w^4)^2=\rho^2,\: w^0>0 \bigr\}.
\end{equation}
Let us introduce differential operators
$$
\square = \frac{\partial^2}{(\partial v^1)^2}
+ \frac{\partial^2}{(\partial v^2)^2}
+ \frac{\partial^2}{(\partial v^3)^2} + \frac{\partial^2}{(\partial v^4)^2},
$$
$$
\deg = v^1\frac{\partial}{\partial v^1} + v^2\frac{\partial}{\partial v^2}
+ v^3\frac{\partial}{\partial v^3} + v^4\frac{\partial}{\partial v^4},
\qquad \degt f = \deg f + f.
$$
By direct computation we obtain:

\begin{lem}
We have
\begin{equation}  \label{deformed_laplacian}
\square_{1,4} = \frac{\partial^2}{\partial \rho^2}
+ \frac4{\rho} \frac{\partial}{\partial \rho} - \frac1{\m^2\rho^2} \square_{\m},
\end{equation}
where
$$
\square_{\m} = \square + \m^2(\deg^2 + 3\deg) 
= \square + \m^2 \bigl( \degt^2 + \degt -2 \bigr).
$$
\end{lem}

We think of
$\frac{\partial^2}{\partial \rho^2}
+ \frac4{\rho} \frac{\partial}{\partial \rho}$
as the ``radial'' part of the wave operator $\square_{1,4}$ and
$\square_{\m}$ as the part ``tangential'' to the hyperboloids $H_{\rho}$.
Notice that when $\m \to 0$, $\square_{\m}$ becomes the ordinary Laplacian.
We identify the space of quaternions $\BB H$ with one sheet of a two-sheeted
hyperboloid in $\BB R^{1,4}$ as follows:
$$
\BB H \ni \quad X = x^0 + ix^1 + jx^2 + kx^3
\quad \leftrightsquigarrow \quad
(\rho,v^1=x^0,v^2=x^1,v^3=x^2,v^4=x^3) \quad \in \BB R^{1,4}_+,
$$
where $\rho$ can be any fixed positive number.
We study functions on $\BB H$ that are annihilated by the conformal Laplacian
$$
\tlap = \square + \m^2( \deg^2 + 3\deg +2)
= \square + \m^2 \bigl( \degt^2 + \degt \bigr).
$$
The following lemma is verified by direct computation.


\begin{lem}  \label{fundamental_sol}
Let $X,Y \in \BB H$, with $Y$ fixed, and let
$$
\hat X = \bigl( \sqrt{\m^{-2}+N(X)},x^0,x^1,x^2,x^3 \bigr) \quad \text{and} \quad
\hat Y = \bigl( \sqrt{\m^{-2}+N(Y)},y^0,y^1,y^2,y^3 \bigr)
\quad \in \BB R^{1,4}_+,
$$
so that $\m \rho \hat X, \m \rho \hat Y \in H_{\rho}$. Then
$$
\degt \biggl( \frac1{\langle \hat X - \hat Y, \hat X - \hat Y \rangle_{1,4}}
\biggr) =
\biggl( 1 - \frac {\sqrt{\m^{-2}+N(Y)}}{\sqrt{\m^{-2}+N(X)}} \biggr)
\frac{2\m^{-2}}
{\bigl( \langle \hat X - \hat Y, \hat X - \hat Y \rangle_{1,4} \bigr)^2}
$$
and
$$
\tlap \biggl( \frac1{\langle \hat X - \hat Y, \hat X - \hat Y \rangle_{1,4}}
\biggr) =0.
$$
\end{lem}

We conclude this section with the following result.

\begin{lem}
Whenever $X, Y \in \BB H$, $X \ne Y$, we have
$\langle \hat X - \hat Y, \hat X - \hat Y \rangle_{1,4} <0$.
\end{lem}

\begin{proof}
We use two inequalities:
$$
\sqrt{\m^{-2}+N(X)} \sqrt{\m^{-2}+N(Y)} \ge \m^{-2} + \sqrt{N(X)N(Y)}
$$
and
$$
\sqrt{N(X)N(Y)} \ge x^0y^0+x^1y^1+x^2y^2+x^3y^3.
$$
The first inequality is strict unless $N(X)=N(Y)$; and the second inequality
is also strict unless $X$ and $Y$ are proportional with a non-negative
proportionality coefficient.
We have:
\begin{multline*}
\langle \hat X - \hat Y, \hat X - \hat Y \rangle_{1,4}
= \langle \hat X , \hat X \rangle_{1,4} + \langle \hat Y, \hat Y \rangle_{1,4}
-2 \langle \hat X, \hat Y \rangle_{1,4}  \\
= 2\m^{-2} - 2\sqrt{\m^{-2}+N(X)} \sqrt{\m^{-2}+N(Y)}
+ 2(x^0y^0+x^1y^1+x^2y^2+x^3y^3) \\
\le 2(x^0y^0+x^1y^1+x^2y^2+x^3y^3) - 2\sqrt{N(X)N(Y)} \le 0,
\end{multline*}
and if $X \ne Y$ at least one of the inequalities is strict.
\end{proof}

\section{Conformal Lie Algebra Action}  \label{conformal-section}

Let $SO^+(1,4)$ denote the connected component of the identity element in
$SO(1,4)$.
In this section we describe the action of $SO^+(1,4)$ and its Lie algebra
$\mathfrak{so}(1,4)$ on the space of solutions of $\tlap \phi =0$.
Then we extend the Lie algebra action to $\mathfrak{so}(1,5)$
(recall that the conformal Lie algebra in classical case is
$\mathfrak{sl}(2,\BB H) \simeq \mathfrak{so}(1,5)$).
Complexifying, we immediately obtain an action of
$\BB C \otimes \mathfrak{so}(1,5) \simeq \mathfrak{so}(6,\BB C)$.
The construction of the action of $\mathfrak{so}(1,5)$ will be very similar
to that of the indefinite orthogonal group $O(p,q)$ acting on the solutions
of the ultrahyperbolic wave equation in $\BB R^{p-1,q-1}$.
See \cite{KobO} (and references therein) for a description of this action of
$O(p,q)$ suitable for our purposes.

Fix a $\rho_0>0$ and recall that $H_{\rho_0}$ denotes
the single sheet of a two-sheeted hyperboloid (\ref{H_rho}).
The group $SO^+(1,4)$ acts linearly on $\BB R^{1,4}$ and preserves each
$H_{\rho_0}$. Hence it acts on functions on $H_{\rho_0}$ by
\begin{equation}  \label{pi-action}
\pi(a) : \quad f(W) \mapsto \bigl(\pi(a)f\bigr)(W) = f(a^{-1} \cdot W),
\qquad a \in SO^+(1,4).
\end{equation}

\begin{prop}  \label{SO(1,4)-action}
This action preserves the kernel of $\tlap$. That is, if $\phi$ is a
function on $H_{\rho_0}$ satisfying $\tlap\phi=0$ and $a \in SO^+(1,4)$,
then $\tlap \bigl(\pi(a)\phi\bigr)=0$.
\end{prop}

\begin{proof}
Let $a \in SO^+(1,4)$. The action of $a$ on $\BB R^{1,4}$
commutes with the wave operator $\square_{1,4}$.
Hence $\pi(a)$ commutes with the tangental part of the wave operator
$\square_{\m}$.
Therefore, $\pi(a)$ commutes with $\tlap = \square_{\m} + 2\m^2$.
\end{proof}

In order to extend the action to $\mathfrak{so}(1,5)$
we consider a 6-dimensional space $\BB R^{1,5}$ with
coordinates $(w^0,w^1,w^2,w^3,w^4,w^5)$ and indefinite inner product
$$
\langle W, W' \rangle_{1,5} = w^0w'^0-w^1w'^1-w^2w'^2-w^3w'^3-w^4w'^4-w^5w'^5.
$$
The group $SO(1,5)$ acts linearly on $\BB R^{1,5}$ preserving this inner product.
We introduce a function $\nu$ on $\BB R^{1,5}$:
$$
\nu(W)=w^5, \qquad W=(w^0,w^1,w^2,w^3,w^4,w^5) \in \BB R^{1,5}.
$$
We realize $SO(1,4)$ as the subgroup of $SO(1,5)$ fixing the last coordinate.
We can embed $\BB R^{1,4}$ into $\BB R^{1,5}$ as a hyperplane $w^5=const$
so that $SO(1,4)$ preserves it; and we choose to fix a particular embedding
$$
\BB R^{1,4} \in (w^0,w^1,w^2,w^3,w^4) \quad \leftrightsquigarrow \quad
(w^0,w^1,w^2,w^3,w^4,\rho_0) \in \BB R^{1,5}.
$$
This way the hyperboloid $H_{\rho_0}$ maps into the light cone in $\BB R^{1,5}$
$$
\operatorname{Cone}_{1,5} = \bigl\{ W \in \BB R^{1,5};\:
(w^0)^2-(w^1)^2-(w^2)^2-(w^3)^2-(w^4)^2-(w^5)^2=0 \bigr\},
$$
and this cone is obviously preserved by the $SO(1,5)$ action.
Let $\tilde H_{\rho_0}$ be the two-sheeted hyperboloid
$$
\tilde H_{\rho_0} = \{ W \in  \operatorname{Cone}_{1,5};\: w^5 = \rho_0 \}
\subset \operatorname{Cone}_{1,5} \subset \BB R^{1,5},
$$
then $H_{\rho_0}$ can be identified with $\{ W \in \tilde H_{\rho_0} ;\: w^0>0 \}$.
The group $SO(1,5)$ acts on $\tilde H_{\rho_0}$ by projective transformations:
$$
\pi(a): \quad W \mapsto \rho_0 \frac{a \cdot W}{\nu(a \cdot W)},
\qquad a \in SO(1,5).
$$
Of course, this action is defined only when $\nu(a \cdot W) \ne 0$.
Then we can extend this action to functions on $\tilde H_{\rho_0}$
by fixing a $\lambda \in \BB C$ and letting
$$
\varpi_{\lambda}(a): \quad f(W) \mapsto
\bigl(\varpi_{\lambda}(a)f\bigr)(W) = \rho_0^{-\lambda} \cdot
\bigl( \nu(a^{-1}\cdot W) \bigr)^{\lambda} \cdot f\bigl( \pi(a^{-1})W \bigr),
\qquad a \in SO(1,5).
$$
Finally, we set $\lambda=-1$ and let $\pi(a)=\varpi_{-1}(a)$:
$$
\pi(a): \quad f(W) \mapsto \bigl(\pi(a)f\bigr)(W) =
\frac{\rho_0}{\nu(a^{-1}\cdot W)} \cdot f\bigl( \pi(a^{-1})W \bigr),
\qquad a \in SO(1,5).
$$
This action extends previously defined action (\ref{pi-action}) of
$SO^+(1,4)$. Differentiating, we obtain an action of the
Lie algebra $\mathfrak{so}(1,5)$ on functions on $\tilde H_{\rho_0}$ and
$H_{\rho_0}$, which we still denote by $\pi$.
Complexifying, we immediately obtain an action of
$\BB C \otimes \mathfrak{so}(1,5) \simeq \mathfrak{so}(6,\BB C)$.

\begin{thm}
The $\pi$-action of the Lie algebra $\mathfrak{so}(6,\BB C)$ preserves the
kernel of $\tlap$. That is, if $\phi$ is a function on $H_{\rho_0}$ satisfying
$\tlap\phi=0$ and $h \in \mathfrak{so}(6,\BB C)$,
then $\tlap \bigl(\pi(h)\phi\bigr)=0$.
\end{thm}

\begin{proof}
It is sufficient to prove the result for
$h \in \mathfrak{so}(1,5) \subset \mathfrak{so}(6,\BB C)$ only.
For $h \in \mathfrak{so}(1,4) \subset \mathfrak{so}(1,5)$ the result is true
by Proposition \ref{SO(1,4)-action}.
As a Lie algebra, $ \mathfrak{so}(1,5)$ is generated by $\mathfrak{so}(1,4)$
and the Lie algebra of the one-parameter family of hyperbolic rotations in
the $(w^0w^5)$-plane:
$$
a_t: \quad w^0 \mapsto w^0 \cosh t + w^5 \sinh t, \quad
w^5 \mapsto w^5 \cosh t + w^0 \sinh t, \qquad t \in \BB R,
$$
$w^1$, $w^2$, $w^3$ and $w^4$ stay unchanged.
To compute $\frac{d}{dt} \bigr|_{t=0} \pi(a_t)$, we let $t \to 0$ and working
modulo terms of order $t^2$ we get
$$
\bigl(\pi(a_t)\phi\bigr)(W) = \frac{\rho_0}{\rho_0-tw^0} \cdot
\phi \biggl( \rho_0\frac{w^0-t\rho_0}{\rho_0-tw^0},
\frac{\rho_0 w^1}{\rho_0-tw^0}, \frac{\rho_0 w^2}{\rho_0-tw^0},
\frac{\rho_0 w^3}{\rho_0-tw^0}, \frac{\rho_0 w^4}{\rho_0-tw^0} \biggr).
$$
Rewriting it in $(\rho,v^1,v^2,v^3,v^4)$ coordinates, we obtain
$$
\bigl(\pi(a_t)\phi\bigr)(W) = \frac{\rho_0}{\rho_0-tw^0} \cdot
\phi \biggl( \rho_0, \frac{\rho_0 v^1}{\rho_0-tw^0},
\frac{\rho_0 v^2}{\rho_0-tw^0}, \frac{\rho_0 v^3}{\rho_0-tw^0},
\frac{\rho_0 v^4}{\rho_0-tw^0} \biggr).
$$
Hence
$$
\frac{d}{dt} \bigl(\pi(a_t)\phi\bigr) \Bigr|_{t=0}
= \frac{w^0}{\rho_0} \phi + \frac{w^0}{\rho_0} \deg \phi
= \frac{w^0}{\rho_0} \degt \phi
= \bigl(1+\m^2N(X)\bigr)^{1/2} \degt \phi,
$$
since $w^0= \rho_0 \bigl(1+\m^2N(X)\bigr)^{1/2}$, where
$N(X) = (v^1)^2+(v^2)^2+(v^3)^2+(v^4)^2$.
Finally, the theorem follows from the lemma below.
\end{proof}

\begin{lem}
The operator $\phi \mapsto \bigl(1+\m^2N(X)\bigr)^{1/2} \degt \phi$
preserves the kernel of $\tlap$.
\end{lem}

\begin{proof}
We need to show that if $\phi$ is a function on $H_{\rho_0}$ satisfying
$\tlap\phi=0$, then
$\tlap \bigl( \bigl(\m^{-2}+N(X)\bigr)^{1/2} \degt \phi \bigr)=0$.
For this purpose we compute the following commutators of operators:
\begin{multline*}
\Bigl[ \square, \bigl( \m^{-2}+N(X) \bigr)^{1/2} \Bigr]
= \square \bigl( \m^{-2} + N(X) \bigr)^{1/2} + 2 \sum_{i=1}^4
\left( \frac{\partial}{\partial v^i} \bigl( \m^{-2} + N(X) \bigr)^{1/2} \right)
\cdot \frac{\partial}{\partial v^i} \\
= \frac4{\bigl(\m^{-2}+N(X)\bigr)^{1/2}}-\frac{N(X)}{\bigl(\m^{-2}+N(X)\bigr)^{3/2}}
+ \frac2{\bigl(\m^{-2}+N(X)\bigr)^{1/2}} \deg,
\end{multline*}
$$
\Bigl[ \deg, \bigl( \m^{-2} + N(X) \bigr)^{1/2} \Bigr]
= \deg \bigl(\m^{-2}+N(X)\bigr)^{1/2} = N(X) \cdot \bigl(\m^{-2}+N(X)\bigr)^{-1/2},
$$
\begin{multline*}
\Bigl[ \deg^2, \bigl( \m^{-2} + N(X) \bigr)^{1/2}\Bigr]
=\deg \circ \Bigl[ \deg, \bigl( \m^{-2} + N(X) \bigr)^{1/2} \Bigr] +
\Bigl[ \deg, \bigl( \m^{-2} + N(X) \bigr)^{1/2} \Bigr] \circ \deg \\
= \frac{2N(X)}{\bigl(\m^{-2}+N(X)\bigr)^{1/2}}
- \frac{N(X)^2}{\bigl(\m^{-2}+N(X)\bigr)^{3/2}}
+ \frac{2N(X)}{\bigl(\m^{-2}+N(X)\bigr)^{1/2}} \deg,
\end{multline*}
\begin{multline*}
\Bigl[ \tlap,  \bigl( \m^{-2} + N(X) \bigr)^{1/2} \Bigr] = \\
\bigl(\m^{-2}+N(X)\bigr)^{-1/2} \biggl( 4-\frac{N(X)}{\m^{-2}+N(X)}+2\deg
+ \m^2 \Bigl( 5N(X) - \frac{N(X)^2}{\m^{-2}+N(X)} + 2N(X)\deg \Bigr)\biggr) \\
= 2\m^2 \bigl(\m^{-2}+N(X)\bigr)^{1/2} (\deg +2).
\end{multline*}
Finally, we get:
\begin{multline*}
\tlap \Bigl( \bigl(\m^{-2}+N(X)\bigr)^{1/2} \degt \phi \Bigr)
= \bigl(\m^{-2}+N(X)\bigr)^{1/2} \tlap \degt \phi
+ \Bigl[ \tlap, \bigl(\m^{-2}+N(X)\bigr)^{1/2} \Bigr] \degt \phi \\
= \bigl(\m^{-2}+N(X)\bigr)^{1/2} \Bigl( 2\square + 2\m^2 (\deg+2)\degt \Bigr)\phi
=2 \bigl(\m^{-2}+N(X)\bigr)^{1/2} \tlap \phi =0.
\end{multline*}
\end{proof}

\section{Extension of Harmonic Functions to $\BB R^{1,4}$}  \label{extension-section}

If we identify the group $SU(2)$ with the unit sphere in $\BB H$,
then functions on $SU(2)$ can be extended to
$\BB H^{\times} = \BB H \setminus \{0\}$ as harmonic functions.
If we require such an extension to be regular either at the origin or at
infinity, then it is unique. For example, let us consider
the matrix coefficients of $SU(2)$ given by (\ref{t}).
The restrictions $t^l_{m\,\underline{n}}(X) \bigr|_{SU(2)}$ can be extended from
$SU(2)$ to $\BB H^{\times}$ as $t^l_{m\,\underline{n}}(X)$ which are homogeneous
polynomials of degree $2l$ -- hence regular at the origin --
or as $N(X)^{-2l-1} \cdot t^l_{m\,\underline{n}}(X)$ which are homogeneous rational
functions of degree $-2l-2$ -- hence regular at infinity -- and in both cases
$$
\square t^l_{m\,\underline{n}}(X) =0, \qquad
\square \bigl( t^l_{m\,\underline{n}}(X) \cdot N(X)^{-2l-1} \bigr) =0.
$$

In this section we start with a function on the unit sphere $S^3$ centered
at the origin in $\BB R^4$,
realize $\BB R^4$ as a hyperplane $\{w^0=const\}$ inside $\BB R^{1,4}$ and
find the function's extensions to $\BB R^{1,4}_+$ which are annihilated both
by the wave operator $\square_{1,4}$ and the conformal Laplacian $\tlap$.
Let $(r, \overrightarrow{n})$ be the spherical coordinates of $\BB R^4$
spanned by $w^1$, $w^2$, $w^3$, $w^4$, so that
$$
r = \sqrt{(w^1)^2+(w^2)^2+(w^3)^2+(w^4)^2} \qquad \text{and} \qquad
\overrightarrow{n} \in S^3.
$$
Then
$$
\frac{\partial^2}{(\partial w^1)^2} + \frac{\partial^2}{(\partial w^2)^2}
+ \frac{\partial^2}{(\partial w^3)^2} + \frac{\partial^2}{(\partial w^4)^2}
= \frac{\partial^2}{\partial r^2} + \frac3r \frac{\partial}{\partial r}
+ \frac{\Delta_{S^3}}{r^2},
$$
where $\Delta_{S^3}$ denotes the spherical Laplacian on the unit sphere in
$\BB R^4$. In particular, we obtain a set of coordinates on $\BB R^{1,4}_+$:
$$
(w^0,w^1,w^2,w^3,w^4) \quad \leftrightsquigarrow \quad
(w^0, r, \overrightarrow{n}).
$$
We perform another change of coordinates
$$
(w^0, r, \overrightarrow{n}) \quad \leftrightsquigarrow \quad
(\rho, \theta, \overrightarrow{n})
$$
with
$$
w^0=\rho \cosh \theta, \qquad r= \rho \sinh \theta, \qquad
\rho = \sqrt{(w^0)^2-r^2}, \qquad \tanh \theta = r/w^0.
$$
Then
$$
\frac{\partial}{\partial w^0} = \cosh\theta \frac{\partial}{\partial \rho}
- \frac{\sinh \theta}{\rho} \frac{\partial}{\partial \theta}, \qquad
\frac{\partial}{\partial r} = - \sinh\theta \frac{\partial}{\partial \rho}
+ \frac{\cosh \theta}{\rho} \frac{\partial}{\partial \theta},
$$
$$
\frac{\partial^2}{(\partial w^0)^2} - \frac{\partial^2}{\partial r^2}
= \frac{\partial^2}{\partial \rho^2}
+ \frac1{\rho} \frac{\partial}{\partial \rho}
- \frac1{\rho^2} \frac{\partial^2}{\partial \theta^2},
$$
and it follows that
$$
\square_{1,4}
= \frac{\partial^2}{(\partial w^0)^2} - \frac{\partial^2}{\partial r^2}
- \frac3r \frac{\partial}{\partial r} - \frac{\Delta_{S^3}}{r^2}
= \frac{\partial^2}{\partial \rho^2}
+ \frac4{\rho} \frac{\partial}{\partial \rho}
- \frac1{\rho^2} \frac{\partial^2}{\partial \theta^2}
- \frac3{\rho^2} \frac{\cosh\theta}{\sinh\theta}\frac{\partial}{\partial \theta}
- \frac{\Delta_{S^3}}{\rho^2\sinh^2\theta}.
$$

Now we look for solutions of $\square_{1,4} \phi(W)=0$ in the separated form
$$
\phi(W) = \rho^{\lambda} \cdot s_l(\theta) \cdot
t^l_{m\,\underline{n}}(\overrightarrow{n}),
$$
where $t^l_{m\,\underline{n}}$'s are the matrix coefficients of $SU(2)$ defined
by (\ref{t}). Since $r^{2l} \cdot t^l_{m\,\underline{n}}(\overrightarrow{n})$
are harmonic homogeneous polynomials in $w^1, \dots, w^4$ of degree $2l$,
it follows that the matrix coefficients $t^l_{m\,\underline{n}}$'s are
eigenfunctions for $\Delta_{S^3}$:
$$
\Delta_{S^3} t^l_{m\,\underline{n}}(\overrightarrow{n})
= -4l(l+1) t^l_{m\,\underline{n}}(\overrightarrow{n}).
$$
Moreover, any eigenfunction of $\Delta_{S^3}$ is a linear combination of
$t^l_{m\,\underline{n}}$'s.

Since $\square_{\m}$ does not depend on $\rho$, by (\ref{deformed_laplacian}),
$$
\square_{1,4} \bigl(\rho^{\lambda} \cdot s_l(\theta) \cdot
t^l_{m\,\underline{n}}(\overrightarrow{n})\bigr) =0
\qquad \Longleftrightarrow \qquad
\bigl( \square_{\m} - \lambda(\lambda+3) \m^2 \bigr)
\bigl( s_l(\theta) \cdot t^l_{m\,\underline{n}}(\overrightarrow{n}) \bigr) =0.
$$
Recall that we are looking for functions annihilated by
$\tlap=\square_{\m}+2\m^2$ as well. Hence
\begin{equation}  \label{lambda-values}
\lambda(\lambda+3)+2=0 \qquad \text{and} \qquad
\lambda=-1 \quad \text{or} \quad \lambda=-2.
\end{equation}

The equation $\square_{1,4} \phi(W)=0$ becomes
an ordinary differential equation for $s_l(\theta)$:
\begin{multline*}
0= \square_{1,4} \phi(W)
= \biggl( \frac{\partial^2}{\partial \rho^2}
+ \frac4{\rho} \frac{\partial}{\partial \rho}
- \frac1{\rho^2} \frac{\partial^2}{\partial \theta^2}
- \frac3{\rho^2} \frac{\cosh\theta}{\sinh\theta}\frac{\partial}{\partial \theta}
- \frac{\Delta_{S^3}}{\rho^2\sinh^2\theta} \biggr)
\rho^{\lambda} \cdot s_l(\theta) \cdot t^l_{m\,\underline{n}}(\overrightarrow{n}) \\
=
\rho^{\lambda-2} \cdot t^l_{m\,\underline{n}}(\overrightarrow{n}) \cdot
\biggl( \lambda(\lambda-1) + 4\lambda - \frac{\partial^2}{\partial \theta^2}
- 3\frac{\cosh\theta}{\sinh\theta}\frac{\partial}{\partial \theta}
+ \frac{4l(l+1)}{\sinh^2\theta} \biggr) s_l(\theta). \\
\end{multline*}
Thus, the function $s_l(\theta)$ satisfies a differential equation
$$
\biggl( \frac{d^2}{d\theta^2}
+ 3\frac{\cosh\theta}{\sinh\theta}\frac{d}{d\theta}
- \frac{4l(l+1)}{\sinh^2\theta} - \lambda(\lambda+3)\biggr) s_l(\theta) =0.
$$
Changing the variable $\theta$ to $t=\cosh\theta$ and using
$\frac{d}{d\theta}=\sinh\theta\frac{d}{dt}$, we can rewrite this equation as
$$
\biggl( (t^2-1)\frac{d^2}{dt^2} + 4t\frac{d}{dt} - \frac{4l(l+1)}{t^2-1}
- \lambda(\lambda+3)\biggr) s_l(t) =0.
$$
But $\lambda(\lambda+3)=-2$ by (\ref{lambda-values}), so
$$
\biggl( (t^2-1)\frac{d^2}{dt^2} + 4t\frac{d}{dt} - \frac{4l(l+1)}{t^2-1}
+2 \biggr) s_l(t) =0.
$$
It is easy to verify directly that
$$
\frac{(t-1)^l}{(t+1)^{l+1}} \quad \text{and} \quad \frac{(t+1)^l}{(t-1)^{l+1}},
\qquad l=0, \frac12, 1, \frac32, 2, \dots,
$$
are two linearly independent solutions of this equation.
Thus we obtain four families of functions on $\BB R^{1,4}_+$ that
simultaneously satisfy $\square_{1,4} \phi =0$ and $\tlap \phi =0$:
$$
\rho^{\lambda} \cdot \frac{(\cosh\theta-1)^l}{(\cosh\theta+1)^{l+1} } \cdot
t^l_{m\,\underline{n}}(\overrightarrow{n}) \quad \text{and} \quad
\rho^{\lambda} \cdot \frac{(\cosh\theta+1)^l}{(\cosh\theta-1)^{l+1}} \cdot
t^l_{m\,\underline{n}}(\overrightarrow{n}), \qquad
\begin{matrix} \lambda=-1 \text{ or } \lambda=-2, \\
l=0,\frac12,1,\frac32,2,\dots. \end{matrix}
$$
Since $\m^2\rho^2 N(X) = r^2 = \rho^2 \sinh^2\theta$, we have
$\sinh^2\theta=\m^2N(X)$ and $\cosh^2\theta = 1+\m^2N(X)$.
Then we can rewrite our functions using
$$
t^l_{m\,\underline{n}}(\overrightarrow{n})
= t^l_{m\,\underline{n}}(X) \cdot N(X)^{-l}
= \m^{2l} \cdot (\sinh\theta)^{-2l} \cdot t^l_{m\,\underline{n}}(X)
= \m^{2l} \cdot (\cosh\theta -1)^{-l} \cdot (\cosh\theta +1)^{-l}
\cdot t^l_{m\,\underline{n}}(X).
$$

We summarize the results of this section as a proposition.

\begin{prop}
We have four families of functions on $\BB R^{1,4}_+$ that
simultaneously satisfy $\square_{1,4} \phi =0$ and $\tlap \phi =0$
\begin{equation} \label{t-extensions}
\frac{\rho^{\lambda} \cdot t^l_{m\,\underline{n}}(X)}
{\bigl(\bigl(1+\m^2N(X)\bigr)^{1/2}+1\bigr)^{2l+1}}
\quad \text{and} \quad
\frac{\rho^{\lambda} \cdot t^l_{m\,\underline{n}}(X)}
{\bigl(\bigl(1+\m^2N(X)\bigr)^{1/2}-1\bigr)^{2l+1}},
\qquad \begin{matrix} \lambda=-1 \text{ or} \\ \lambda=-2, \end{matrix}
\end{equation}
where $\rho = \sqrt{(w^0)^2 - (w^1)^2 - (w^2)^2 - (w^3)^2 - (w^4)^2}$ and
$X = (x^0,x^1,x^2,x^3) = \bigl( \frac{w^1}{\m \rho}, \frac{w^2}{\m \rho},
\frac{w^3}{\m \rho}, \frac{w^4}{\m \rho} \bigr)$.

Up to proportionality coefficients, these functions extend the
matrix coefficient functions $t^l_{m\,\underline{n}}(\overrightarrow{n})$ on
$$
S^3 = \bigl\{ W \in H_{\rho_0}; \: (w^1)^2+(w^2)^2+(w^3)^2+(w^4)^2=1 \bigr\}
\quad \subset \BB R^{1,4}_+.
$$
Moreover, any other extension of $t^l_{m\,\underline{n}}(\overrightarrow{n})$
to $\BB R^{1,4}_+$ satisfying both $\square_{1,4} \phi =0$ and
$\tlap \phi =0$ is a linear combination of functions (\ref{t-extensions}).
\end{prop}

\section{Spaces of Solutions of $\tlap\phi=0$ and
an Invariant Bilinear Pairing} \label{H_mu-section}

We denote by $\overline{{\cal H}_{\m}}$ the space of solutions of
$\tlap\phi=0$ on $\BB H^{\times}$; these solutions are real analytic functions.
We introduce algebraic subspaces
$$
{\cal H}_{\m}^+ = \text{$\BB C$-span of }
\frac{t^l_{m\,\underline{n}}(X)}
{\bigl( \bigl(1+\m^2N(X)\bigr)^{1/2}+1 \bigr)^{2l+1}},
\qquad \begin{matrix} l = 0, \frac12, 1, \frac32, 2, \dots, \\
m,n = -l , -l+1, \dots, l, \end{matrix}
$$
$$
{\cal H}_{\m}^- = \text{$\BB C$-span of }
\frac{t^l_{m\,\underline{n}}(X)}
{\bigl( \bigl(1+\m^2N(X)\bigr)^{1/2}-1 \bigr)^{2l+1}},
\qquad \begin{matrix} l = 0, \frac12, 1, \frac32, 2, \dots, \\
m,n = -l , -l+1, \dots, l, \end{matrix}
$$
and ${\cal H}_{\m} = {\cal H}_{\m}^+ \oplus {\cal H}_{\m}^-$.
Note that when $\m \to 0$,
$$
\frac{2^{2l+1} t^l_{m\,\underline{n}}(X)}
{\bigl( \bigl(1+\m^2N(X)\bigr)^{1/2}+1 \bigr)^{2l+1}}
\quad \to \quad t^l_{m\,\underline{n}}(X)
$$
and
$$
\frac{2^{-2l-1}\m^{4l+2} \cdot t^l_{m\,\underline{n}}(X)}
{\bigl( \bigl(1+\m^2N(X)\bigr)^{1/2}-1 \bigr)^{2l+1}}
\quad \to \quad t^l_{m\,\underline{n}}(X) \cdot N(X)^{-2l-1}.
$$
Since the restrictions of $t^l_{m\,\underline{n}}(X) \cdot
\bigl( \bigl(1+\m^2N(X)\bigr)^{1/2} \pm 1 \bigr)^{-(2l+1)}$
to the unit ball in $\BB H$ are dense in the space of all analytic
functions on that ball, ${\cal H}_{\m}$ is dense in $\overline{{\cal H}_{\m}}$,
justifying the notation. Taking closures we obtain a decomposition
$\overline{{\cal H}_{\m}} = \overline{{\cal H}_{\m}^+}
\oplus \overline{{\cal H}_{\m}^-}$.
The space ${\cal H}_{\m}$ can be characterized as the space of all
$SO(4)$-finite solutions of $\tlap\phi=0$ on $\BB H^{\times}$.
Then ${\cal H}_{\m}^+$ can be characterized as the subspace of
${\cal H}_{\m}$ consisting of functions that are regular at the origin.
Finally, ${\cal H}_{\m}^-$ can be characterized as the subspace of
${\cal H}_{\m}$ consisting of functions that decay at infinity
(or ``regular at infinity'').

We introduce a bilinear pairing between $\overline{{\cal H}_{\m}^+}$ and
$\overline{{\cal H}_{\m}^-}$:
\begin{multline}  \label{h-pairing}
(\phi_1, \phi_2)_{\m} =
\frac{\sqrt{1+\m^2R^2}}{2\pi^2} \int_{X \in S^3_R} (\degt\phi_1)(X) \cdot \phi_2(X)
\,\frac{dS}R  \\
=-\frac{\sqrt{1+\m^2R^2}}{2\pi^2} \int_{X \in S^3_R}
\phi_1(X) \cdot (\degt\phi_2)(X) \,\frac{dS}R, \qquad
\phi_1 \in \overline{{\cal H}_{\m}^+}, \: \phi_2 \in \overline{{\cal H}_{\m}^-},
\end{multline}
where $S^3_R$ denotes a sphere of radius $R>0$ in $\BB H$ centered at the
origin and $dS$ denotes the standard Euclidean measure on $S^3_R$ inherited
from $\BB H$.

\begin{prop}
The two expressions in (\ref{h-pairing}) agree; the resulting bilinear
pairing is $SO^+(1,4)$-invariant, $\mathfrak{so}(6,\BB C)$-invariant,
non-degenerate and independent of $R$. Moreover,
\begin{equation} \label{t-orthog-desitter}
\left(
\frac{t^l_{m\,\underline{n}}(X)}{\bigl( \bigl(1+\m^2N(X)\bigr)^{1/2}+1 \bigr)^{2l+1}},
\frac{t^{l'}_{n'\,m'}(X^+)}{\bigl( \bigl(1+\m^2N(X)\bigr)^{1/2}-1 \bigr)^{2l+1}}
\right)_{\m}
= \frac{\delta_{ll'}\delta_{mm'}\delta_{nn'}}{\m^{4l+2}}.
\end{equation}
\end{prop}

\begin{proof}
Using
\begin{equation}  \label{eq2}
\deg \bigl(1+\m^2N(X)\bigr)^{1/2} =
\m^2N(X) \cdot \bigl(1+\m^2N(X)\bigr)^{-1/2},
\end{equation}
we obtain
$$
\degt \biggl(
\frac{t^l_{m\,\underline{n}}(X)}
{\bigl(\bigl(1+\m^2N(X)\bigr)^{1/2}\pm1 \bigr)^{2l+1}}\biggr)
=
 \frac{\pm(2l+1)t^l_{m\,\underline{n}}(X)}{\bigl(1+\m^2N(X)\bigr)^{1/2} \cdot
\bigl( \bigl(1+\m^2N(X)\bigr)^{1/2}\pm 1 \bigr)^{2l+1}}.
$$
Then (\ref{t-orthog-desitter}) follows from
the orthogonality relations (\ref{t-orthog}).

Since these basis functions are dense in $\overline{{\cal H}_{\m}^+}$ and
$\overline{{\cal H}_{\m}^-}$ respectively, this computation proves that
the two expressions in (\ref{h-pairing}) agree, independent of $R>0$
and the resulting bilinear pairing is non-degenerate.
It remains to prove that it is $SO^+(1,4)$-invariant
and $\mathfrak{so}(6,\BB C)$-invariant.
The proof will be given in Corollary \ref{SO(1,4)-invariance}.
\end{proof}

\section{Poisson Formula}  \label{Poisson-section}

In this section we prove a Poisson-type formula for functions on $\BB H$
annihilated by $\tlap$. As an intermediate step, we derive an expansion for
$\bigl(\langle \hat X - \hat Y,\hat X - \hat Y\rangle_{1,4}\bigr)^{-1}$
similar to the matrix coefficient expansions for $N(X-Y)^{-1}$ given by
Proposition 25 from \cite{FL1} and restated here as
equation (\ref{1/N-expansion}).
We recall the notations of Lemma \ref{fundamental_sol}:
for $X, Y \in \BB H$, let
$$
\hat X = \bigl( \sqrt{\m^{-2}+N(X)},x^0,x^1,x^2,x^3 \bigr) \quad \text{and} \quad
\hat Y = \bigl( \sqrt{\m^{-2}+N(Y)},y^0,y^1,y^2,y^3 \bigr)
\quad \in \BB R^{1,4}_+.
$$

\begin{prop}  \label{1/N-exp-deformed}
We have the following expansion:
$$
- \frac1{\langle \hat X - \hat Y, \hat X - \hat Y \rangle_{1,4}}
= \sum_{l,m,n} \frac{t^l_{m\,\underline{n}}(X)}
{\bigl( \bigl(1+\m^2N(X)\bigr)^{1/2}+1 \bigr)^{2l+1}}
\cdot \frac{\m^{4l+2}\cdot t^l_{n\,\underline{m}}(Y^+)}
{\bigl((1+\m^2N(Y))^{1/2}-1\bigr)^{2l+1}},
$$
which converges pointwise absolutely in the region
$\{ (X,Y) \in \BB H \times \BB H ;\: N(X) < N(Y) \}$.
The sum is taken first over all $m,n=-l,-l+1,\dots,l$, then over
$l = 0,\frac12,1,\frac32,2,\dots$.
\end{prop}

\begin{proof}
Let $X, Y \in \BB H$ and
$$
\overrightarrow{u} = \frac{X}{\sqrt{N(X)}}, \quad
\overrightarrow{v} = \frac{Y}{\sqrt{N(Y)}}
\quad \in SU(2) \quad \subset \BB H,
$$
then $\overrightarrow{u} \overrightarrow{v}^{-1}$ is similar to a diagonal
matrix $\bigl(\begin{smallmatrix} \lambda & 0 \\
0 & \lambda^{-1} \end{smallmatrix}\bigr)$,
where $\lambda \in \BB C$, $|\lambda|=1$.
Define $\theta_1$, $\theta_2$, $t_1$ and $t_2$ by
\begin{equation}  \label{theta-t}
\sinh \theta_1 = \m \sqrt{N(X)}, \qquad
\sinh \theta_2 = \m \sqrt{N(Y)}, \qquad
t_1 = \cosh \theta_1, \qquad t_2 = \cosh \theta_2.
\end{equation}
Using the multiplicativity property of matrix coefficients (\ref{t-mult}),
we compute:
\begin{multline*}
\sum_{l,m,n} \frac{t^l_{m\,\underline{n}}(X)}
{\bigl( \bigl(1+\m^2N(X)\bigr)^{1/2}+1 \bigr)^{2l+1}}
\cdot \frac{\m^{4l+2}\cdot t^l_{n\,\underline{m}}(Y^+)}
{\bigl((1+\m^2N(Y))^{1/2}-1\bigr)^{2l+1}}\\
= \m^2 \sum_l \frac{(t_1-1)^l}{(t_1+1)^{l+1}} \cdot \frac{(t_2+1)^l}{(t_2-1)^{l+1}}
\cdot \chi_l(\overrightarrow{u} \overrightarrow{v}^{-1}) \\
= \frac{\m^2}{\lambda-\lambda^{-1}} \sum_l \frac{(t_1-1)^l}{(t_1+1)^{l+1}} \cdot
\frac{(t_2+1)^l}{(t_2-1)^{l+1}} \cdot (\lambda^{2l+1}-\lambda^{-2l-1}).
\end{multline*}
Let
\begin{equation}  \label{ab}
a= (e^{\theta_1/2}+e^{-\theta_1/2})(e^{\theta_2/2}-e^{-\theta_2/2}), \qquad
b= (e^{\theta_1/2}-e^{-\theta_1/2})(e^{\theta_2/2}+e^{-\theta_2/2}),
\end{equation}
then
\begin{multline*}
\sum_{l,m,n} \frac{t^l_{m\,\underline{n}}(X)}
{\bigl( \bigl(1+\m^2N(X)\bigr)^{1/2}+1 \bigr)^{2l+1}}
\cdot \frac{\m^{4l+2}\cdot t^l_{n\,\underline{m}}(Y^+)}
{\bigl((1+\m^2N(Y))^{1/2}-1\bigr)^{2l+1}}\\
= \frac{4\m^2 (\lambda-\lambda^{-1})^{-1}}{(e^{\theta_1}+e^{-\theta_1}+2)
(e^{\theta_2}+e^{-\theta_2}-2)} \sum_l a^{-2l} \cdot b^{2l} \cdot
(\lambda^{2l+1}-\lambda^{-2l-1}) \\
= \frac{4\m^2}{(\lambda-\lambda^{-1}) a} \biggl( \frac{\lambda}{a-\lambda b} 
- \frac{\lambda^{-1}}{a-\lambda^{-1}b} \biggr)
= \frac{4\m^2}{(a-\lambda b)(a-\lambda^{-1}b)}
= \frac{4\m^2}{N(b\overrightarrow{u} - a\overrightarrow{v})}.
\end{multline*}
Since $X=\m^{-1}\sinh\theta_1\overrightarrow{u}$ and
$Y=\m^{-1}\sinh\theta_2\overrightarrow{v}$,
\begin{multline}  \label{1/N}
\frac{4\m^2}{N(b\overrightarrow{u} - a\overrightarrow{v})}
= 4\m^2 \left[ N\biggl(\frac{\m bX}{\sinh\theta_1}
- \frac{\m aY}{\sinh\theta_2} \biggr) \right]^{-1}  \\
= \left[ N\biggl( \frac{e^{\theta_2/2}+e^{-\theta_2/2}}{e^{\theta_1/2}+e^{-\theta_1/2}}X -
\frac{e^{\theta_1/2}+e^{-\theta_1/2}}{e^{\theta_2/2}+e^{-\theta_2/2}}Y \biggr)\right]^{-1}
= \frac{(\cosh\theta_1+1) (\cosh\theta_2+1)}
{N\bigl( (\cosh\theta_2+1)X - (\cosh\theta_1+1)Y \bigr)} \\
= \frac{1/2}{\m^{-2}\cosh\theta_1\cosh\theta_2 - \m^{-2}
- (x^0y^0+x^1y^1+x^2y^2+x^3y^3)} \\
= - \frac1{\langle \hat X , \hat X \rangle_{1,4}
+ \langle \hat Y, \hat Y \rangle_{1,4} - 2 \langle \hat X, \hat Y \rangle_{1,4}}
= - \frac1{\langle \hat X - \hat Y,\hat X - \hat Y \rangle_{1,4}}.
\end{multline}
This expansion holds whenever $|b/a|<1$.
Since
$$
\frac ba = \frac{\tanh(\theta_1/2)}{\tanh(\theta_2/2)} \ge 0
$$
and $\tanh \theta$ is monotone increasing, the expansion holds whenever
$\theta_1<\theta_2$ or, equivalently, $N(X)<N(Y)$.
\end{proof}

We introduce a notation
$$
K_{\m}(X,Y) =
-\frac1{\langle \hat X - \hat Y, \hat X - \hat Y \rangle_{1,4}}.
$$
Now we are ready to prove the Poisson-type formula.
Let $S^3_R$ denote a sphere of radius $R$ in $\BB H$ centered at the origin.

\begin{thm} \label{Poisson5}
Let $\phi$ be a real analytic solution of $\tlap \phi =0$
defined on a closed ball $\{X \in \BB H ;\: N(X) \le R^2 \}$, for some $R>0$.
Then, for all $Y \in \BB H$ with $N(Y)<R^2$,
\begin{align*}
\phi(Y) &= \bigl( \phi(X), K_{\m}(X,Y) \bigr)_{\m} \\
&= \frac{\sqrt{1+\m^2R^2}}{2\pi^2} \int_{X \in S^3_R} (\degt\phi)(X) \cdot
K_{\m}(X,Y) \,\frac{dS}R \\
&= - \frac{\sqrt{1+\m^2R^2}}{2\pi^2} \int_{X \in S^3_R}
\bigl( \degt_X K_{\m}\bigr) (X,Y) \cdot \phi(X) \,\frac{dS}R.
\end{align*}

Similarly, suppose $\phi$ is a real analytic solution of $\tlap \phi =0$
defined on a closed set $\{X \in \BB H ;\: N(X) \ge R^2 \}$, for some $R>0$,
and regular at infinity. Then, for all $Y \in \BB H$ with $N(Y)>R^2$,
\begin{align*}
\phi(Y) &= \bigl( K_{\m}(X,Y), \phi(X) \bigr)_{\m} \\
&= - \frac{\sqrt{1+\m^2R^2}}{2\pi^2} \int_{X \in S^3_R} (\degt\phi)(X) \cdot
K_{\m}(X,Y) \,\frac{dS}R \\
&= \frac{\sqrt{1+\m^2R^2}}{2\pi^2} \int_{X \in S^3_R}
\bigl( \degt_X K_{\m} \bigr) (X,Y) \cdot \phi(X) \,\frac{dS}R.
\end{align*}
\end{thm}

\begin{proof}
It is sufficient to prove the formulas when
$$
\phi(X) = \frac{t^l_{m\,\underline{n}}(X)}
{\bigl( \bigl(1+\m^2N(X)\bigr)^{1/2} \pm1 \bigr)^{2l+1}}.
$$
Then the result follows from the expansion of the kernel $K_{\m}(X,Y)$ and
orthogonality relations (\ref{t-orthog-desitter}).
\end{proof}

\begin{cor} \label{SO(1,4)-invariance}
The bilinear pairing (\ref{h-pairing}) is $SO^+(1,4)$-invariant and
$\mathfrak{so}(6,\BB C)$-invariant.
\end{cor}

\begin{proof}
Since the group is connected, to prove $SO^+(1,4)$-invariance, it is
sufficient to show invariance for $a \in SO^+(1,4)$ sufficiently close to the
identity only.
Choose $R_1$, $R_2$ such that $0<R_2<R<R_1$ and, using the Poisson formula, write
$$
\phi_i(X) = (-1)^{i+1}\frac{\sqrt{1+\m^2R_i^2}}{2\pi^2} \int_{Y_i \in S^3_{R_i}}
(\degt\phi_i)(Y_i) \cdot K_{\m}(X,Y_i) \,\frac{dS}{R_i}, \qquad i=1,2.
$$
In short,
$$
\phi_1(X) = \bigl( \phi_1(Y_1), K_{\m}(X,Y_1) \bigr)_{\m}
\qquad \text{and} \qquad
\phi_2(X) = \bigl( K_{\m}(X,Y_2), \phi_2(Y_2) \bigr)_{\m}.
$$
Then
\begin{multline*}
\bigl( \bigl(\pi(a)\phi_1\bigr)(X), \bigl(\pi(a)\phi_2\bigr)(X) \bigr)_{\m}  \\
= \Bigl( \bigl( \phi_1(Y_1), \pi(a)_X K_{\m}(X,Y_1) \bigr)_{\m},
\bigl( \pi(a)_X K_{\m}(X,Y_2), \phi_2(Y_2) \bigr)_{\m}
\Bigr)_{\m} \\
= \bigl( (\phi_1(Y_1), K_{\m}(X,Y_1))_{\m},
(K_{\m}(X,Y_2), \phi_2(Y_2))_{\m} \bigr)_{\m}
= (\phi_1(X), \phi_2(X))_{\m}
\end{multline*}
because the Poisson formula is $SO^+(1,4)$-equivariant and
\begin{multline*}
\bigl(\pi(a)_X K_{\m}(X,Y_1), \pi(a)_X K_{\m}(X,Y_2) \bigr)_{\m} \\
= \bigl( \pi(a)_{Y_1} \circ \pi(a)_X K_{\m}(X,Y_1), K_{\m}(X,Y_2) \bigr)_{\m}
= (K_{\m}(X,Y_1), K_{\m}(X,Y_2))_{\m}.
\end{multline*}

Since the action of $\mathfrak{so}(6,\BB C)$ is generated by the actions of
$\mathfrak{so}(1,4)$ and the operator
$\phi \mapsto \bigl(1+\m^2N(X)\bigr)^{1/2} \degt \phi$,
to prove $\mathfrak{so}(6,\BB C)$-invariance, it is sufficient to prove
invariance under the operator
$\phi \mapsto \bigl(1+\m^2N(X)\bigr)^{1/2} \degt \phi$ only.
By Lemma \ref{fundamental_sol},
\begin{multline*}
\bigl(1+\m^2N(X)\bigr)^{1/2} \degt_X K_{\m}(X,Y)
= \frac2{\m^2} \frac{\bigl(1+\m^2N(Y)\bigr)^{1/2} - \bigl(1+\m^2N(X)\bigr)^{1/2}}
{\bigl( \langle \hat X - \hat Y, \hat X - \hat Y \rangle_{1,4} \bigr)^2} \\
= - \bigl(1+\m^2N(Y)\bigr)^{1/2} \degt_Y K_{\m}(X,Y),
\end{multline*}
and then the proof continues in exactly the same way as for
$SO^+(1,4)$-invariance.
\end{proof}

\section{Regular Functions}  \label{reg-fun-section}

In order to define analogues of left and right regular functions,
we need to factor $\tlap$ as a product of two Dirac-like operators.
The factorization
$$
\square = \nabla\nabla^+ = \nabla^+\nabla
\quad \text{can be rewritten as}\quad
\begin{pmatrix} 0 & \nabla^+ \\ \nabla & 0 \end{pmatrix}^2
= \begin{pmatrix} \square & 0 \\ 0 & \square \end{pmatrix}
= \square \cdot I_{2 \times 2},
$$
and
$$
\begin{pmatrix} 0 & \nabla^+ \\ \nabla & 0 \end{pmatrix}
\begin{pmatrix} f_1 \\ f_2 \end{pmatrix} =0
\qquad \Longleftrightarrow \qquad
\text{$\nabla f_1=0$ and $\nabla^+f_2=0$},
$$
i.e. $f_2$ is left-regular and $f_1$ is anti-left-regular.

\begin{prop} \label{lap-factorization}
Let
$$
\nabla_{\m} =
\begin{pmatrix} \m(X\nabla-\degt) & \bigl(1+\m^2N(X)\bigr)^{1/2} \nabla^+ \\
\bigl(1+\m^2N(X)\bigr)^{1/2} \nabla & \m(X^+\nabla^+-\degt) \end{pmatrix},
$$
then we have a factorization
$$
\tlap = \nabla_{\m}(\nabla_{\m}-\m).
$$
\end{prop}

\begin{proof}
We use the following identities:
\begin{equation}  \label{eq1}
\nabla \bigl(1+\m^2N(X)\bigr)^{1/2} = \frac{\m^2X^+}{\bigl(1+\m^2N(X)\bigr)^{1/2}},
\qquad
\nabla^+ \bigl(1+\m^2N(X)\bigr)^{1/2} = \frac{\m^2X}{\bigl(1+\m^2N(X)\bigr)^{1/2}},
\end{equation}
\begin{equation}  \label{eq3}
X^+\nabla^+ + \nabla X = \nabla^+X^+ + X\nabla = 2(2+\deg),
\end{equation}
the last identity in turn implies
\begin{equation}  \label{eq4}
(X\nabla-\degt)^2 = (X^+\nabla^+-\degt)^2 = \degt^2 - N(X)\square.
\end{equation}
First, we find that $\nabla_\m^2$ is equal to
$$
\begin{pmatrix} \begin{matrix} \m^2(X\nabla-\degt)^2 + \\
\bigl(1+\m^2N(X)\bigr)^{1/2}\nabla^+\bigl(1+\m^2N(X)\bigr)^{1/2}\nabla\end{matrix}
&  & \begin{matrix} \m(X\nabla-\degt)\bigr(1+\m^2N(X)\bigr)^{1/2}\nabla^+ \\
+ \m\bigl(1+\m^2N(X)\bigr)^{1/2}\nabla^+(X^+\nabla^+-\degt) \end{matrix} \\
\quad &  & \quad \\
\begin{matrix} \m(X^+\nabla^+-\degt)\bigl(1+\m^2N(X)\bigr)^{1/2}\nabla \\
+ \m\bigl(1+\m^2N(X)\bigr)^{1/2}\nabla(X\nabla-\degt) \end{matrix} &  &
\begin{matrix} \m^2(X^+\nabla^+-\degt)^2 +\\
\bigl(1+\m^2N(X)\bigr)^{1/2}\nabla\bigl(1+\m^2N(X)\bigr)^{1/2}\nabla^+
\end{matrix} \end{pmatrix}
$$
and then work out each entry separately.
By (\ref{eq1}) and (\ref{eq4}), the diagonal terms are
\begin{align*}
&\m^2\bigl(\degt^2-N(X)\square\bigr) + \bigl(1+\m^2N(X)\bigr)\square+\m^2X\nabla
= \square + \m^2(\degt^2 + X\nabla),\\
& \m^2\bigl(\degt^2-N(X)\square\bigr)
+\bigl(1+\m^2N(X)\bigr)\square+\m^2X^+\nabla^+
= \square + \m^2(\degt^2 + X^+\nabla^+).
\end{align*}
By (\ref{eq2}), (\ref{eq1}) and (\ref{eq3}), the off-diagonal terms are
\begin{multline*}
\m(X\nabla-\degt)\bigl(1+\m^2N(X)\bigr)^{1/2}\nabla^+
+ \m\bigl(1+\m^2N(X)\bigr)^{1/2}\nabla^+(X^+\nabla^+-\degt) \\
= \m\bigl(1+\m^2N(X)\bigr)^{1/2}
\bigl( X\square - \degt\nabla^+ + \nabla^+X^+\nabla^+ - \nabla^+\degt \bigr)
= \m\bigl(1+\m^2N(X)\bigr)^{1/2}\nabla^+
\end{multline*}
and similarly
\begin{multline*}
\m(X^+\nabla^+-\degt)\bigl(1+\m^2N(X)\bigr)^{1/2}\nabla
+ \m\bigl(1+\m^2N(X)\bigr)^{1/2}\nabla(X\nabla-\degt) \\
= \m\bigl(1+\m^2N(X)\bigr)^{1/2}
\bigl( X^+\square - \degt\nabla + \nabla X\nabla - \nabla\degt \bigr)
= \m\bigl(1+\m^2N(X)\bigr)^{1/2}\nabla.
\end{multline*}
This proves $\nabla_\m^2 = \tlap + \m \nabla_{\m}$.
\end{proof}

The two equations (\ref{eq3}) can be combined into a single equation as
$$
\begin{pmatrix} 0 & \nabla^+ \\ \nabla & 0 \end{pmatrix}
\begin{pmatrix} 0 & X \\ X^+ & 0 \end{pmatrix}
+ \begin{pmatrix} 0 & X \\ X^+ & 0 \end{pmatrix}
\begin{pmatrix} 0 & \nabla^+ \\ \nabla & 0 \end{pmatrix}
= 2 \begin{pmatrix} \deg+2 & 0 \\ 0 & \deg+2 \end{pmatrix} = 2(\deg+2).
$$
In our context this formula becomes
$$
(\nabla_{\m}-\m) \begin{pmatrix} 0 & X \\ X^+ & 0 \end{pmatrix}
+ \begin{pmatrix} 0 & X \\ X^+ & 0 \end{pmatrix} \nabla_{\m}
= 2\bigl( 1+\m^2N(X) \bigr)^{1/2}(\deg+2).
$$


We introduce another operator, $\overleftarrow{\nabla}_{\m}$,
which we apply to functions on the right:
\begin{multline*}
\overleftarrow{\nabla}_{\m}: (g_1, g_2) \quad \mapsto \quad
(g_1, g_2)
\begin{pmatrix} \m(\degt-\nabla^+X^+) & \nabla^+\bigl(1+\m^2N(X)\bigr)^{1/2} \\
\nabla \bigl(1+\m^2N(X)\bigr)^{1/2} & \m(\degt-\nabla X) \end{pmatrix} \\
= \begin{pmatrix} \m(\degt g_1-(g_1\nabla^+)X^+) +
\bigl(1+\m^2N(X)\bigr)^{1/2}(g_2\nabla) \\
\bigl(1+\m^2N(X)\bigr)^{1/2}(g_1\nabla^+)
+ \m(\degt g_2-(g_2\nabla)X) \end{pmatrix}.
\end{multline*}

\begin{prop}
We have a factorization
$$
\tlap = \overleftarrow{\nabla}_{\m} (\overleftarrow{\nabla}_{\m}+\m).
$$
\end{prop}

\begin{rem}  \label{reg-remark}
One can produce functions satisfying $\nabla_{\m} f=0$ and
$g \overleftarrow{\nabla}_{\m}=0$ as follows.
Start with a $\BB C$-valued function $\phi$ annihilated by $\tlap$
(for example, take $\phi \in {\cal H}_{\m}$). Then the two columns of
$(\nabla_{\m}-\m)\phi$ satisfy $\nabla_{\m} f=0$ and the two rows of
$\phi (\overleftarrow{\nabla}_{\m}+\m)$ satisfy $g \overleftarrow{\nabla}_{\m}=0$.
\end{rem}

\section{Basic Properties of Regular Functions}  \label{reg-fun-prop-section}

Recall from \cite{FL1} that
$$
Dx = dx^1 \wedge dx^2 \wedge dx^3 - i dx^0 \wedge dx^2 \wedge dx^3
+ j dx^0 \wedge dx^1 \wedge dx^3 - k dx^0 \wedge dx^1 \wedge dx^2
$$
is an $\BB H$-valued 3-form on $\BB H$ that is Hodge dual to
$$
dX = dx^0 + idx^1 + jdx^2 + kdx^3.
$$
We also introduce
$$
Dx^+ = dx^1 \wedge dx^2 \wedge dx^3 + i dx^0 \wedge dx^2 \wedge dx^3
- j dx^0 \wedge dx^1 \wedge dx^3 + k dx^0 \wedge dx^1 \wedge dx^2,
$$
which is Hodge dual to
$$
dX^+ = dx^0 - idx^1 - jdx^2 - kdx^3
$$
and
$$
Dr=  x^0 dx^1 \wedge dx^2 \wedge dx^3 - x^1 dx^0 \wedge dx^2 \wedge dx^3
+ x^2 dx^0 \wedge dx^1 \wedge dx^3 - x^3 dx^0 \wedge dx^1 \wedge dx^2,
$$
which is Hodge dual to
$$
dr= x^0dx^0 + x^1dx^1 + x^2dx^2 + x^3dx^3.
$$
Recall that $dV= dx^0 \wedge dx^1 \wedge dx^2 \wedge dx^3$
is the volume form on $\BB H$.
If $f= \bigl( \begin{smallmatrix} f_1 \\ f_2 \end{smallmatrix} \bigr)$ and
$g= (g_1, g_2)$ are two functions defined on an open set in $\BB H$, then
$$
\begin{matrix}
d(Dx \cdot f) = (\nabla^+ f)dV, & \qquad & d(Dx^+\cdot f) = (\nabla f)dV,
& \qquad & d(f \cdot Dr)= (\deg f +4f)dV, \\
d(g \cdot Dx) = (g\nabla^+)dV, & \qquad & d(g \cdot Dx^+) = (g\nabla) dV.
\end{matrix}
$$
Note that
$$
Dr = \frac12 (X \cdot Dx^+ + Dx \cdot X^+)
= \frac12 (X^+ \cdot Dx + Dx^+ \cdot X).
$$
Consider a matrix-valued 3-form on $\BB H$
$$
Dx_{\m} = \begin{pmatrix} \m \frac{X \cdot Dx^+ - Dr}{(1+\m^2N(X))^{1/2}} &
Dx \\ Dx^+ & \m \frac{X^+ \cdot Dx - Dr}{(1+\m^2N(X))^{1/2}} \end{pmatrix}
= \begin{pmatrix}
\frac{\m}2 \frac{X \cdot Dx^+ - Dx \cdot X^+}{(1+\m^2N(X))^{1/2}} & Dx \\
Dx^+ & \frac{\m}2 \frac{X^+ \cdot Dx - Dx^+ \cdot X}{(1+\m^2N(X))^{1/2}}
\end{pmatrix}.
$$

\begin{lem}
\begin{equation*}
d(g \cdot Dx_{\m} \cdot f) = \bigl( 1+\m^2N(X)\bigr)^{-1/2}
\bigl( (g \overleftarrow{\nabla}_{\m})f+g(\nabla_{\m}f) \bigr)dV.
\end{equation*}
\end{lem}

\begin{proof}
Using
$$
\bigl[ X\nabla-\deg, \bigl( 1+\m^2N(X)\bigr)^{\alpha} \bigr] =
\bigl[ X^+\nabla^+-\deg, \bigl( 1+\m^2N(X)\bigr)^{\alpha} \bigr] = 0,
$$
we compute the components of $d(g \cdot Dx_{\m} \cdot f)$ coming from
the diagonal entries of $Dx_{\m}$:
\begin{multline*}
d\bigl( g_1 (X \cdot Dx^+ - Dx \cdot X^+) f_1 \bigr)
= \Bigl( \bigl( (g_1X)\nabla-(g_1\nabla^+)X^+ \bigr) \cdot f_1
+ g_1 \cdot \bigl( X (\nabla f_1) - \nabla^+(X^+f_1) \bigr) \Bigr) dV \\
= 2 \Bigl( \bigl( \deg g_1 - (g_1\nabla^+)X^+ \bigr) \cdot f_1
+ g_1 \cdot \bigl( X (\nabla f_1) - \deg f_1 \bigr) \Bigr) dV,
\end{multline*}
\begin{multline*}
d\bigl( g_2 (X^+ \cdot Dx - Dx^+ \cdot X) f_2 \bigr)
= \Bigl( \bigl( (g_2X^+)\nabla^+ - (g_2\nabla)X \bigr) \cdot f_2
+ g_2 \cdot \bigl( X^+ (\nabla^+ f_2) - \nabla(Xf_2) \bigr) \Bigr) dV \\
= 2 \Bigl( \bigl( \deg g_2 - (g_2\nabla)X \bigr) \cdot f_2
+ g_2 \cdot \bigl( X^+ (\nabla^+ f_2) - \deg f_2 \bigr) \Bigr) dV.
\end{multline*}
Then the result follows.
\end{proof}

As an immediate consequence we obtain Cauchy's integral theorem:

\begin{cor}  \label{cauchy}
Let $f= \bigl( \begin{smallmatrix} f_1 \\ f_2 \end{smallmatrix} \bigr)$ and
$g= (g_1, g_2)$ be two functions defined on an open set $U \subset \BB H$ such
that $\nabla_{\m}f=0$ and $g \overleftarrow{\nabla}_{\m}=0$.
Then $g \cdot Dx_{\m} \cdot f$ is a closed 3-form.
In particular, if $C$ is a 3-cycle in $U$ (with compact support),
then the integral $\int_C g \cdot Dx_{\m} \cdot f$ depends only on
the homology class of $C$.
\end{cor}

\begin{lem}
Let $a, d \in \BB H$ with $N(a)=N(d)=1$, then the pull-back of $Dx_{\m}$ under
the map $X \mapsto aXd^{-1}$ is
$$
\begin{pmatrix} a & 0 \\ 0 & d \end{pmatrix} Dx_{\m}
\begin{pmatrix} a & 0 \\ 0 & d \end{pmatrix}^{-1}.
$$
\end{lem}

\begin{lem}  \label{rotation}
Let $f= \bigl(\begin{smallmatrix} f_1 \\ f_2 \end{smallmatrix} \bigr)$ be a
left-regular function (i.e. satisfying $\nabla_{\m} f=0$). Then so is
$$
\begin{pmatrix} a^{-1} f_1(aXd^{-1}) \\ d^{-1}f_2(aXd^{-1}) \end{pmatrix},
\qquad \text{for any $a, d \in \BB H$ with $N(a)=N(d)=1$}.
$$
\end{lem}

\begin{proof}
Using
$$
\nabla \bigl( a^{-1} f_1(aXd^{-1}) \bigr) = d^{-1} (\nabla f_1) \bigr|_{aXd^{-1}}
\qquad \text{and} \qquad
\nabla^+ \bigl( d^{-1} f_2(aXd^{-1}) \bigr) = a^{-1}(\nabla^+ f_2) \bigr|_{aXd^{-1}},
$$
we compute
$$
\nabla_{\m} \begin{pmatrix} a^{-1} f_1(aXd^{-1}) \\ d^{-1}f_2(aXd^{-1}) \end{pmatrix}
= \begin{pmatrix} a & 0 \\ 0 & d \end{pmatrix}^{-1}
\nabla_{\m} \begin{pmatrix} f_1 \\ f_2 \end{pmatrix} \biggr|_{aXd^{-1}} =0.
$$
\end{proof}

\section{Analogue of the Cauchy-Fueter Formula}  \label{Cauchy-Fueter-section}

The Cauchy-Fueter kernel in our setting is
$$
k_{\m}(X,Y) = -\frac12 K_{\m}(X,Y) (\overleftarrow{\nabla}_{\m}+\m).
$$
If $U \subset \BB H$ is an open region with piecewise ${\cal C}^1$ boundary
$\partial U$, we define a preferred orientation on $\partial U$ as follows.
The positive orientation of $U$ is determined by the vectors
$\{1, i, j, k \}$ (or the volume form $dV$).
Pick a non-singular point $p \in \partial U$ and let $\overrightarrow{n_p}$
be a non-zero vector in $T_p\BB H$ perpendicular to $T_p\partial U$ and
pointing outside of $U$.
Then $\{\overrightarrow{\tau_1}, \overrightarrow{\tau_2},
\overrightarrow{\tau_3}\} \subset T_p \partial U$ is positively oriented
in $\partial U$ if and only if
$\{\overrightarrow{n_p}, \overrightarrow{\tau_1}, \overrightarrow{\tau_2},
\overrightarrow{\tau_3}\}$ is positively oriented in $\BB H$.
Now we can prove our analogue of the Cauchy-Fueter formula:

\begin{thm}  \label{Cauchy-Fueter-5}
Let $U \subset \BB H$ be an open bounded subset with piecewise ${\cal C}^1$
boundary $\partial U$. Suppose that $f(X)$ is left-regular on a neighborhood
of the closure $\overline{U}$ (i.e. satisfying $\nabla_{\m} f=0$), then
$$
\frac 1 {2\pi^2} \int_{\partial U} k_{\m}(X,Y) \cdot Dx_{\m} \cdot f(X) =
\begin{cases}
f(Y) & \text{if $Y \in U$;} \\
0 & \text{if $Y \notin \overline{U}$.}
\end{cases}
$$
\end{thm}

\begin{rem}
There is a similar formula for right-regular functions
(i.e. functions satisfying $g\overleftarrow{\nabla}_{\m} =0$).
The Cauchy-Fueter kernel in that case is
$k'_{\m}(X,Y) = -\frac12 (\nabla_{\m}-\m) K_{\m}(X,Y)$.
\end{rem}

\begin{proof}
By Remark \ref{reg-remark} and Corollary \ref{cauchy} the integrand is a
closed 3-form.
If $Y \notin \overline{U}$, by Corollary \ref{cauchy} the integral is zero.
So let us assume $Y \in U$.
Consider a sphere $S^3_{\epsilon}(Y)$ of radius $\epsilon$ centered at $Y$,
and let $\epsilon$ be small enough so that the closed ball of radius
$\epsilon$ centered at $Y$ lies inside $U$. Then
$$
\int_{\partial U} k_{\m}(X,Y) \cdot Dx_{\m} \cdot f(X)
= \int_{S^3_{\epsilon}(Y)} k_{\m}(X,Y) \cdot Dx_{\m} \cdot f(X).
$$
The right hand side is independent from $\epsilon$
and it is sufficient to show that
$$
\lim_{\epsilon \to 0^+} \int_{S^3_{\epsilon}(Y)} k_{\m}(X,Y) \cdot Dx_{\m} \cdot f(X)
= 2\pi^2 \cdot f(Y).
$$
We compute
\begin{multline*}
k_{\m}(X,Y) \cdot
\bigl( \langle \hat X - \hat Y, \hat X - \hat Y \rangle_{1,4} \bigr)^2 \\
= \begin{pmatrix} \m YX^+ & X\sqrt{1+\m^2N(Y)} - Y\sqrt{1+\m^2N(X)} \\
X^+\sqrt{1+\m^2N(Y)} - Y^+\sqrt{1+\m^2N(X)} & \m Y^+X \end{pmatrix} \\
+ 2\m^{-1} +\frac{\m}2 \tr(XY^+)
- 2\m^{-1}\sqrt{1+\m^2N(X)}\cdot\sqrt{1+\m^2N(Y)}.
\end{multline*}
Let $X'=X-Y$, then $N(X')=\epsilon^2$ and by Lemma 6 of \cite{FL1}
$$
Dx_{\m} \bigr|_{S^3_{\epsilon}(Y)} =
\begin{pmatrix} \frac{\m}2 \frac{XX'^+-X'X^+}{\sqrt{1+\m^2N(X)}} & X' \\
X'^+ & \frac{\m}2 \frac{X^+X'-X'^+X}{\sqrt{1+\m^2N(X)}} \end{pmatrix}
\frac{dS}{\epsilon}.
$$
To simplify upcoming expressions, we introduce a notation
$b = \sqrt{1+\m^2N(Y)}$. Working with the lowest order terms with respect to
$X'$ and ignoring the higher order terms we get:
$$
f(X) \sim f(Y), \qquad
\bigl( \langle \hat X - \hat Y, \hat X - \hat Y \rangle_{1,4} \bigr)^2
\sim \epsilon^{4} \cdot
\biggl(1- \frac{\m^2}{4\epsilon^2b^2} \bigl( \tr(X'Y^+) \bigr)^2 \biggr)^{2},
$$
$$
Dx_{\m} \bigr|_{S^3_{\epsilon}(Y)} \sim \epsilon^{-1} \cdot
\begin{pmatrix} \frac{\m}{2b} (YX'^+-X'Y^+) & X' \\
X'^+ & \frac{\m}{2b} (Y^+X'-X'^+Y) \end{pmatrix} dS,
$$
\begin{multline*}
k_{\m}(X,Y) \sim \\
\epsilon^{-4} \cdot
\biggl(1- \frac{\m^2}{4\epsilon^2b^2} \bigl( \tr(X'Y^+) \bigr)^2 \biggr)^{-2}
\cdot \begin{pmatrix} \m YX'^+ - \frac{\m}2 \tr(X'Y^+) &
b X' - \frac{\m^2}{2b} Y \cdot \tr(X'Y^+) \\
b X'^+ - \frac{\m^2}{2b} Y^+ \cdot \tr(X'Y^+) &
\m Y^+X' - \frac{\m}2 \tr(X'Y^+) \end{pmatrix}.
\end{multline*}
Using Lemma \ref{rotation} we can assume that $Y$ is real. Then
$$
Dx_{\m} \bigr|_{S^3_{\epsilon}(Y)} \sim \epsilon^{-1} \cdot
\begin{pmatrix} - \frac{\m}{b} Y \cdot \im(X') & X' \\
X'^+ & \frac{\m}{b} Y \cdot \im(X') \end{pmatrix} dS,
$$
\begin{multline*}
k_{\m}(X,Y) \sim \\
\epsilon^{-4} \cdot
\biggl(1- \frac{\m^2 Y^2}{\epsilon^2b^2} \bigl( \re(X') \bigr)^2 \biggr)^{-2}
\cdot
\begin{pmatrix} -\m Y \cdot \im(X') & b X' - \frac{\m^2}b Y^2 \cdot \re(X') \\
b X'^+ - \frac{\m^2}b Y^2 \cdot \re(X') & \m Y \cdot \im(X') \end{pmatrix},
\end{multline*}
\begin{multline*}
k_{\m}(X,Y) \cdot Dx_{\m} \bigr|_{S^3_{\epsilon}(Y)} \sim \\
\epsilon^{-5} \cdot
\biggl(1- \frac{\m^2 Y^2}{\epsilon^2b^2} \bigl( \re(X') \bigr)^2 \biggr)^{-2}
\cdot \biggl( \epsilon^2 b + \frac{\m^2 Y^2}b \Bigl( \bigl(\im(X')\bigr)^2
- \bigl(\re(X')\bigr)^2 \Bigr) \biggr) dS \\
+ \epsilon^{-5} \cdot
\biggl(1- \frac{\m^2 Y^2}{\epsilon^2b^2} \bigl( \re(X') \bigr)^2 \biggr)^{-2}
\cdot \frac{\m^2 Y^2}{b^2} \re(X') \cdot \im(X')
\begin{pmatrix} b & -\m Y \\ \m Y & -b \end{pmatrix} dS.
\end{multline*}
The integral over $S^3_{\epsilon}(Y)$ of the last term is zero,
and the first term simplifies to
$$
\epsilon^{-3} \cdot b^{-1}
\biggl(1- \frac{\m^2 Y^2}{\epsilon^2b^2} \bigl( \re(X') \bigr)^2 \biggr)^{-2} dS.
$$
We finish the proof by integrating in spherical coordinates and
using an integral
$$
\int_{\theta=-\pi/2}^{\theta=\pi/2} \frac{\sin^2\theta\,d\theta}{(1-a\cos^2\theta)^2}
= \frac{\pi}{2\sqrt{1-a}}, \qquad |a|<1,
$$
with $a=\m^2 Y^2 b^{-2}$.
\end{proof}

\section{Deformation of $\Zh$ and the Second Order Pole}  \label{Zh_mu-section}

Similarly to how we did in Section \ref{H_mu-section}, we introduce
a space of functions $\Zh_{\m}$, which is a deformation of $\Zh$.
Then we discuss the analogues of the second order pole formulas
given in Corollary \ref{proj-cor} and Theorem \ref{Zh^0-projector}.
Thus we introduce a vector space
$$
\Zh_{\m} = \text{$\BB C$-span of }
\frac{t^l_{m\,\underline{n}}(X) \cdot \bigl( \bigl(1+\m^2N(X)\bigr)^{1/2}-1 \bigr)^k}
{\bigl( \bigl(1+\m^2N(X)\bigr)^{1/2}+1 \bigr)^{2l+k+2}},
\qquad \begin{matrix} k \in \BB Z,\quad l = 0, \frac12, 1, \frac32, \dots, \\
m,n = -l, -l+1, \dots, l. \end{matrix}
$$
Note that when $\m \to 0$,
$$
2^{2l+2k+2}\m^{-2k}
\frac{t^l_{m\,\underline{n}}(X) \cdot \bigl( \bigl(1+\m^2N(X)\bigr)^{1/2}-1 \bigr)^k}
{\bigl( \bigl(1+\m^2N(X)\bigr)^{1/2}+1 \bigr)^{2l+k+2}}
\quad \to \quad t^l_{m\,\underline{n}}(X) \cdot N(X)^k.
$$
We can extend these functions to an open neighborhood of $\BB H^{\times}$ in
$\HC^{\times}$ as follows.
(We exclude $Z \in \HC$ such that $N(Z)=0$ because
$\bigl(1+\m^2N(X)\bigr)^{1/2}-1$ vanishes there.)
The matrix coefficient functions $t^l_{m\,\underline{n}}(X)$'s are
polynomials in $X$, hence extend to $\HC$ without any problem.
The only obstacle to extending the functions spanning $\Zh_{\m}$
is the square root in $\bigl(1+\m^2N(X)\bigr)^{1/2} \pm 1$.
Thus we choose the branch of $z^{1/2}$ defined on the complex plane without the
negative real axis and observe that the functions
$$
f_{k,l,m,n}(Z) =
\frac{t^l_{m\,\underline{n}}(Z) \cdot \bigl( \bigl(1+\m^2N(Z)\bigr)^{1/2}-1 \bigr)^k}
{\bigl( \bigl(1+\m^2N(Z)\bigr)^{1/2}+1 \bigr)^{2l+k+2}}, \qquad Z \in \HC,
$$
are well defined as long as $N(Z) \notin (-\infty,- \m^{-2}]$ and $N(Z) \ne 0$.
For this reason we introduce an open region in $\HC$
$$
\BB U_{\m} =\{ Z \in \HC^{\times} ;\: N(Z) \notin (-\infty,- \m^{-2}] \}.
$$

Our next task is to define a natural bilinear pairing on $\Zh_{\m}$.
Fix an $R>0$ and parameterize $U(2)_R$ as in Chapter III, \S 1, of \cite{V}:
\begin{multline*}
Z(\alpha, \phi, \theta, \psi) = R e^{i\alpha}
\begin{pmatrix} e^{i\frac{\phi}2} & 0 \\ 0 & e^{-i\frac{\phi}2} \end{pmatrix}
\begin{pmatrix} \cos\frac{\theta}2 & i\sin\frac{\theta}2 \\ 
i\sin\frac{\theta}2 & \cos\frac{\theta}2 \end{pmatrix}
\begin{pmatrix} e^{i\frac{\psi}2} & 0 \\ 0 & e^{-i\frac{\psi}2} \end{pmatrix} \\
= \begin{pmatrix} \cos\frac{\theta}2 \cdot e^{i\frac{\phi+\psi}2} &
i\sin\frac{\theta}2 \cdot e^{i\frac{\phi-\psi}2} \\ 
i\sin\frac{\theta}2 \cdot e^{i\frac{\psi-\phi}2} &
\cos\frac{\theta}2 \cdot e^{-i\frac{\phi+\psi}2} \end{pmatrix}, \qquad
\begin{matrix} 0 \le \alpha < \pi, \\ 0 \le \phi < 2\pi, \\
0 < \theta < \pi, \\ -2\pi \le \psi < 2\pi.\end{matrix}
\end{multline*}
By direct computation we find
\begin{multline*}
dV \Bigl|_{U(2)_R} = dz^0 \wedge dz^1 \wedge dz^2 \wedge dz^3 \Bigl|_{U(2)_R} \\
= \frac14 dz_{11} \wedge dz_{12} \wedge dz_{21} \wedge dz_{22} \Bigl|_{U(2)_R} 
= \frac{R^4}{8i} e^{4i\alpha} \sin\theta
\, d\alpha \wedge d\phi \wedge d\theta \wedge d\psi.
\end{multline*}
For $0 < R < \m^{-1}$, define a measure on $U(2)_R$ by
$$
dV_{R,\m} = \frac{R^4}{16} e^{4i\alpha} \sin\theta \,
d\phi \wedge d\theta \wedge d\psi \wedge
d \log \biggl( \frac{(1+\m^2 R^2 e^{2i\alpha})^{1/2}-1}
{(1+\m^2 R^2 e^{2i\alpha})^{1/2}+1} \biggr)
$$
and define a bilinear pairing on $\Zh_{\m}$ as
\begin{equation}  \label{Zh_mu-pairing}
\langle f_1, f_2 \rangle_{\m} = 
\frac i{2\pi^3} \int_{Z \in U(2)_R} f_1(Z) \cdot f_2(Z) \,dV_{R,\m},
\qquad f_1,f_2 \in \Zh_{\m}.
\end{equation}
(The parameter $R$ is restricted to $0 < R < \m^{-1}$ so that
$U(2)_R \subset \BB U_{\m}$.)
We have the following analogue of the orthogonality relations
(\ref{orthogonality}):

\begin{prop}
The symmetric pairing (\ref{Zh_mu-pairing}) is independent of the choice of $R$
(as long as $0 < R < \m^{-1}$) and non-degenerate. Let
$$
f'_{k,l,m,n}(Z) =
\frac{t^l_{n\,\underline{m}}(Z^+) \cdot \bigl(\bigl(1+\m^2N(Z)\bigr)^{1/2}+1\bigr)^k}
{\bigl( \bigl(1+\m^2N(Z)\bigr)^{1/2}-1 \bigr)^{2l+k+2}} \quad \in \Zh_{\m},
$$
then we have orthogonality relations
\begin{equation}  \label{mu-orthogonality}
\bigl\langle f_{k,l,m,n}(Z), f'_{k',l',m',n'}(Z) \bigr\rangle_{\m}
= \frac{\m^{-4l-4}}{2l+1} \delta_{kk'}\delta_{ll'} \delta_{mm'} \delta_{nn'},
\end{equation}
where the indices $k,l,m,n$ are $k \in \BB Z$,
$l = 0, \frac12, 1, \frac32, \dots$, $m,n \in \BB Z +l$, $-l \le m,n \le l$
and similarly for $k',l',m',n'$.
\end{prop}

\begin{proof}
Since each family of functions $f_{k,l,m,n}(Z)$'s and $f'_{k,l,m,n}(Z)$'s
generates $\Zh_{\m}$, the independence of $R$ and non-degeneracy of the
bilinear pairing follow from the orthogonality relations
(\ref{mu-orthogonality}).
Using the orthogonality relations (\ref{t-orthog}), we obtain
\begin{multline*}
-2\pi^3 i \cdot \bigl\langle f_{k,l,m,n}(Z), f'_{k',l',m',n'}(Z) \bigr\rangle_{\m} \\
= \int_{Z \in U(2)_R}
\frac{t^l_{m\,\underline{n}}(Z) \cdot \bigl( \bigl(1+\m^2N(Z)\bigr)^{1/2}-1 \bigr)^k}
{\bigl( \bigl(1+\m^2N(Z)\bigr)^{1/2}+1 \bigr)^{2l+k+2}} \cdot
\frac{t^{l'}_{n'\,\underline{m'}}(Z^+) \cdot
\bigl( \bigl(1+\m^2N(Z)\bigr)^{1/2}+1 \bigr)^{k'}}
{\bigl( \bigl(1+\m^2N(Z)\bigr)^{1/2}-1 \bigr)^{2l'+k'+2}} \,dV_{R,\m} \\
= \int
\frac{\m^{-4} \cdot t^l_{m\,\underline{n}}(Z) \cdot t^{l'}_{n'\,\underline{m'}}(Z^+)
\cdot N(Z)^{-2}}{\bigl( \bigl(1+\m^2N(Z)\bigr)^{1/2}+1 \bigr)^{2l} \cdot 
\bigl( \bigl(1+\m^2N(Z)\bigr)^{1/2}-1 \bigr)^{2l'}}
\biggl( \frac{\bigl(1+\m^2N(Z)\bigr)^{1/2}-1}
{\bigl(1+\m^2N(Z)\bigr)^{1/2}+1} \biggr)^{k-k'} dV_{R,\m} \\
= -\pi^2 \frac{\m^{-4l-4}}{2l+1} \delta_{ll'} \delta_{mm'} \delta_{nn'}
\int_{\alpha=0}^{\alpha=\pi} 
\biggl( \frac{(1+\m^2 R^2 e^{2i\alpha})^{1/2}-1}
{(1+\m^2 R^2 e^{2i\alpha })^{1/2}+1} \biggr)^{k-k'}
d\log \biggl( \frac{(1+\m^2 R^2 e^{2i\alpha})^{1/2}-1}
{(1+\m^2 R^2 e^{2i\alpha})^{1/2}+1} \biggr)  \\
= -\pi^2 \frac{\m^{-4l-4}}{2l+1} \delta_{ll'} \delta_{mm'} \delta_{nn'}
\oint z^{k-k'-1} \,dz
= -2\pi^3 i \frac{\m^{-4l-4}}{2l+1} \delta_{kk'}\delta_{ll'} \delta_{mm'} \delta_{nn'}.
\end{multline*}
\end{proof}




Let us recall Proposition 27 from \cite{FL1} restated here as
Proposition \ref{prop27}. We want to obtain a similar expansion for
$\bigl(\langle \hat X - \hat Y,\hat X - \hat Y\rangle_{1,4}\bigr)^{-2}$.
We proceed as in the proof of Proposition \ref{1/N-exp-deformed}.
Let $X, Y \in \BB U_{\m}$, let
$$
t_1=\sqrt{1+\m^2N(X)}, \qquad t_2=\sqrt{1+\m^2N(Y)} \quad\in \BB C
$$
and choose $\theta_1, \theta_2 \in \BB C$ so that
$$
\cosh\theta_1 = t_1 = \sqrt{1+\m^2N(X)} \quad \text{and} \quad
\cosh\theta_2 = t_2 = \sqrt{1+\m^2N(Y)}.
$$
(The square roots $\sqrt{1+\m^2N(X)}$ and $\sqrt{1+\m^2N(Y)}$ are uniquely
defined, but $\theta_1$ and $\theta_2$ are not.)
Then $\sinh^2\theta_1 = \m^2N(X)$ and $\sinh^2\theta_2 = \m^2N(Y)$.
Define
$$
\overrightarrow{u} = \frac{\m X}{\sinh\theta_1}, \quad
\overrightarrow{v} = \frac{\m Y}{\sinh\theta_2} \quad \in \HC,
$$
and suppose that $\overrightarrow{u} \overrightarrow{v}^{-1}$ is similar to a
diagonal matrix $\bigl(\begin{smallmatrix} \lambda & 0 \\
0 & \lambda^{-1} \end{smallmatrix}\bigr)$,
where $\lambda \in \BB C$.
Using the multiplicativity property of matrix coefficients (\ref{t-mult})
and our previous notations (\ref{ab}), we compute a sum over all
$m,n = -l, -l+1, \dots, l$, then over $k=0,1,2,3,\dots$ and
$l = 0, \frac12, 1, \frac32, \dots$:
\begin{multline}  \label{1/N^2}
\sum_{k,l,m,n} (2l+1) \m^{4l+4} f_{k,l,m,n}(X) \cdot f'_{k,l,m,n}(Y) \\
= \m^4 \sum_{k,l} (2l+1) \frac{(t_1-1)^{l+k}}{(t_1+1)^{l+k+2}} \cdot
\frac{(t_2+1)^{l+k}}{(t_2-1)^{l+k+2}}
\cdot \chi_l(\overrightarrow{u} \overrightarrow{v}^{-1}) \\
= \frac{16\m^4 (\lambda-\lambda^{-1})^{-1}}{(e^{\theta_1}+e^{-\theta_1}+2)^2
(e^{\theta_2}+e^{-\theta_2}-2)^2} \sum_{k,l} (2l+1) \frac{b^{2l+2k}}{a^{2l+2k}} \cdot
(\lambda^{2l+1}-\lambda^{-2l-1}) \\
= \frac{16\m^4 (\lambda-\lambda^{-1})^{-1}}{(e^{\theta_1/2}+e^{-\theta_1/2})^4
(e^{\theta_2/2}-e^{-\theta_2/2})^4 (1-b^2/a^2)} \sum_l (2l+1) \frac{b^{2l}}{a^{2l}}
\cdot (\lambda^{2l+1}-\lambda^{-2l-1}) \\
= \frac{16\m^4}{(a-\lambda b)^2(a-\lambda^{-1}b)^2}
= \frac{16\m^4}{N^2(b\overrightarrow{u} - a\overrightarrow{v})}
= \frac1{\bigl( \langle \hat X - \hat Y,\hat X - \hat Y \rangle_{1,4}\bigr)^2},
\end{multline}
where in the last step we used (\ref{1/N}) and
$$
\hat X = \bigl( \sqrt{\m^{-2}+N(X)},x^0,x^1,x^2,x^3 \bigr) \quad \text{and} \quad
\hat Y = \bigl( \sqrt{\m^{-2}+N(Y)},y^0,y^1,y^2,y^3 \bigr)
\quad \in \BB R^{1,4}_+.
$$
Like the expansion of
$\bigl(\langle \hat X - \hat Y,\hat X - \hat Y\rangle_{1,4}\bigr)^{-1}$,
this expansion holds whenever
$$
|\lambda b/a|<1 \quad \text{and} \quad |\lambda^{-1} b/a|<1
$$
and, in particular, for those $X$, $Y$ the denominator does not turn to zero.

We avoid finding the region where these inequalities are satisfied and
impose instead an assumption that $|\lambda|=1$. We have:
$$
\frac{b^2}{a^2} = \frac{\tanh^2(\theta_1/2)}{\tanh^2(\theta_2/2)}
= \frac{t_1-1}{t_1+1} \frac{t_2+1}{t_2-1}
= \frac{\sqrt{1+\m^2N(X)}-1}{\sqrt{1+\m^2N(X)}+1}
\frac{\sqrt{1+\m^2N(Y)}+1}{\sqrt{1+\m^2N(Y)}-1}.
$$
Thus the expansion (\ref{1/N^2}) certainly holds for $X, Y \in \BB U_{\m}$
such that $XY^{-1}$ is diagonalizable with both eigenvalues having the same
length and
$$
\biggl| \frac{\sqrt{1+\m^2N(X)}-1}{\sqrt{1+\m^2N(X)}+1} \biggr|
< \biggl| \frac{\sqrt{1+\m^2N(Y)}-1}{\sqrt{1+\m^2N(Y)}+1} \biggr|.
$$
The condition that $XY^{-1}$ is diagonalizable with both eigenvalues having
the same length is automatically satisfied if $X \in U(2)_{R_1}$ and
$Y \in U(2)_{R_2}$.

Using single variable calculus, we can prove:

\begin{lem}
Let $0 < R < \m^{-1}$. Then, as $X$ ranges over $U(2)_R$,
$$
\frac{\sqrt{1+\m^2R^2}-1}{\sqrt{1+\m^2R^2}+1} \le
\biggl| \frac{\sqrt{1+\m^2N(X)}-1}{\sqrt{1+\m^2N(X)}+1} \biggr|
\le \frac{1-\sqrt{1-\m^2R^2}}{\sqrt{1-\m^2R^2}+1}.
$$
\end{lem}

Define subspaces of $\Zh_{\m}$ similar to $\Zh^+$, $\Zh^-$ and $\Zh^0$:
\begin{align*}
\Zh^+_{\m} &= \BB C-\text{span of }
\bigl\{ f_{k,l,m,n}(Z);\: k \ge 0 \bigr\}, \\
\Zh^-_{\m} &= \BB C-\text{span of }
\bigl\{ f_{k,l,m,n}(Z);\: k \le -(2l+2) \bigr\}, \\
\Zh^0_{\m} &= \BB C-\text{span of }
\bigl\{ f_{k,l,m,n}(Z);\: -(2l+1) \le k \le -1 \bigr\}
\end{align*}
(and the ranges of indices $l$, $m$, $n$ are
$l = 0, \frac12, 1, \frac32, \dots$, $m,n = -l , -l+1, \dots, l$, as before).
Thus $\Zh_{\m} = \Zh^-_{\m} \oplus \Zh^0_{\m} \oplus \Zh^+_{\m}$.

\begin{rem}
The spaces $\Zh_{\m}$, $\Zh^-_{\m}$, $\Zh^0_{\m}$ and $\Zh^+_{\m}$ are defined by
analogy with spaces ${\cal H}_{\m}$ and $\Zh = \Zh^- \oplus \Zh^0 \oplus \Zh^+$.
We believe that these spaces can be characterized as images under the
multiplication maps on ${\cal H}^{\pm}_{\m} \otimes {\cal H}^{\pm}_{\m}$.
Thus $\Zh_{\m}$, $\Zh^-_{\m}$, $\Zh^0_{\m}$ and $\Zh^+_{\m}$ should be the images
of ${\cal H}_{\m} \otimes {\cal H}_{\m}$,
${\cal H}^-_{\m} \otimes {\cal H}^-_{\m}$,
${\cal H}^-_{\m} \otimes {\cal H}^+_{\m}$ and
${\cal H}^+_{\m} \otimes {\cal H}^+_{\m}$ respectively.
(Compare with Lemma \ref{image-lemma}.)
\end{rem}

Note that
$$
\Zh^-_{\m} = \BB C-\text{span of }
\bigl\{ f'_{k,l,m,n}(Z);\: k \ge 0 \bigr\}.
$$
From our expansion of
$\bigl(\langle \hat X - \hat Y,\hat X - \hat Y\rangle_{1,4}\bigr)^{-2}$
we immediately obtain the following analogue of Corollary \ref{proj-cor}:

\begin{prop}
Let $0 < R < \m^{-1}$ and $r>0$.
\begin{enumerate}
\item
If $Y \in U(2)_r$ and
$$
\biggl| \frac{\sqrt{1+\m^2N(Y)}-1}{\sqrt{1+\m^2N(Y)}+1} \biggr|
< \frac{\sqrt{1+\m^2R^2}-1}{\sqrt{1+\m^2R^2}+1},
$$
then $r<R$ and the map
$$
f \mapsto (\P_{\m}^+ f)(Y) = 
\frac i{2\pi^3} \int_{X \in U(2)_R} \frac{f(X) \,dV_{R,\m}}
{\bigl(\langle \hat X - \hat Y,\hat X - \hat Y\rangle_{1,4}\bigr)^2},
\qquad f \in \Zh_{\m},
$$
is a projector onto $\Zh^+_{\m}$ annihilating $\Zh^-_{\m} \oplus \Zh^0_{\m}$ and,
in particular, provides a reproducing formula for functions in $\Zh^+_{\m}$;
\item
If $Y \in U(2)_r \cap \BB U_{\m}$ and
$$
\frac{1-\sqrt{1-\m^2R^2}}{\sqrt{1-\m^2R^2}+1} <
\biggl| \frac{\sqrt{1+\m^2N(Y)}-1}{\sqrt{1+\m^2N(Y)}+1} \biggr|,
$$
then $r>R$ and the map
$$
f \mapsto (\P_{\m}^- f)(Y) = 
\frac i{2\pi^3} \int_{X \in U(2)_R} \frac{f(X) \,dV_{R,\m}}
{\bigl(\langle \hat X - \hat Y,\hat X - \hat Y\rangle_{1,4}\bigr)^2},
\qquad f \in \Zh_{\m},
$$
is a projector onto $\Zh^-_{\m}$ annihilating $\Zh^0_{\m} \oplus \Zh^+_{\m}$ and,
in particular, provides a reproducing formula for functions in $\Zh^-_{\m}$.
\end{enumerate}
\end{prop}

The reproducing kernel and projector for the space $\Zh^0_{\m}$ can be
obtained formally as in Section \ref{formal-section} and with full rigor
as in Section \ref{embeddings-section}.
The advantage of the anti de Sitter deformation of $\Zh^0$ (and also $\Zh^{\pm}$)
is that now we can extend the functions from this representation to the
ambient five-dimensional Minkowski space $\BB R^{1,4}$, and we expect some
additional flexibility in the permissible choices of integration cycles for the
quaternionic analogues of Cauchy's formula for the second order pole
(cf. Theorem \ref{Zh^0-projector} in this paper for the scalar case and
Theorem 77 in \cite{FL1} for the spinor case).



\separate

\separate

\noindent
{\em Department of Mathematics, Yale University,
P.O. Box 208283, New Haven, CT 06520-8283}\\
{\em Department of Mathematics, Indiana University,
Rawles Hall, 831 East 3rd St, Bloomington, IN 47405}   

\end{document}